\documentclass[reqno,12pt,letterpaper]{amsart}

\usepackage{amsmath,amssymb,amsthm,graphicx,mathrsfs,url,bm,subfigure,accents,paralist,xargs,slashed,enumitem,mathtools,array}
\usepackage[usenames,dvipsnames]{xcolor}
\usepackage[colorlinks=true,linkcolor=Red,citecolor=Green]{hyperref}

\usepackage{tikz}
\usepackage{pgfplots}
\usetikzlibrary{positioning}
\usetikzlibrary{decorations.markings}
\usetikzlibrary{calc}
\usetikzlibrary{patterns}
\usetikzlibrary{intersections}
\usetikzlibrary{backgrounds}
\usetikzlibrary{arrows}
\def\?[#1]{\textbf{[#1]}\marginpar{\Large{\textbf{??}}}}

\def\smath#1{\text{\scalebox{1}{$#1$}}}
\def\sfrac#1#2{\smath{\frac{#1}{#2}}}

\newlist{inlineroman}{enumerate*}{1}
\setlist[inlineroman]{itemjoin*={{, and }},afterlabel=~,label=\roman*.}

\newcommand{\inlineitem}[1][]{%
	\ifnum\enit@type=\tw@
	{\descriptionlabel{#1}}
	\hspace{\labelsep}
	\else
	\ifnum\enit@type=\z@
	\refstepcounter{\@listctr}\fi
	\quad\@itemlabel\hspace{\labelsep}
	\fi}

\DeclareFontEncoding{FMS}{}{}
\DeclareFontSubstitution{FMS}{futm}{m}{n}
\DeclareFontEncoding{FMX}{}{}
\DeclareFontSubstitution{FMX}{futm}{m}{n}
\DeclareSymbolFont{fouriersymbols}{FMS}{futm}{m}{n}
\DeclareSymbolFont{fourierlargesymbols}{FMX}{futm}{m}{n}
\DeclareMathDelimiter{\VERT}{\mathord}{fouriersymbols}{152}{fourierlargesymbols}{147}

\newcommand{\RN}[1]{%
	\textup{\uppercase\expandafter{\romannumeral#1}}%
}

\theoremstyle{plain}
\newtheorem{theo}{Theorem}
\newtheorem{prop}{Proposition}[section]

\newtheorem{lem}[prop]{Lemma}

\theoremstyle{remark}

\theoremstyle{definition}
\newtheorem{rem}[prop]{Remark}
\newtheorem{defi}[prop]{Definition}

\numberwithin{equation}{section}

\DeclareMathOperator{\codim}{codim}

\DeclareMathOperator{\esssupp}{esssupp}

\newcommand{\Op}{\mathrm{Op}}

\newcommand{\Opbh}{\mathrm{Op}_{\mathrm b,h}}

\newcommand{\sgn}{\mathrm{sgn}}

\DeclareMathOperator{\supp}{supp}

\newcommand{\WF}{\mathrm{WF}_{h}}
\newcommand{\WFb}{\mathrm{WF}_{\mathrm{b},h}}

\newcommand{\ELL}{\mathrm{ell}_h}
\newcommand{\ELLb}{\mathrm{ell}_{\mathrm{b}}}

\newcommand{\RR}{\mathbb{R}}
\newcommand{\CC}{\mathbb{C}}
\newcommand{\NN}{\mathbb{N}}

\newcommand{\OL}[1]
{
	\overline{#1}
}
\DeclareMathOperator{\Diff}{Diff}

\newcommand{\Lap}{\Delta}
\newcommand{\pa}{{\partial}}
\newcommand{\hamvf}{\mathsf{H}}
\newcommand{\bhamvf}{\mathsf{H}^{\mathrm b}}
\newcommand{\hyp}{\mathcal{H}}
\newcommand{\gl}{\mathcal{G}}
\newcommand{\ellip}{\mathcal{E}}

\newcommand{\loc}{\mathrm{loc}}

\newcommand{\Psibh}{\Psi_{\mathrm{b},h}}

\newcommand{\chare}{\Sigma}
\newcommand{\cchare}{\dot{\Sigma}}
\newcommand{\bT}{{}^\mathrm{b} T}

\newcommand{\bdotT}{{}^{\mathrm{b}} \dot T}

\newcommand{\Psibhc}{\Psi_{\mathrm{bc},h}}
\newcommand{\CI}{\mathcal{C}^\infty}
\newcommand{\CcI}{\mathcal{C}_c^\infty}
\newcommand{\CmI}{\mathcal{C}^{-\infty}}
\newcommand{\CdI}{\dot{\mathcal{C}}^{\infty}}

\newcommand{\CdmI}{\dot{\mathcal{C}}^{-\infty}}

\newcommand{\C}{\mathcal{C}}
\newcommand{\Diffh}{\mathrm{Diff}_{h}}
\newcommand{\symbbc}{S_\mathrm{bc}}
\newcommand{\symbb}{S_\mathrm{b}}
\newcommand{\symbbh}{S_{\mathrm{b},h}}

\newcommand{\symbbch}{S_{\mathrm{bc},h}}
\newcommand{\B}{\mathrm{b}}
\newcommand{\indicial}{\widehat{N}}

\newcommand{\Vb}{\mathcal{V}_{\mathrm{b}}}
\newcommand{\bsymbol}{\sigma_{\mathrm{b},h}}
\newcommand{\GBB}{\mathrm{GBB}}

\newcommand{\diag}{\mathrm{diag}}
\newcommand{\COMP}{\mathrm{comp}}

\newcommand{\hmg}{\sphericalangle}

\renewcommand{\Im}{\operatorname{Im}}

\newcommand{\blue}[1]{#1}

\title[Semiclassical diffraction by conormal potential singularities]
{Semiclassical diffraction by conormal potential singularities}

\author{Oran Gannot \and Jared Wunsch}
\email{ogannot@berkeley.edu \and jwunsch@math.northwestern.edu}
\address{Department of Mathematics, Lunt Hall, Northwestern University,
	Evanston, IL 60208, USA}

\begin{document}

\begin{abstract}

  We establish propagation of singularities for the semiclassical
  Schr\"odinger equation, where the potential is conormal to a
  hypersurface. We show that semiclassical wavefront set propagates
  along generalized broken bicharacteristics, hence reflection of
  singularities may occur along trajectories reaching the hypersurface
  transversely.  The reflected wavefront set is weaker, however, by
  a power of $h$ that depends on the regularity of the potential.  We
  also show that for sufficiently regular potentials, wavefront set
  may not stick to the hypersurface, but rather detaches from it at
  points of tangency to  travel along ordinary bicharacteristics.

\end{abstract}

	\maketitle

\thispagestyle{empty}

\section{Introduction}

\subsection{Statement of results} \label{subsect:statementofresults} Let $(X,g)$ be a smooth $n$-dimensional Riemannian manifold, and $Y \subset X$ a hypersurface. We study propagation of semiclassical singularities for the Schr\"odinger operator
\begin{equation} \label{eq:P}
P = -h^2 \Delta_g + V,
\end{equation}
where the \blue{real-valued} potential $V$ is conormal  to $Y$.  \blue{Semiclassical
  propagation of singularities theorems constrain the
  distribution of energy in phase space of a solution to \eqref{eq:P},
  asymptotically as $h \to 0$:
  for $V$ smooth, it is known that the energy concentrates on the
  classical energy surface and is invariant under the associated classical
  dynamics.  Here, by contrast, the singularities of the potential $V$
  play an important role, diffracting energy along \emph{broken} classical trajectories.}

\blue{The class of potentials $V$ that we consider are real-valued \emph{conormal
    distributions} with respect to $Y,$ a class of
  distributions that are smooth functions except at $Y.$  If $x$ is a defining function of $Y$ then
  $x_+^\alpha$ is an instructive example, with $\alpha> 0.$}
More generally,
we assume throughout that $V \in I^{[-1-\alpha]}(Y)$ for some $\alpha
> 0$.  This means that $V$ is locally the inverse Fourier
transform of a Kohn--Nirenberg symbol of order $-1-\alpha$, transverse to $Y$. In particular, $V$ is $1+\alpha$ orders
more regular than the delta distribution along $Y.$  If
$\alpha \geq k + \gamma$ with $k \in \NN$ and $\gamma \in (0,1)$, then
$V \in \C^{k,\gamma}(X)$, but $V$ is $\mathcal{C}^\infty$ away
from $Y$. (See Section~\ref{subsect:conormaldistributions} below for details.)

\blue{Let $p=\lvert\xi\rvert^2_g+V$ denote the semiclassical principal
symbol of $P$.   Let $\hamvf_p$ denote its associated Hamilton vector
field, e.g., $\hamvf_p=2 \xi \cdot \partial_x-(\pa_x V) \cdot \pa_\xi$
if $g$ is the Euclidean metric.  Recall that $\WF^s(u),$ the
semiclassical wavefront set of order $s,$ measures where, in $T^*X,$
the family $u$ fails to be $\mathcal{O}_{L^2}(h^s).$}
If $P u=0$,
\blue{then known results imply that} the semiclassical wavefront set $\WF^s(u)$ of order $s$ is
contained in the characteristic set \blue{$\Sigma\equiv \{p=0\}$}, and is invariant under
the $\hamvf_{p}$ flow for each $s \in \RR \cup \{+\infty\}$, at least away from $Y$. This
result breaks down for singularities striking $T^*_Y X$: the conormal
singularity of $V$ causes ray splitting, generating wavefront set along both
the reflected and transmitted components.

To make the  notion of ray-splitting precise, we introduce a suitable \emph{generalized
  broken bicharacteristic} $(\GBB)$ flow, taking into account both
transverse and tangential incidence to $Y$. Properties of this $\GBB$
flow are described in detail in Section \ref{subsect:GBB}; \blue{its
  main feature is that the allowed trajectories are continuous in
  space but potentially discontinuous in momentum, with momentum
  tangent to $Y$ conserved at interactions with this hypersurface, in
  accordance with the laws of reflection and refraction.  The GBB flow
  is, consequently, not defined on the usual cotangent bundle, where
  it would be discontinuous.  Instead, we introduce an adapted notion
  of semiclassical wavefront set by using a variant of Melrose's
  \emph{b-calculus} of pseudodifferential operators.  This gives rise
  to a \emph{semiclassical b-wavefront set} which lives in a rescaling
  of the usual cotangent bundle, and agrees with the usual
  semiclassical wavefront set away from $Y,$ but has the combined virtue and
  defect of not distinguishing different normal momenta over $Y$
  itself.  The compressed characteristic set employed below is
  likewise an appropriately rescaled version of the set $\{p=0\},$
  which does not distinguish among different normal momenta over $Y.$}
(For details, including the relevant notation, see Section
\ref{sect:bpseudo}.)

\begin{theo}[{Propagation of singularities}] \label{theo:GBBpropagation} 
	Let $\alpha > 0$ and $s \in \RR \cup \{+\infty\}$. If $u$ is $h$-tempered in $H^1_{h,\loc}(X)$, then $\WFb^{s}(u) \setminus \WFb^{-1,s+1}(Pu)$ is the union of maximally extended $\GBB$s within the compressed characteristic set $\cchare$.
\end{theo}

\blue{Suppose $Pu=0.$  Then Theorem~\ref{theo:GBBpropagation} tells us that a given point in
  the wavefront set must give rise to wavefront set along \emph{at
    least} one maximally extended GBB through it, but does not
  distinguish among the various possibilities.  The theorems that
  follow draw subtler distinctions among them, and in particular give
  a special role to GBBs that are in fact ordinary solutions to 
  Hamilton's equations of motion.
Thus we now return to the usual
  cotangent bundle, where we may consider the usual Hamilton flow
  provided that there is enough regularity for it to make sense.}
Introduce local coordinates $(x,y)$ such that $Y = \{x=0\}$, and let $(x,y,\xi,\eta)$ be the corresponding canonical coordinates on $T^*X$. Even though Hamilton's equations become singular over $Y$ when $\alpha \leq 1$, the integral curves of $\hamvf_{p}$ are well defined near transversally incident points 
\[
\varpi_\pm = (0,y_0,\pm \xi_0,\eta_0) \in \Sigma
\] 
where the normal momentum $\pm \xi_0$ does not vanish; see Lemma \ref{lem:caratheodory}. The integral curves $\gamma_\pm$ with $\gamma_\pm(0) = \varpi_\pm$ therefore exist on some interval $(-\varepsilon,\varepsilon)$. To use the terminology of \cite{de2014diffraction}, the points $\varpi_\pm$ are said to be related, in the sense of having the same tangential momentum. Since $\WFb^s(u) = \WF^s(u)$
 away from $Y$, Theorem \ref{theo:GBBpropagation} states the following at transversally incident points: if $\gamma_+((-\varepsilon,0))$ and $\gamma_-((-\varepsilon,0))$ are both disjoint from $\WF^s(u)$, then
\begin{equation} \label{eq:disjointWF}
\gamma_+((0,\varepsilon)) \cap \WF^s(u) = \emptyset.
\end{equation}
On the other hand, the reflected singularity (namely the contribution of incident wavefront set along $\gamma_-((-\varepsilon,0))$ to outgoing wavefront set along $\gamma_+((\varepsilon,0))$) 
is expected to be weaker than the original incident singularity along $\gamma_-((-\varepsilon,0))$. In other words, if $\gamma_+((-\varepsilon,0))$ is disjoint from $\WF^s(u)$ and $\gamma_-((-\varepsilon,0))$ is disjoint from $\WF^r(u)$, then \eqref{eq:disjointWF} should hold for a range of $s$ depending on $\alpha$ and $r$. We show that at least when $\alpha>1$, this holds for $s \leq r+\alpha$.

\begin{theo}[Diffractive improvement at transverse reflection] \label{theo:propagation} Let
  $\alpha > 1$ and $s \leq r+ \alpha$, where
  $s,r \in \RR \cup \{+\infty\}$. Suppose that $u$ is $h$-tempered in
  $H^1_{h,\loc}(X)$ with $Pu \in L^2_{\loc}(X)$, and $\WF^{s+1}(Pu) = \emptyset$. Let
	\[
	\varpi_\pm = (0,y_0,\pm \xi_0,\eta_0) \in \Sigma
	\]
	with $\xi_0 \neq 0$, and let $\gamma_\pm$ be as above. If $\varpi_+ \in \WF^s(u)$, then there exists $\varepsilon > 0$ such that 
	\[
	\gamma_+((-\varepsilon,0)) \subset \WF^s(u) \text{ or } \gamma_-((-\varepsilon,0)) \subset \WF^r(u).
	\]
      \end{theo}

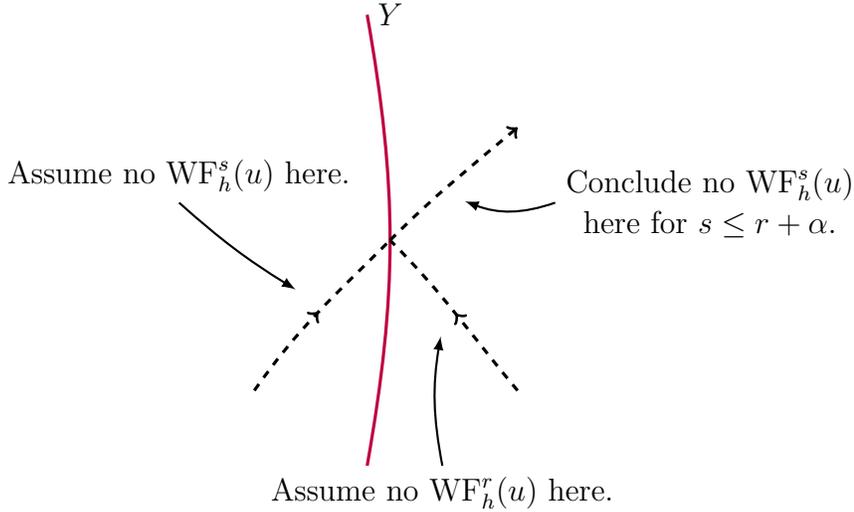
\begin{figure}
\begin{center}
	
	\begin{tikzpicture}

          \tikzset{->-/.style={decoration={
  markings,
  mark=at position .5 with {\arrow{>}}},postaction={decorate}}}

	\coordinate  (A) at (4,0);
	\coordinate [label=right:$Y$] (B) at (4,6);
	\draw[very thick, color = purple] (A) to [bend right = 10] (B);
	\path (A) to [bend right = 10] coordinate [pos=0.5] (q)  (B) ;

		\draw[->,very thick,dashed] (q) to [out=45,in=-140] (6,4.5);
		\path (q) to [out=45,in=-150] coordinate [pos=0.5] (m1)  (6,4.5);
		\draw[thick,->,>=latex] (6.5,3.5) to [bend left = 20] ([xshift = 2mm,yshift = -3mm] m1);
		\node[right, yshift=0mm, align=center] at (6.5,3.5)(bichar){Conclude
                  no $\WF^s(u)$ \\ here for {$s\leq r+\alpha$}.};

		\draw[->-,very thick,dashed] (6,1) to [out=130,in=-45] (q);
		\path (6,1) to [out=130,in=-45] coordinate [pos=0.5] (m2)  (q);
		\draw[thick,->,>=latex] (5,0) to [bend left = 10] ([xshift = -2mm, yshift = -3mm] m2);
		\node[below, yshift=0mm] at (5,0)(bichar){Assume no $\WF^r(u)$ here.};
		
		\draw[->-,very thick,dashed] (2.5,1) to [out=55,in=-135] (q);
		\path (2.5,1) to [out=55,in=-135] coordinate [pos=0.5] (m3)  (q);
		\draw[thick,->,>=latex] (1.5,3.5) to [bend right = 5] ([xshift = -3mm, yshift = 3mm] m3);
		\node[above, yshift=0mm] at (1.5,3.5)(bichar){Assume no
                  $\WF^s(u)$ here.};

	\end{tikzpicture}
	
\end{center}
\caption{Illustration of the diffractive improvement.  The trajectory
  at lower left is $\gamma_+((-\varepsilon,0))$; its continuation across
  the interface is $\gamma_+((0,\varepsilon)).$  The other incident
  trajectory at lower right is $\gamma_-((-\varepsilon,0))$.
  \blue{The limitation on the propagation of regularity through the
    interface is $s \leq r+\alpha.$}\label{fig:diffractive}}

\end{figure}

           For an illustration, see Figure~\ref{fig:diffractive}. We refer to this result as a ``diffractive improvement'' as it shows
that corrections to the naive
geometric optics ansatz (wherein singularities propagate along ordinary
bicharacteristics) is in fact a small perturbation.  It is perhaps
easier to visualize the following reinterpretation in terms of reflection: let $Pu = 0$, 
where $\WF^0(u) = \emptyset$. This of course allows $\gamma_-((-\varepsilon,0))$ to possibly contain incoming singularities in $\WF^\delta(u)$ for $\delta > 0$. On the other hand, assume that
$\WF^\infty(u)$ is disjoint from $\gamma_+((-\varepsilon,0))$.  Then using
the background regularity $r=0$, the
theorem guarantees absence of $\WF^{\alpha}(u)$ along
$\gamma_+((0,\varepsilon))$. No matter how small $\delta > 0$, any incident singularity in $\WF^\delta(u)$ is
partially \emph{reflected} (the sign of $ \xi$ has flipped) to produce
at most a milder singularity --- see Figure~\ref{fig:reflection}.

The threshold $s \leq r+\alpha$ is in general sharp, as we show by example in the next section. The same example indicates that Theorem \ref{theo:propagation} may hold for $\alpha >0$, rather than just $\alpha > 1$.

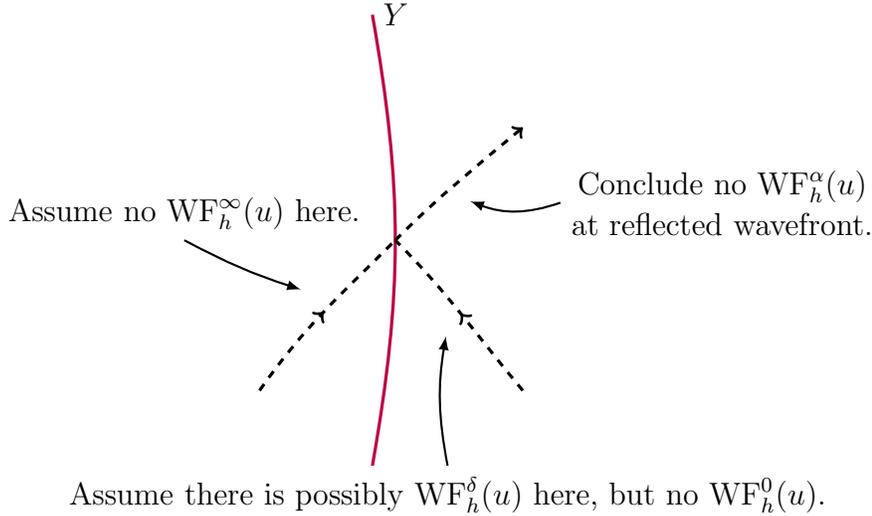
\begin{figure}
\begin{center}
	
	\begin{tikzpicture}

          \tikzset{->-/.style={decoration={
  markings,
  mark=at position .5 with {\arrow{>}}},postaction={decorate}}}

	\coordinate  (A) at (4,0);
	\coordinate [label=right:$Y$] (B) at (4,6);
	\draw[very thick, color = purple] (A) to [bend right = 10] (B);
	\path (A) to [bend right = 10] coordinate [pos=0.5] (q)  (B) ;

	\draw[->,very thick,dashed] (q) to [out=45,in=-140] (6,4.5);
\path (q) to [out=45,in=-150] coordinate [pos=0.5] (m1)  (6,4.5);
\draw[thick,->,>=latex] (6.5,3.5) to [bend left = 20] ([xshift = 2mm,yshift = -3mm] m1);
\node[right, yshift=0mm, align=center] at (6.5,3.5)(bichar){Conclude no
                  $\WF^{\alpha}(u)$ \\ at reflected wavefront.};

		\draw[->-,very thick,dashed] (6,1) to [out=130,in=-45] (q);
		\path (6,1) to [out=130,in=-45] coordinate [pos=0.5] (m2)  (q);
		\draw[thick,->,>=latex] (5,0) to [bend left = 10] ([xshift = -2mm, yshift = -3mm] m2);
		\node[below, yshift=0mm] at (5,0)(bichar){Assume there
                  is possibly $\WF^\delta(u)$ here, but no $\WF^0(u)$.};
		
		\draw[->-,very thick,dashed] (2.5,1) to [out=55,in=-135] (q);
		\path (2.5,1) to [out=55,in=-135] coordinate [pos=0.5] (m3)  (q);
		\draw[thick,->,>=latex] (1.5,3) to [bend right = 5] ([xshift = -3mm, yshift = 3mm] m3);
		\node[above, yshift=0mm] at (1.5,3)(bichar){Assume no
                  $\WF^\infty(u)$ here.};

	\end{tikzpicture}
	
      \end{center}
      \caption{Diffractive reflection of a single incident singularity.\label{fig:reflection}}
\end{figure}

One might further ask exactly what happens to semiclassical wavefront set at
points tangent to $Y;$  an understanding of diffractive improvements
along this set is essential in
understanding global propagation phenomena.  For instance, propagation
along generalized broken bicharacteristics as in
Theorem~\ref{theo:GBBpropagation} permits singularities to ``stick''
to the boundary of a convex $Y$ rather than detaching from it.  Our
final result shows that, at least for slightly more regular $V,$ this
sticking phenomenon does not in fact occur.

We consider points in the \emph{glancing set} $\gl$ (defined below in
\eqref{ellhypgl}) which is essentially the points in the
characteristic set where rays are tangent to the boundary; as $\gl$ is
technically a subset of the \emph{compressed} cotangent bundle (a
quotient of $T^*X,$ also
defined in Section
\ref{sect:bpseudo}), it is actually points in $\pi^{-1}(\gl)\subset T^*X$
at which we consider microlocal regularity, where $\pi$ is the relevant quotient map.

For the
moment we continue to assume that $\alpha > 1$, in which case $\hamvf_p$ is
a $\C^0$ vector field, hence we in general have existence but not
uniqueness of bicharacteristics (see Remark \ref{rem:nonunique} for an example where uniqueness fails). Thus, given any
$\varpi_0 \in \chare$, there exists at least one bicharacteristic
$\gamma : (-\varepsilon,\varepsilon) \rightarrow \chare$ with
$\gamma(0) = \varpi_0$. If $\alpha > 2$, then the Hamilton vector field is
Lipschitz and this bicharacteristic is
unique.

\begin{theo}[Diffractive improvement at glancing] \label{theo:improvedglancing}
	Let $\alpha > 1$ and $r \in \RR$. Let $\varpi_0 \in \pi^{-1}(\gl)$. Suppose that $u$ is $h$-tempered in $H^1_{h,\loc}(X)$ with $Pu \in L^2_\loc(X)$, and $\WF^{r+1}(Pu) = \emptyset$. If $\varpi_0 \in \WF^{r}(u)$, then there exists $\varepsilon > 0$ and a bicharacteristic $\gamma$ with $\gamma(0) = \varpi_0$ such that 
	\[
	\gamma((-\varepsilon,0)) \subset \WF^{r}(u).
	\]
\end{theo}

While this theorem certainly holds for the range $\alpha>1,$ it is
considerably more powerful when $\alpha>2$ since the set
\[
\{\gamma((-\varepsilon,0)): \gamma \text{ is a bicharacteristic}, \, \gamma(0) = \varpi_0\} 
\]
consists of the \emph{unique} solution to Hamilton's equations on
$(-\varepsilon, 0]$ with
$\gamma(0)=\varpi_0$; in this case, the theorem proves the
``non-sticking'' alluded to above, as it shows that a singularity in
$\gl$ propagates along the unique ordinary bicharacteristic through
that point rather than along one of the many possible generalized
broken bicharacteristics: 
to see this we use
Theorem~\ref{theo:improvedglancing} to obtain absence of $\WF^r(u)$ at
$\varpi_0$ based on regularity along the backward bicharacteristic; if
the bicharacteristic is, e.g., tangent to $Y$ at the single point
$\varpi_0$ before leaving it, then since $\WF^r(u)$ is closed, we
obtain this regularity at nearby points, and may propagate it forward
over $X\backslash Y$ (by the usual propagation of singularities) to
obtain absence of $\WF^r(u)$ along the whole bicharacteristic --- see
Figure~\ref{fig:glancing}.

It would be of considerable interest to know in more detail what
happens in the range $1<\alpha< 2.$ We at least know that
singularities propagate along one or more of the non-unique
bicharacteristics; it is possible that bicharacteristics sticking to
the interface $Y$ may gain regularity at a fixed rate as they do so.

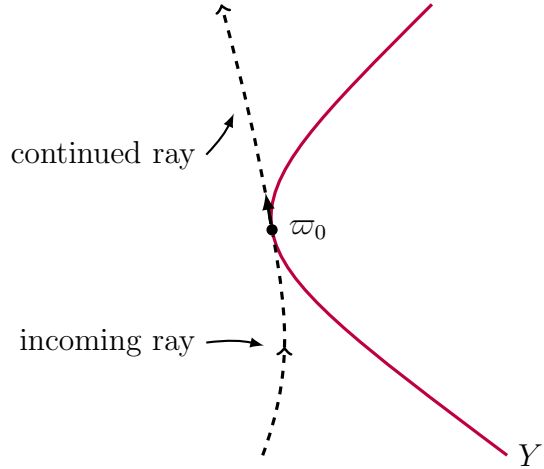
\begin{figure}
\begin{center}
	
	\begin{tikzpicture}
	
	          \tikzset{->-/.style={decoration={
				markings,
				mark=at position .5 with {\arrow{>}}},postaction={decorate}}}
	
	\coordinate [label=right:$Y$] (A) at (5,0);

		\draw[very thick,color=purple] (5,0) .. controls (1,3) .. (4,6);
	
	\path (5,0) .. controls (1,3) .. (4,6)
	node [circle,inner sep=1.5pt, pos=0.5, fill = black, label=right:$\varpi_0$] (q) {} 
	;

	\draw[->,very thick,>=latex] (q) to  ++(100:5mm);

	\coordinate (tangent) at (1.2,6);

	{
	\path (1.75,0) to [out=80,in=-80] coordinate [pos=0.5] (m1)  (q);
	\draw[thick,->,>=latex] (1,1.5) to [bend left = 10] ([yshift = 0mm, xshift = -2mm] m1);
		\node[left] at (1,1.5)(bichar){incoming ray};
}

	{
	\draw[very thick,->-,dashed] (1.75,0) to [out=70,in=-80] (q);
	}

	{
			\path (q) to [out = 100, in = -75] coordinate [pos=0.5] (m2)  (tangent);
		\draw[->,very thick,dashed] (q) to [out = 100, in = -75] (tangent);
			\draw[thick,->,>=latex] (1,4) to [bend right = 10] ([yshift = 0mm, xshift = -2mm] m2);
		\node[left] at (1,4)(bichar){continued ray};
	}
	
	\end{tikzpicture}
      \end{center}
      \caption{A bicharacteristic (dashed line) that is
        tangent to $Y.$  For any $r$, absence of $\WF^r(u)$ on the part of the
        bicharacteristic marked ``incoming ray'' implies absence of
        $\WF^r(u)$ at $\varpi_0;$ since wavefront set is closed, ordinary
        propagation of singularities then gives absence of $\WF^r(u)$ on the part
        of the bicharacteristic labeled ``continued ray,'' i.e.,
        propagation of regularity along this bicharacteristic.  (We
        are assuming $\alpha>2.$) \label{fig:glancing}}
    
\end{figure}

\subsection{A one-dimensional example} \label{subsect:1d}

On $\RR$, consider a compactly supported potential $V \in L^\infty(\RR)$ with the following properties:
\begin{itemize} \itemsep6pt 
	\item  $V = x_+^\alpha$ on an interval $(-\infty,x_0)$ with $x_0 \in (0,1)$, where $\alpha > 0$.
	\item  $V$ is $\CI$ away from $x=0$, and $\sup V < 1$.
\end{itemize}
Observe that $V \in I^{[-1-\alpha]}(\{x=0\})$. Consider the operator $P = (hD_x)^2 + V$. Working at energy $E=1$, away from the support of $V$ solutions to $(P-1)u = 0$ are linear combinations of $e^{\pm ix/h}$. There is a unique solution of the equation $(P-1)u=0$ such that
\begin{equation} \label{eq:scattering}
u = \begin{cases} e^{ix/h} + Re^{-ix/h} &\text{ for } x \leq 0, \\ Te^{ix/h} &\text{ for } x \gg 1, \end{cases}
\end{equation}
where $R,T \in \CC$. 

\begin{prop} \label{prop:1d}
	If $\alpha \in (0,1)$, then $R \sim 2^{-\alpha-2}e^{i\alpha\pi/2}\Gamma(\alpha+1)h^{\alpha}$ as $h\rightarrow 0$.
\end{prop}

Note that to leading order $R$ is independent of the choice of potential satisfying the properties above. Thus reflected waves exist and are exactly order $h^\alpha$ in this
simple example.  A proof of this result is given in Appendix
\ref{appendix:1d}.

Proposition \ref{prop:1d} is almost certainly true for $\alpha \geq 1$
as well; see Figure \ref{fig:numerics} for a numerical example. An analytic proof would require computing lower order terms in various
asymptotic expansions that quickly becomes impractical. For the case of
integer $\alpha=k \in \NN$ an analysis of this problem can be found in
Berry \cite{berry1982semiclassically}, where it is shown that if the
$k$'th derivative of the potential is discontinuous, then the
reflection coefficients are (to top order) explicit multiples of the jump in
$V^{(k)}$ times $h^k$.

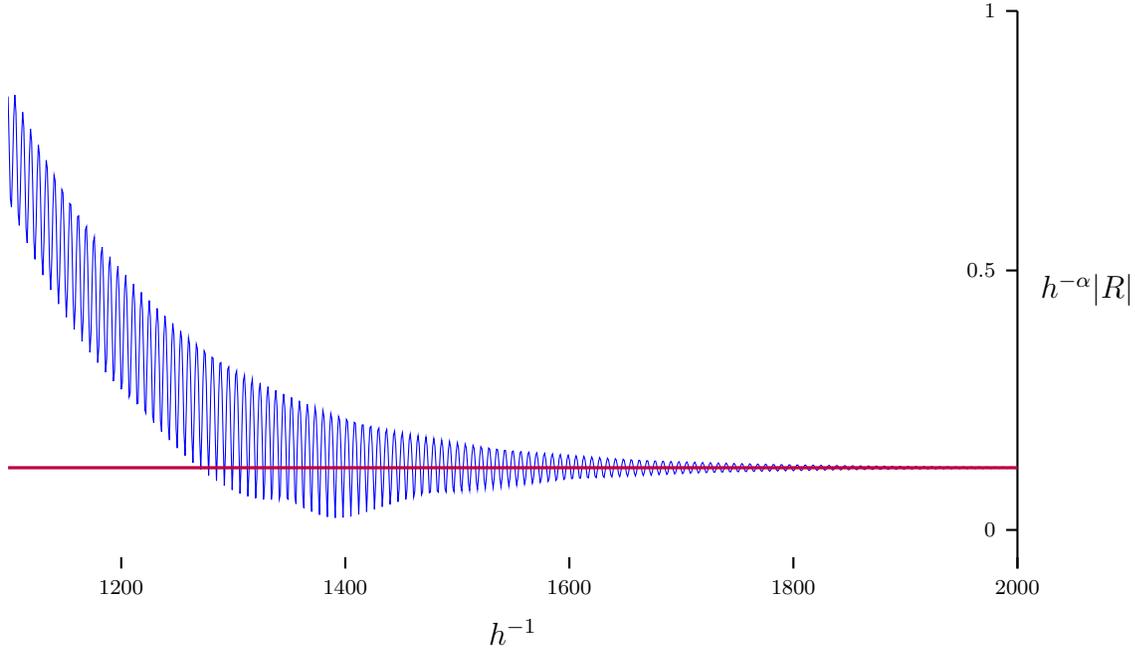
\begin{figure}
	\begin{center}
\begin{tikzpicture}
	\pgfkeys{/pgf/number format/1000 sep={}}
\begin{axis}[
height=9cm, width=15cm, xlabel=$h^{-1}$, ylabel=${h^{-\alpha}|R|}$, ticklabel style = {font=\tiny},
axis y line*=right,
ylabel near ticks, 
axis x line*=bottom,
axis line style={-},
xmin = 1099,
xmax = 2000,
ymax = 1,
x axis line style={draw=none},
xlabel style={yshift=0cm},
ylabel style={yshift=-.3cm},
xtick distance={200},
ytick distance={.5},
ylabel style={rotate=-90},
yticklabel style={anchor=east},
yticklabel shift=-4pt,
thick,
every tick/.style={color=black, thick}
]

\addplot[color=blue, thin] file {reflection_coeff.dat};

\addplot[mark=none, domain=1099:2000, color = purple, very thick] {0.1199};

\end{axis}
\end{tikzpicture}
\end{center}
  \caption{The rescaled reflection amplitude corresponding to a potential as in Section \ref{subsect:1d} with $\alpha = 1.2$ plotted against $h^{-1}$. The horizontal line represents the analytic expression from Proposition \ref{prop:1d}. The limiting asymptotics only emerge for very small values $h \sim 10^{-3}$; this phenomenon was already observed in \cite{berry1982semiclassically}.
  	\label{fig:numerics}}. 
\end{figure}

\subsection{Related work}
While there is little literature on semiclassical problems with
rough coefficients, the related problem of the wave equation with a
rough metric has attracted considerable attention.  In particular,
there is a long history of propagation of singularities theorems in
the setting of $\mathcal{C}^{k,\alpha}$ coefficients, showing
propagation of smoothness along bicharacteristics up to a maximum
level of regularity as in our Theorem~\ref{theo:propagation}; see Bony
\cite{bony1981calcul}, Beals--Reed \cite{beals1984microlocal}, Smith
\cite{smith1998parametrix,smith2014propagation}, Geba--Tataru
\cite{geba2007phase}, Taylor \cite{taylor2007tools}.

While the papers listed above are primarily focused on
unstructured coefficient singularities, the only prior study on
\emph{conormal} singularities appears to be the work of De
Hoop--Uhlmann--Vasy \cite{de2014diffraction}.  This paper, which deals
with the wave equation with coefficients in $I^{[-1-\alpha]}(Y)$ for
$Y$ a hypersurface and $\alpha>1,$ was the primary inspiration for our work.  The
authors are able to show that singularities propagate along
generalized broken bicharacteristics and that transversely reflected singularities
are weaker, in analogy with our first two theorems, although the
regularity obtained for the reflected wave (i.e., the threshold
regularity up to which one can obtain propagation results based on a
fixed level of background regularity) does not appear to be sharp.
Differences in the approach taken here include use of mixed-norm rather than
$L^2$ estimates in the commutator arguments, as well as a precise
decomposition of the potential into high and low frequencies.

In the semiclassical case, there are explicit one-dimensional
computations due to Berry
\cite{berry1982semiclassically}. Semiclassical diffraction
effects from potentials with conical singularities have been studied
by Fermanian-Kammerer--G\'erard--Lasser \cite{fermanian2013wigner} and
Chabu \cite{chabu2016analyse}, \cite{chabu2017semiclassical}; \blue{see also
Harris--Lukkarinen--Teufel--Theil \cite{harris2008energy} for a
discussion of potentials with singularities of the form $\lvert x \rvert$}.  A closely related problem 
of propagation of semiclassical defect measure across
an interface whose width shrinks at an
$h$-dependent speed has also been studied
by Nier \cite{nier1996semi} and Miller \cite{miller1997short}.

The principal novelties of this paper, in addition to obtaining in a
semiclassical setting results analogous to those of
\cite{de2014diffraction}, are, first, the sharpness of the regularity
of the diffracted wave, and, second the improvement at glancing, which
ensures that for $\alpha>2$ there is no sticking of singularities to the
boundary.

\subsection{Strategy of proof}

We follow the same overall strategy as employed in the study of the
wave equation in \cite{de2014diffraction}.  We obtain
Theorem~\ref{theo:GBBpropagation} by a commutator argument in a
semiclassical version of Melrose's \emph{b-calculus} of
pseudodifferential operators.  This calculus, which loosely speaking
consists of operators
$$
A=A(x,y, h x D_x, h D_y)
$$
where $x$ is a defining function for $Y,$ are effective at localizing
in both position and tangential momentum with respect to $Y,$ but not
in the normal momentum, since $h D_x$ is not in the calculus.  This
makes these operators useful for proving that the tangential momentum
is conserved in the interaction of singularities with the boundary,
which is the main content of the propagation along $\GBB$s theorem
(albeit at glancing the connection to the definition of $\GBB$s is
somewhat tricky to untangle).  Such a strategy, employing a positive
commutator argument, goes back to the
original work of Melrose--Sj\"ostrand on boundary problems
\cite{melrose1978singularities,melrose1982singularities}; our
approach is strongly influenced by Vasy's work on manifolds with corners \cite{vasy2008propagation}.

The diffractive improvement at transverse reflections is obtained
instead via a commutator argument involving a commutant that is an
\emph{ordinary} semiclassical pseudodifferential operator, ignoring
the singularity of the operator $P$ across $Y.$ The price one pays is
that the commutator is then no longer a pseudodifferential operator,
but involves operators whose Schwartz kernels are \emph{paired
  Lagrangian distributions}, which must be estimated separately.  It
is in the estimates of these terms that we are forced to use
assumptions on the background regularity of $u,$ and it is here that
limitations are placed on the range of exponents for which we can
expect to obtain propagation of regularity directly across the
interface.

Paired Lagrangians were introduced in the setting of
homogeneous microlocal analysis by Guillemin--Uhlmann
\cite{guillemin1981oscillatory} and Melrose--Uhlmann
\cite{melrose1979lagrangian} and studied by Antoniano--Uhlmann
\cite{antoniano1985pseudodifferential}, Greenleaf--Uhlmann
\cite{greenleaf1990estimates,greenleaf1993recovering}, and De
Hoop--Uhlmann--Vasy \cite{de2014diffraction}.  There seems to be very
little literature on these objects in the semiclassical setting,
however, so we have provided a self-contained presentation of the
basic theory here.  

One key to obtaining the sharp threshold regularity in the transverse
reflection theorem is to estimate certain terms by using
\emph{mixed-norm} estimates in the space $L^\infty(\RR_x; L^2(Y))$
(where $x$ is a defining function for $Y$) rather than the $L^2$
estimates customary in commutator arguments.  Our ability to work in
this space relies on a simple energy estimate similar to the estimates
standard in hyperbolic problems.  Another novelty to our approach is
the decomposition of the potential $V$ into low- and
high-frequency pieces, which simplifies the decomposition of the
commutator into paired Lagrangian pieces, one of which is nearly
microlocal.  This decomposition is readjusted from step to step in the
iterative commutator argument to allow for shrinking microsupports
necessary in the iteration.

The improvement in the glancing region is obtained much as in the
case of transverse interaction, with the important difference that we
are able to microlocalize the necessary background regularity more
finely: we require only background regularity in a region of specified
tangential momentum very close to glancing.  In this region,
b-regularity and ordinary regularity turn out to be essentially
interchangeable, and we are thus able to make a propagation argument
that can be \emph{iterated} as in the usual commutator proof, with the necessary
background regularity being obtained at each inductive step by the
output of the previous one.

The structure of the paper is as follows.   In
Section~\ref{sec:microlocal} we discuss background from
microlocal analysis, starting with a description of the properties of the
class of conormal distributions from which $V$ is drawn (Section~\ref{subsect:conormaldistributions}).
We then discuss pseudodifferential operators, starting with the
ordinary semiclassical calculus and associated conormal distributions (to set notation and as a point of
comparison), also recalling some basic energy estimates. Next, we move on to the \emph{semiclassical b-calculus}
(Section~\ref{sect:bpseudo}), which is the essential tool in proving
Theorem~\ref{theo:GBBpropagation}.  \blue{This section introduces the
  wavefront sets that we use to measure regularity; we need both
  the semiclassical b-wavefront set with respect to $L^2$, and the analogous wavefront
  set measured with respect to the energy space and its dual. The
  relationships among these wavefront sets, elucidated in
  Lemma~\ref{lem:equivalentWF}, explain the different wavefront sets
  arising in Theorem~\ref{theo:GBBpropagation}.}

In Section~\ref{sec:bichars}
  we then discuss the geometry of bicharacteristics, which for our
  purposes are of two  kinds: the \emph{generalized
    broken bicharacteristics}, the largest set along which
  singularities may propagate, and the ordinary solutions to
  Hamilton's equations (well-defined whenever $\alpha>1$, and for transverse rays even when $\alpha > 0$) which are
  distinguished by our diffractive improvements at hyperbolic (i.e.,
  transverse) and
  glancing sets.

  Section~\ref{sec:GBB} is devoted to the proof of
  Theorem~\ref{theo:GBBpropagation}. This splits into three steps,
  where we must first treat estimates on the \emph{elliptic set} for the
  operator and then prove distinct propagation estimates on the
  \emph{hyperbolic set} (rays transverse to $Y$) and on the
  \emph{glancing set} (rays tangent to $Y$).

We then turn to setting the stage for the proofs of Theorems
\ref{theo:propagation} and \ref{theo:improvedglancing}.  We
begin in Section~\ref{sec:paired} by introducing the calculus of
\emph{semiclassical paired Lagrangian distributions}, together with
associated operator estimates.   Finally in
Section~\ref{sec:diffractive} we prove Theorems \ref{theo:propagation}
and \ref{theo:improvedglancing}.
  
\tableofcontents

\subsection*{Acknowledgements}
The authors are very grateful to Andr\'as Vasy for advice and
encouragement as well as to Jeff Galkowski for helpful discussions.
\blue{Dean Baskin and Jeremy Marzuola provided helpful comments on the
manuscript (including corrections to Lemma~\ref{lem:sigmanotzero}).
We are especially grateful to Luc Hillairet for his many suggestions
and corrections,
as well as to two anonymous referees.}
OG was partially supported by NSF grant DMS--1502632; JW was partially
supported by NSF grant DMS--1600023.

\section{Microlocal and semiclassical preliminaries}
\label{sec:microlocal}
\subsection{Conormal distributions} \label{subsect:conormaldistributions} In this section we record H\"older and integrability properties of conormal distributions not discussed in standard references such as \cite[Chapter 18.2]{hormander1994analysis}. While these facts are well known, we were unable to find a suitable reference in the existing literature.

Let $X$ be an $m$-dimensional manifold without boundary, and $Y
\subset X$ a codimension-$k$ submanifold.  \blue{Let $\CmI_c(X)$
  denote the space of compactly supported distributions on $X.$}

\blue{Given
a closed conic Lagrangian submanifold $\Lambda \subset T^*X$, let
$I^m(X;\Lambda)$ denote the space of Lagrangian distributions of order
$m$ as defined in \cite[Chapter
25.1]{hormander1985analysis}.  For $\mu \in \RR$, we define the conormal
distributions of order $\mu$ with respect to $Y$ as
\begin{equation}\label{muorder}
I^{[\mu]}(Y) = I^{\mu + (2k-m)/4}(X; N^*Y).
\end{equation}
Recall that our standing assumption on the potential $V$ is that $V
\in I^{[-1-\alpha]}(Y)$ is real valued, with $\alpha>0$ and $Y$ a hypersurface.}

In elucidating the class of conormal distributions, we first recall the local characterization of $u\in I^{[\mu]}(Y)$ via the Fourier transform. Let $\mathcal{U}$ be a coordinate patch intersecting $Y$ with local coordinates 
\[
x  = (x',x'') = (x'_1,\ldots,x'_k,x''_1,\ldots,x''_{m-k})
\] 
such that $\mathcal{U} \cap Y = \{x'=0\}$. Assume that $u$ has compact support in $\mathcal{U}$; since $I^{[\mu]}(Y)$ is a $\CI(X)$-module, one can always reduce to this case by passing to a partition of unity subordinate to a covering of $X$ by coordinate patches. Thus we consider $u \in \CmI_c(\RR^m)$ of the form 
\begin{equation} \label{eq:inverseFT}
u(x) = (2\pi)^{-(m+2k)/4}\int e^{i\left<x',\xi'\right>} a(x,\xi') \,d\xi'
\end{equation}
for a symbol $a \in S^{\mu}(\RR^{m}_{x};\RR^k_{\xi'})$.

If $\mu < -k$, then $a \in L^1(\RR^m)$, so certainly $u$ is continuous by the Riemann--Lebesgue lemma. In fact, $u$ has much stronger continuity properties; to describe these properly, we must first recall the Zygmund spaces. If $1 = \sum_{j \geq 0} \psi_j$ is a dyadic partition of unity on $\RR^k$ with $\psi_j(\xi) = \psi_1(2^{-j}\xi)$ and $\supp \psi_1 \subset \{1/2 \leq |\xi| \leq 2\}$, then the Zygmund space $\C_*^s(\RR^k)$ consists of all distributions $v \in \mathcal{S}'(\RR^k)$ for which
\[
\| v \|_{\C^s_*} = \sup_j 2^{sj} \| \psi_j(D_{x'}) v \|_{L^\infty} < \infty. 
\]
Directly from the Littlewood--Paley characterization of $\C_*^s$ given above, any $u$ of the form \eqref{eq:inverseFT} satisfies 
\[
u \in \CI(\RR^{m-k}_{x''}; \C^{-\mu-k}_*(\RR^k_{x'})).
\]
We now return to the assumption that $\mu < -k$. It is well known (see
e.g. \cite[Section 13.8]{taylor1997partial})
that if $s = r + \alpha$ for some $r \in \mathbb{N}$ and $\alpha \in (0,1)$, then $\C^s_*(\RR^k)$ agrees with the H\"older space $\C^{r,\alpha}(\RR^k)$. From this, we immediately obtain the following lemma.

\begin{lem}\label{lemma:holder}
	If $\mu < -k$, then there exists $\theta \in (0,1]$ depending only on $\mu + k$ such that any $u \in \CmI_c(\RR^m)$ of the form \eqref{eq:inverseFT} satisfies
	\begin{equation} \label{eq:holder}
	|u(x) - u(y)| \leq C (|x'-y'|^{\theta} + |x''-y''|)
	\end{equation} 
	for each $x,y\in \RR^m$.
\end{lem}
\begin{proof}
	If $-k - \mu \in (0,1)$, then we can let $\theta = -k-\mu$. If $-k - \mu > 1$, then actually $u \in \C^1(\RR^m)$, and we can take $\theta = 1$. The case $-k - \mu= 1$ is borderline in the sense that $\C_*^1(\RR^k)$ functions are not necessarily Lipschitz,  although \eqref{eq:holder} is certainly valid for any $\theta \in (0,1)$.
\end{proof}

More concisely, if $\mu_0 > \mu$, then we can take
$\theta = \min(1,-\mu_0-k)$ in \eqref{eq:holder}. For general
$\mu \in \RR$, the distribution $u$ need not be represented by a
locally integrable function; on the other hand, we have the following
sufficient criterion:

\begin{lem} \label{lem:conormalintegrable}
	If $ -k < \mu < 0$, then any $u \in \CmI_c(\RR^m)$ of the form \eqref{eq:inverseFT} satisfies $u \in L^1(\RR^m)$, and moreover 
	\[
	|u(x)| \leq C|x'|^{-\mu-k}
	\]
	for $x' \neq 0$.
\end{lem}
\begin{proof}
  Since $u \in \CI(\RR^m \setminus \{x'=0\})$, it follows that
  $u(x) = \sum_{j \geq 0} \psi_j(D_{x'})u(x)$ for $x' \neq 0$. Now
  $|x'|^{-\mu-k}$ is locally integrable (since $-\mu-k > -k$) so by
  the dominated convergence theorem it suffices to show that
	\begin{equation} \label{eq:muk}
	\sum_{j=0}^N |\psi_j(D_{x'})u(x)| \leq C|x'|^{-\mu-k}
	\end{equation}
	for $x'\neq 0$ and every $N \geq 0$, where $C$ does not depend
        on $N$.  \blue{(This will establish that $u$ differs from a
         locally $L^1$ function by a distribution supported along $\{x'=0\},$ and
          the latter are ruled out since $\mu<0.$)}
           \blue{Integration by parts using the
        operator $\Lap_{\xi'}^M$ now yields
	\begin{equation} \label{eq:mukM} 
	|\psi_j(D_{x'}) u(x) | \leq C_M |x'|^{-2M} 2^{j(\mu+k - 2M)}
	\end{equation}
	for each $M \in \mathbb{N}$. Now simply split the sum
        \eqref{eq:muk} into two pieces, the first where $2^j <
        |x'|^{-1}$, taking $M= 0$ in \eqref{eq:mukM} and using  that
        $\mu + k > 0$, and the second where $2^j \geq |x'|^{-1}$,
        taking $2M > \mu + k$ in \eqref{eq:mukM}. }
\end{proof}

Under the
hypotheses of Lemma \ref{lem:conormalintegrable} and applying the mean value theorem in the $x''$ variables,
\begin{equation} \label{eq:L1lipschitz}
|u(x',x'') - u(x',y'')| \leq C|x'|^{-\mu-k}|x''-y''|
\end{equation}
for $x' \neq 0$ and $x'',y'' \in \RR^{m-k}$. This estimate will be important when discussing Hamilton's equations in Section \ref{subsect:hamiltonflow}.

\begin{lem} \label{lem:vanishes}
	Let $\mu < -k + 1$. If $u$ is given by \eqref{eq:inverseFT} and $f \in \C^1(\RR^m)$ vanishes along $Y = \{x'=0\}$, then $f u$ vanishes along $Y$.
\end{lem}
\begin{proof}
	We may assume that $f$ is given by one of the coordinate functions $f= x_j'$.  Upon splitting $\xi' = (\xi'_j,\xi'')$,
	\[
	(fu)(0,x'')  = \int D_{\xi_j'} a(0,x'',\xi) \, d\xi = \int_{\RR^{k-1}} \int_{\RR} D_{\xi_j'}a(0,x'',\xi) \, d\xi'_j \, d\xi'' = 0
	\]
	by Fubini's theorem, since $D_{\xi_j'}a(0,x'',\cdot) \in L^1(\RR^k)$. 
\end{proof}	

Suppose that $u$ and $f$ are as in Lemma \ref{lem:vanishes}, where $\mu < -k+1$. Combined with the H\"older bound \eqref{eq:holder}, we conclude that 
\begin{equation} \label{eq:vanishingholder}
|(f u)(x)| \leq C|x'|^\theta
\end{equation}
for some $\theta \in (0,1)$ depending only on $\mu + k$.

\subsection{Semiclassical pseudodifferential operators} \label{subsect:semiclassicalpseudos}
Next, we give a brief overview of the semiclassical analysis used in this
paper. For a detailed exposition, the reader is referred to \cite{zworski2012semiclassical} and \cite[Appendix E]{zworski:resonances}.

We say that an $h$-dependent  family of symbols $a(x,\theta) = a(x,\theta;h)$ is in $S^m(\RR^p_x;\RR^q_\theta)$ if the usual symbol  bounds 
\[
|D_{x}^\alpha D_{\theta}^\beta a(x,\theta)| \leq C_{\alpha\beta}\left<\theta\right>^{m-|\beta|}
\]
are uniform in $h\in (0,1)$. We also say that $a(x,\theta) \in S^{\COMP}(\RR^p;\RR^q)$ if $a$ is supported in an $h$-independent compact set, and its $\CcI(\RR^p\times \RR^q)$ seminorms are all uniformly bounded in $h$. 

On $\RR^n$, we obtain an operator from $a(x,\xi) \in S^m(\RR^n;\RR^n)$ by the standard left quantization procedure,
\begin{equation} \label{eq:semiclassicalquantizationRn}
\Op_{h}(a)u(x) = (2\pi h)^{-n} \int e^{\sfrac i h \langle x-y, \xi \rangle}a(x,\xi)u(y) \, dy d\xi.
\end{equation}
This operator acts on $\mathcal{S}(\RR^n)$ and $\mathcal{S}'(\RR^n)$.

For a manifold $X$, we similarly define the class of $h$-dependent symbols on $T^*X$, which we continue to denote by $S^m(T^*X)$. The space $S^\COMP(T^*X)$ is defined analogously. We use semiclassical pseudodifferential operators $\Psi^m_h(X)$ with symbols in $S^m(T^*X)$. For simplicity, assume that $X$ is compact; this is only used to avoid issues such as proper supports, and is inessential. The space $\Psi^m_h(X)$ enjoys the following properties:
\begin{enumerate}  [label=(\Roman*)] \itemsep6pt
	
	\item Each $A \in \Psi_h^m(X)$ maps $\CI(X) \rightarrow \CI(X)$ and $\CmI(X) \rightarrow \CmI(X)$.
	
	\item There is a principal symbol map $\sigma_h: \Psi_h^m(X) \rightarrow S^m(T^*X)/hS^{m-1}(T^*X)$ such that the sequence
	\[
          0 \rightarrow
         h\Psi_h^{m-1}(X) \rightarrow \Psi_h^m(X)
         \xrightarrow{\sigma_h} S^m(T^*X)/hS^{m-1}(T^*X)\rightarrow 0
	\]
	is exact.
	\item There exists a (non-canonical) quantization map $\Op_h : S^m(T^*X) \rightarrow \Psi_h^m(X)$ such that if $a \in S^m(T^*X)$, then 
	\[
	\sigma_h(\Op_h(a)) = a
	\]
	in $S^m(T^*X)/hS^{m-1}(T^*X)$.

	\item If $A \in \Psi^m_h(X)$, then $A^* \in \Psi^m_h(X)$ with principal symbol 
	\[
	\sigma_h(A^*) = \overline{\sigma_h(A)}.
	\] 
	Here the adjoint is taken with respect to any fixed density on $X$. 
	
	\item If $A \in \Psi_h^m(X)$ and $B\in\Psi_h^{m'}(X)$, then $[A,B] \in h\Psi_h^{m+m'-1}(X)$ with principal symbol 
	\[
	\sigma_h(\tfrac{i}{h}[A,B]) = \{\sigma_h(A),\sigma_h(B)\} = \hamvf_{\sigma_h(A)} \sigma_h(B)
	\]
	where $\{\cdot,\cdot\}$ is the Poisson bracket, and $\hamvf_f$ is the Hamilton vector field of a function $f$ on $T^*X$.
	\item Each $A \in \Psi_h^m(X)$ extends to a bounded operator $H^s_h(X) \rightarrow H^{s-m}_h(X)$. Moreover, if $A \in \Psi_h^0(X)$, then there exists $A' \in \Psi_h^{-\infty}(X)$ such that
	\begin{equation} \label{eq:semiclassicalL2bounded}
	\|Au\|_{L^2} \leq 2\sup |\sigma_h(A)| \|u\|_{L^2} + \mathcal{O}(h^\infty)\|A'u\|_{L^2}
	\end{equation}
	for each $u\in L^2(X)$. Here $\sigma_h(A)$ is any representative of the principal symbol in $S^0(T^*X)/hS^{-1}(T^*X)$.
\end{enumerate}
In \eqref{eq:semiclassicalL2bounded}, $H^{s}_h(X)$ refers to the usual Sobolev space $H^s(X)$ but equipped with its semiclassically rescaled Sobolev norm $\| \cdot \|_{H^s_h}$. In particular, given $u \in H^1_h(X)$, we can take
\begin{equation} \label{eq:H1norm}
\| u \|_{H^1_h} = \int_X |u|^2 + h^2|du|^2 \, dg,
\end{equation}
where $dg$ is the volume density for a Riemannian metric $g$, and the magnitude of $du$ is computed with respect to $g$.

The negligible operators $h^\infty \Psi_h^{-\infty}(X)$ in this calculus are precisely those with smooth Schwartz kernels, such that each $\CI(X)$ seminorm is of order $\mathcal{O}(h^\infty)$. Given $A \in \Psi^m_h(X)$, there exists $a \in S^m(T^*X)$ such that
\begin{equation} \label{eq:Opsurjective}
A = \Op_h(a) + h^\infty \Psi_h^{-\infty}(X).
\end{equation}
The operator wavefront set (also known as the microsupport) $\WF(A)$ of $A \in \Psi^m_h(X)$ can be defined as the essential support of its full symbol in any coordinate representation. Here essential support is meant in the semiclassical sense: if $a(x,\theta) \in S^{m}(\RR^p;\RR^q)$, then
\[
\esssupp(a)^\complement = \{(x,\theta) : \, a \in h^\infty S^{-\infty}(\RR^p; \RR^q) \text{ near } (x,\theta) \}.
\]
Note that we are viewing $\esssupp(a)$ as a subset of the radial
compactification $\RR^p \times \OL{\RR^q}$. Thus $\WF(A)$ is a subset
of the fiber-radially compactified cotangent bundle $\OL{T^*}X$ (see
\cite[Section E.2]{zworski:resonances}). We also write $\ELL(A)$ for
the elliptic set of $A \in \Psi^m_h(X)$, again viewed as a subset of
$\OL{T^*}X$: \blue{this is the set where the principal symbol is invertible.}
 
 The compactly microlocalized operators $\Psi_h^\COMP(X) \subset
 \Psi^{-\infty}_h(X)$ are defined to be those with compact
 operator wavefront set in $T^*X \subset \OL{T^*}X$. Equivalently, $A \in \Psi_h^\COMP(X)$ if $A$ can be written in the form \eqref{eq:Opsurjective} with $a \in S^\COMP(T^*X)$. If $X$ is not compact, we also assume that the Schwartz kernel of $A \in \Psi_h^\COMP(X)$  has compact support in $X\times X$.

We need to consider distributions which are $h$-tempered relative to a fixed order Sobolev space.
\begin{defi}
We say that an $h$-dependent family $u = u(h) \in \CmI(X)$ is $h$-tempered in $H^s_h(X)$ if there exists $C,N>0$ such that
\[
\| u \|_{H^s_h} \leq Ch^{-N}.
\]
\end{defi}
Thus the usual notion of an $h$-tempered distribution $u \in \CmI(X)$ is that $u$ is $h$-tempered in some $H^{-M}_h(X)$.

\begin{defi} \blue{Let $r \in \RR.$}
	If $u$ is $h$-tempered in $L^2(X)$ we say that $(x,\xi) \notin \WF^r(u)$ if there exists $A \in \Psi^0_h(X)$ elliptic at $(x,\xi)$ such that
	\[
	\| A u \|_{L^2} \leq Ch^{r}.
	\]
	If $r = +\infty$, we write $\WF(u)$ for $\WF^\infty(u)$.
      \end{defi}
\noindent \blue{Recall that ellipticity, and hence wavefront set, is defined at
  points in the fiber-compactified cotangent bundle.}
      
              We will also occasionally employ a wavefront set measured with respect to
        spaces other than $L^2:$
        \begin{defi}\label{def:sobolevWF}
Let $r,s \in \RR.$          If $u$ is $h$-tempered in $H_h^{s} (X)$ we say that
          $(x,\xi) \notin \WF^{s,r}(u)$ if there exists
          $A \in \Psi^0_h(X)$ elliptic at $(x,\xi)$ such that
	\[
	\| A u \|_{H^{s}_h} \leq Ch^{r}.
	\]
      \end{defi}

Lastly, we consider a class of ``tangential'' pseudodifferential operators on $\RR^{d+1}$. Fix a splitting of coordinates $x = (x_1,x') \in \RR \times \RR^d$. Given $k \in \NN \cup \{+\infty\}$, we consider operators
\[
Q \in \C^k(\RR_{x_1}; \Psi^m_h(\RR^{d}_{x'})).
\]
Thus we can write $Q = \Op_h(q)$, where $q \in \C^k(\RR; S^m(\RR^d))$ and $\Op_h$ denotes the quantization procedure \eqref{eq:semiclassicalquantizationRn} on $\RR^d$. However, since $q$ is not necessarily smooth in $x_1$, the notion of operator wavefront set must be modified. We say that $(x,\xi') \notin \esssupp(q)$ if there is a neighborhood of $(x,\xi')$ in $\RR^{d+1} \times \OL{\RR^{d}}$ where
\[
D_{x_1}^j D_{x'}^\alpha D_{\xi'}^\beta q(x_1,x',\xi') = \mathcal{O}(h^\infty \left<\xi'\right>^{-\infty})
\]
 for $j \leq k$. We then define $\WF(Q) = \esssupp(q)$. This definition guarantees that $\WF(\partial_{x_1}^k(Q)) \subset \WF(Q)$ for $k\geq 1$.

\subsection{Energy estimates}
\label{subsect:energy}

In this section we prove a microlocal energy estimate that will eventually be applied to the operator $P$ in \eqref{eq:P}. These estimates follow the strategy used in \cite[Sections
23.1--23.2]{hormander1994analysis} for hyperbolic operators; similar
estimates for semiclassical problems have also been obtained in
\cite[Section 3.2]{christianson2011quantum}.

\blue{In what follows we will employ the notation $\Diff^k_h$ for the algebra of
semiclassical differential operators
$$
\sum_{\lvert \alpha \rvert\leq k} a_\alpha (x; h) (hD)^\alpha
$$
with $a_\alpha \in \CI,$ uniformly in $h\to 0.$}

We work on $\RR^{d+1}$. Let $x = (x_1,x') \in\RR \times \RR^{d}$, and
consider a \blue{differential} operator
\[
L = (hD_{x_1})^2 - R + hR_0
\]
where $R \in \C^1(\RR; \Diff^2_h(\RR^{d}))$ and $R_0 \in \C^1(\RR; \Diff^1_h(\RR^d))$. Writing $r(x,\xi') = \sigma_h(R)$, we make the following microlocal hyperbolicity assumption:
\[
r(x,\xi') > 0 \text{ near } (-\varepsilon,\varepsilon) \times U,
\]
where $U \subset T^*\RR^{d}$ is open with compact closure. Therefore we can find a self-adjoint tangential operator $\Lambda \in \C^1(\RR;\Psi^{\COMP}_h(\RR^{d}))$ with $\sigma_h(\Lambda) = r^{1/2}$ 
near $(-\varepsilon,\varepsilon) \times U$ such that
\[\Lambda^2 = R + R',\] where $R' \in \C^1(\RR;\Psi^{2}_h(\RR^d))$ and $ (-\varepsilon,\varepsilon) \times U \cap \WF(R') = \emptyset$. Then we have 
\begin{align*}
(hD_{x_1} \mp \Lambda)(hD_{x_1} \pm \Lambda) &= (hD_{x_1})^2 - \Lambda^2 \pm [hD_{x_1},\Lambda] 
\\ &= L + R' \pm hR_1,
\end{align*}
where $R_1 = h^{-1}[hD_{x_1},\Lambda] \pm R_0 \in \C^0(\RR; \Psi^\COMP_h(\RR^d)) +  \C^1(\RR; \Diff^1_h(\RR^d))$.  Given $u \in \CI(\RR^{d+1})$, write $u(x_1)$ for the function $x' \mapsto u(x_1,x')$ on $\RR^d$.

\begin{lem} \label{lem:energyestimate}
	If $A \in \CI(\RR;\Psi^{\COMP}_h(\RR^d))$ satisfies $\WF(A) \subset (-\varepsilon,\varepsilon) \times U$ and $B \in \CI(\RR;\Psi^{\COMP}_h(\RR^d))$ is elliptic on $\WF(A)$, then 
\begin{equation}\label{Luc}	\begin{aligned}
	\| Au(x_1) \|_{L^2(\RR^d)} &\leq C\| \left<hD_{x_1}\right> Bu \|_{L^2} + C h^{-1} \int_{-\varepsilon}^\varepsilon \| BLu(s) \|_{L^2(\RR^d)} \, ds \\ &+ \mathcal{O}(h^\infty)\| \left<hD_{x_1}\right> u \|_{L^2}
	\end{aligned} 	\end{equation}
	for every $u \in \CI(\RR^{n+1})$ and $x_1 \in (-\varepsilon,\varepsilon)$.
\end{lem} 
\noindent \blue{Since \eqref{Luc} is the first of many
  estimates of this form, we clarify that the inequality means
  that there exists $C$ fixed such that for every $M \in \NN,$ there
  exist $C_M$ and $h_0=h_0(M)$ such that for $h \in (0, h_0),$
	\begin{align*}
	\| Au(x_1) \|_{L^2(\RR^d)} &\leq C\| \left<hD_{x_1}\right> Bu
                                     \|_{L^2} + C h^{-1}
                                     \int_{-\varepsilon}^\varepsilon
                                     \| BLu(s) \|_{L^2(\RR^d)} \, ds
          \\ &+ C_M h^M\| \left<hD_{x_1}\right> u \|_{L^2}.
	\end{align*} 	}

\begin{proof} 
	The usual energy inequalities hold for the operators $hD_{x_1} \pm \Lambda$, cf \cite[Lemma 23.1.1]{hormander1994analysis}: for each $x_1,t \in \RR$,
	\begin{equation} \label{eq:energyestimateorder1}
	\| u(x_1) \|_{L^2(\RR^d)} \leq \| u(t) \|_{L^2(\RR^d)} + h^{-1} \int_{t}^{x_1} \| (hD_{x_1} \pm \Lambda)u(s)\|_{L^2(\RR^d)} \, ds.
	\end{equation}
	Given $B_1 \in \CI(\RR; \Psi^\COMP_h(\RR^d))$, set $v_\pm = B_1(hD_{x_1}\mp \Lambda)u$ and 
	compute
	\[
	(hD_{x_1}\pm \Lambda)v_\pm = B_1 Lu + [hD_{x_1} \pm \Lambda,B_1]u  \pm h B_1R_1 u  + B_1R'u 
	\]
	Take $B_1$ elliptic on $\WF(A)$ with $\WF(B_1) \subset
	(-\varepsilon,\varepsilon) \times U$ and let $B$ be elliptic on
	$\WF(B_1).$  Then
	\begin{align*}
	\| (hD_{x_1}\pm \Lambda)v_\pm(x_1) \|_{L^2(\RR^d)} &\leq C\| BLu(x_1) \|_{L^2(\RR^d)} \\ &+  Ch\| Bu(x_1) \|_{L^2(\RR^d)} + \mathcal{O}(h^\infty)\| u(x_1) \|_{L^2(\RR^d)}
	\end{align*}
	for $x_1 \in (-\varepsilon,\varepsilon)$. Applying \eqref{eq:energyestimateorder1} to $v_\pm$ yields the estimate
	\begin{multline*}
	\|v_\pm(x_1)\|_{L^2(\RR^d)}\leq \| v_\pm(t) \|_{L^2(\RR^d)}\\ + C \int_{t}^{x_1} h^{-1} \| B Lu(s) \|_{L^2(\RR^d)} + \|B u(s)\|_{L^2(\RR^d)} + \mathcal{O}(h^\infty)\| u(s)\|_{L^2(\RR^d)}\,ds
	\end{multline*}
	for $x_1,t \in (-\varepsilon,\varepsilon)$. Furthermore, we can estimate 
	\[
	\|v_\pm(t) \|_{L^2(\RR^d)} \leq C \|\left<hD_{x_1}\right> Bu(t)\|_{L^2(\RR^d)}.
	\]
	On the other hand, since $\WF(A) \subset \ELL(\Lambda)$,
	\begin{align*}
	\|Au(x_1) \|_{L^2(\RR^d)} &\leq C (\| v_{+}(x_1) \|_{L^2(\RR^d)} + \| v_{-}(x_1) \|_{L^2(\RR^d)})  +  \mathcal{O}(h^\infty) \| u(x_1) \|_{L^2(\RR^d)}.
	\end{align*}
	Estimating the $\mathcal{O}(h^\infty)\| u(x_1)\|_{L^2(\RR^d)}$ term on the right hand side by \eqref{eq:energyestimateorder1}, we conclude that
	\begin{multline*}
	\| Au(x_1) \|_{L^2(\RR^d)} \leq  C \| \left<hD_{x_1} \right>B u(t) \|_{L^2(\RR^d)} \\ + C \int_{t}^{x_1} h^{-1} \| B Lu(s) \|_{L^2(\RR^d)} + \|B u(s)\|_{L^2(\RR^d)} + \mathcal{O}(h^\infty)\|\left<hD_{x_1} \right>u(s)\|_{L^2(\RR^d)}\,ds
	\end{multline*}
	for $x_1,t \in (-\varepsilon,\varepsilon)$. Integrating in $t$ finishes the proof.
\end{proof}

\subsection{Semiclassical conormal distributions}
We return to the setting of Section \ref{subsect:conormaldistributions}, adopting the notation there. 
\begin{defi}
If $u \in \CmI_c(X)$ has compact support in a coordinate patch $\mathcal{U}$ as in Section \ref{subsect:conormaldistributions}, we say that $u \in I^{p}_h(X;N^*Y)$ if
\begin{equation} \label{eq:semiclassicalconormaldistribution}
u = (2\pi h)^{-(m+2k)/4} \int  e^{\sfrac{i}{h}\langle x',\xi' \rangle }a(x,\xi')\, d\xi'
\end{equation}
for some $a(x,\xi') = a(x,\xi';h) \in S^{p+(m-2k)/4}(\RR^m_x;\RR^{k}_{\xi'})$.
\end{defi}
The general definition of $I^{p}(X;N^*Y)$ is obtained by localization. If $u \in \CmI_c(\RR^n)$ is given by \eqref{eq:semiclassicalconormaldistribution}, then $u$ is certainly $h$-tempered, and
\[
\WF(u) \subset \{ (x,\xi) \in \OL{N^*}Y: (x,\xi') \in \esssupp(a)\}. 
\]
Here we have written $\OL{N^*}Y \subset \OL{T^*}X$ for the fiber-radially compactified conormal bundle to $Y$. 

We say that $u \in I^\COMP_h(X;N^*Y)$ if $u \in I^{-\infty}_{h}(X;N^*Y)$ has compact support, and $\WF(u)$ is compact in $T^*X$. Equivalently, $u$ can locally be written in the form \eqref{eq:semiclassicalconormaldistribution} with $a \in S^\COMP(\RR^n;\RR^k)$, modulo an $h^\infty \CcI(\RR^m)$ remainder.

\section{Semiclassical b-pseudodifferential operators} \label{sect:bpseudo}

\subsection{b-Tangent and b-cotangent bundles} \label{subsect:btangentcotangent} Let $X$ be a manifold with boundary. 
Let $\mathcal{V}(X)$ denote the Lie algebra of smooth vector fields on $X$, and $\Vb(X)$ the subalgebra of vector fields tangent to $\partial X$. Let $(x,y) = (x,y_1,\ldots,y_n)$ be local coordinates on a chart $\mathcal{U}$ intersecting $\partial X$, such that $\mathcal{U} \cap \partial X = \{x=0\}$. With respect to these coordinates, elements of $\Vb(X)$ are locally of the form
\begin{equation} \label{eq:btangentsection}
f(x,y)x\partial_x + \sum g_i(x,y) \partial_{y_i}.
\end{equation} 
Furthermore, $\Vb(X)$ coincides with sections of a bundle, the b-tangent bundle $\bT X$. There is also a natural bundle map 
\begin{equation} \label{eq:ib}
i : \bT X \rightarrow TX
\end{equation}
induced by the inclusion $\Vb(X) \hookrightarrow \mathcal{V}(X)$. Over
$q \in X^\circ$ (the interior of $X$) this map is an isomorphism,
which gives the identification $\bT_{X^\circ} X = T X^\circ$.  \blue{Here we
use the notation $\bT_Z X$ for the restriction of $\bT X$ to the
submanifold $Z.$}

The dual bundle to $\bT X$ is the b-cotangent bundle $\bT^*X = (\bT X)^*$. In coordinates $(x,y)$ near the boundary, sections of $\bT^*X$ are of the form
\begin{equation} \label{eq:bcotangentsection}
\sigma(x,y) \frac{dx}{x} + \sum \eta_i(x,y) dy_i.
\end{equation}
Thus $(x,y,\sigma,\eta)$ provide coordinates on $\bT^*X$. 
Let $\pi : T^*X \rightarrow \bT^*X$
denote the adjoint of \eqref{eq:ib}. Over the interior, $\pi$ induces a dual identification $\bT_{X^\circ}^* X = T^*X^\circ$.
On the other hand, if $(x,y,\xi,\eta)$ are the usual coordinates on $T^*X$ induced by $(x,y)$, then
\[
\pi(x,y,\xi,\eta) = (x,y,x\xi,\eta).
\]
In particular, \blue{since it maps to $\sigma=x\xi,$} $\pi$ is not surjective over $\partial X$. We denote by $\bdotT^*X$ the image $T^*X$ under $\pi$, referred to as the \emph{compressed cotangent bundle}.

By a slight abuse of notation, we also consider $T^*\partial X$ as a subset of $\bT^*_{\partial X} X$. More precisely, $i$ takes $\bT_{\partial X} X$ onto $T\partial X$, and the inclusion  $T^*\partial X \hookrightarrow \bT^*_{\partial X} X$ is the adjoint of this restriction; in local coordinates, it is just the map $(y,\eta) \mapsto (0,y,0,\eta)$.

While the definitions above apply to a manifold with boundary, for our purposes we need to replace $\partial X$ with an embedded interior hypersurface $Y \subset X$, where $X$ is now boundaryless. In that case we consider the relative b-tangent bundle $\bT(X;Y)$. Sections of $\bT(X;Y)$ coincide with the subalgebra $\Vb(X;Y) \subset \mathcal{V}(X)$ of vector fields tangent to $Y$. The discussion above applies verbatim to $\bT(X;Y)$ by replacing $\partial X$ with $Y$, and $X^\circ = X \setminus \partial X$ with $X\setminus Y$.

 \subsection{b-Pseudodifferential operators} \label{subsect:b-pseudos}

 We now describe the class of semiclassical b-pseudodifferential
 operators on a compact manifold $X$ with boundary.  This is a variant
 on the \emph{b-calculus} introduced in the setting of homogeneous
 microlocal analysis by Melrose \cite{melrose1981transformation},
 \cite{melrose1983elliptic} (see also \cite{melrose1993atiyah} for a
 detailed treatment).  A description of the semiclassical b-calculus
 employed here can be found in \cite[Appendix A]{hintz2016global}.

We begin by defining the class of residual operators $h^\infty
\Psibhc^{-\infty}(X)$. Here we resort to a geometric description in
terms of a certain \emph{blow-up} of $X\times X$ since this yields the most
concise definition.  (We refer the reader to \cite{melrose1993atiyah}
for a discussion of real blow-up in the context of the b-calculus and for further references.)

 Recall that the b-stretched product $X \times_\B X$ is defined by blowing up the corner $\partial X \times \partial X$ in $X \times X$,
 \[
 X \times_\B X= [X\times X; \partial X \times \partial X].
 \]
 The blow-down map is denoted by $\beta_{\B} : X\times_\B X \rightarrow X\times X$.  The front face, namely the lift of $\partial X \times \partial X$, is denoted $\mathrm{ff}$, whereas the lifts of $X^\circ \times \partial X$ and $\partial X \times X^\circ$ are denoted $\mathrm{lf}$ and $\mathrm{rf}$, respectively.
 
  If $M$ is a manifold with corners, we use the notation $\mathcal{A}(M)$ for the space of $L^\infty$ based conormal distributions on $M$:
 \[
 \mathcal{A}(M) = \{u \in \CmI(M): \mathcal{V}_\B(M)^k u \in L^\infty(M) \text{ for all } k \in \NN\}.
 \]
 Returning to the b-stretched product, let $\rho_{\mathrm{sf}}$ be a total boundary defining function for the side faces. We then consider operators $A$ with Schwartz kernels in $\rho_{\mathrm{sf}}^\infty \mathcal{A}(X\times_\B X)$. Note that this space has a natural family of seminorms.

\blue{In what follows $\CdI(X)$ denotes the set of
smooth functions on $X$ vanishing to infinite order at the boundary (cf.\ \cite[Appendix
B.2]{hormander1994analysis}).}
 
 \begin{defi}
 A family of operators $A= A(h) : \CdI(X) \rightarrow \CdI(X)$ belongs to $h^\infty \Psibhc^{-\infty}(X)$ if its kernel $K_A$ is the pushforward by $\beta_{\B}$ of an element \[
 \widetilde K = \widetilde{K}(h)  \in \rho_{\mathrm{sf}}^\infty \mathcal{A}(X \times_\B X),
 \] 
where each seminorm of $\widetilde{K}$ is of order $\mathcal{O}(h^\infty)$. We say that $A$ belongs to $h^\infty \Psibh^{-\infty}(X)$ if $\widetilde K$ is in addition smooth up to $\mathrm{ff}$.
 \end{defi}
 
In general, semiclassical b-pseudodifferential operators have Schwartz kernels with additional singularities on the diagonal. We choose to give a definition via localization. First we describe the appropriate semiclassical symbol classes. Let us identify 
\[
\bT^*\RR_+^n = \RR^n_+ \times \RR^n,
\]
with coordinates $(x,y) \in \RR_+ \times \RR^{n-1}$ in the first
factor, and $(\sigma,\eta)\in \RR\times \RR^{n-1}$ in the second. 
In that case, we define $h$-dependent Kohn--Nirenberg $\symbbc^m(\bT^*\RR^n_+)$ corresponding to symbol bounds of the form
\begin{equation} \label{eq:conormalsymbol}
|(xD_x)^j D_y^\alpha D_{\sigma}^k D_{\eta}^\beta a(x,y,\sigma,\eta)| \leq C_{kj\alpha\beta}\left<(\sigma,\eta)\right>^{m-k-|\beta|}
\end{equation}
uniformly in $h$. Thus $a$ need not be smooth up to the boundary of $\bT^*\RR^n_+$. If we wish to require smoothness, we can define $\symbb^m(\bT^*\RR^n_+)$ by replacing $xD_x$ with $D_x$ in \eqref{eq:conormalsymbol}. In general,  $\symbbc^m(\bT^*X)$ is defined by localization, and similarly for $\symbb^m(\bT^*X)$.

We now define a left quantization procedure on $\RR^n_+$. For this, fix $\phi \in \CI_c((1/2,2))$ such that $\phi(s) = 1$ near $s=1$. Given $a \in \symbbch^m(\bT^*\RR_+^n)$, define  $\Opbh(a)$ by
\begin{multline} \label{eq:localquantization}
\Opbh(a)u(x,y) \\ = (2\pi h)^{-n}\int e^{\sfrac i h ((x-\tilde x)\xi + \left<y-\tilde y,\eta\right>)} \phi(x/\tilde x) a(x,y,\eta,x\xi) u(\tilde x, \tilde y)\, d\xi d\eta d\tilde x d\tilde y.
\end{multline}
Semiclassical b-pseudodifferential operators are defined in general by localization:

\begin{defi} \label{defi:b-pseudosbylocalization}
	A family of operators $A = A(h) : \CdI(X) \rightarrow \CdI(X)$ belongs to $\Psibhc^m(X)$ if the following properties hold.
	\begin{enumerate} \itemsep6pt
		\item If $\varphi,\psi \in \CI(X)$ have disjoint supports, then $\varphi A \psi \in h^\infty \Psibh^{-\infty}(X)$.
		\item If $\psi \in \CI_c(O)$ has support in an interior coordinate patch $O$ and $\kappa : O \rightarrow O_\kappa \subset \RR^n$ is a diffeomorphism, then  
		$(\kappa^*)^{-1} \psi A \psi \kappa^* \in \Psi_h^m(\RR^n)$.
		\item If $\psi \in \CI_c(O)$ has support in a boundary coordinate patch $O$ and $\kappa : O \rightarrow O_\kappa \subset \RR_+^n$ is a diffeomorphism, then
		\begin{equation} \label{eq:localrep}
		(\kappa^{-1})^*\psi A \psi \kappa^* = \Opbh(a) + R
		\end{equation}
		for some $a \in \symbbch^m(\bT^*\RR_+^n)$ and $R \in h^\infty \Psibhc^{-\infty}(\RR_+^n)$
\end{enumerate}
	We say that $A$ belongs to $\Psibh^m(X)$ if $\eqref{eq:localrep}$ holds for some $a \in \symbbh^m(\bT^*\RR_+^n)$ and $R \in h^\infty \Psibh^{-\infty}(X)$.	
\end{defi}

The space \blue{$\Psibhc (X)$} of semiclassical b-pseudodifferential operators with conormal coefficients on a compact manifold $X$ with boundary has the following properties.
\begin{enumerate}  [label=(\Roman*)] \itemsep6pt
	
	\item Each $A \in \Psibhc^m(X)$ maps $\CdI(X) \rightarrow \CdI(X)$ and $\CmI(X) \rightarrow \CmI(X)$.
	\label{it:mapping}
	
	\item There is a principal symbol map $\bsymbol: \Psibhc^m(X) \rightarrow \symbbc^m(\bT^*X)/h\symbbc^{m-1}(\bT^*X)$ such that the sequence
	\[
	0 \rightarrow h\Psibhc^{m-1}(X) \rightarrow \Psibhc^m(X)
        \xrightarrow{\bsymbol}
        \symbbc^m(\bT^*X)/h\symbbc^{m-1}(\bT^*X) \rightarrow 0
	\]
	is exact. \label{it:symbol}
	\item There exists a non-canonical quantization map $\Opbh : \symbbc^m(\bT^*X) \rightarrow \Psibhc^m(X)$ such that if $a \in \symbbc^m(\bT^*X)$, then 
	\[
	\bsymbol(\Opbh(a)) = a
	\]\label{it:quantization}
	in $\symbbc^m(\bT^*X)/h\symbbc^{m-1}(\bT^*X)$.
	
	\item If $A \in \Psibhc^m(X)$, then $A^* \in \Psibhc^m(X)$ with principal symbol 
	\[
	\bsymbol(A^*) = \overline{\bsymbol(A)}.
	\] 
	Here the adjoint is taken with respect to any fixed density on $X$.  \label{it:adjoint}
	
	\item If $A \in \Psibhc^m(X)$ and $B\in\Psibhc^{m'}(X)$, then $[A,B] \in h\Psibhc^{m+m'-1}(X)$ with principal symbol 
	\[
	\bsymbol(\tfrac{i}{h}[A,B]) = \{\bsymbol(A),\bsymbol(B)\} = \bhamvf_{\bsymbol(A)}\bsymbol(B)
	\]
where the Poisson bracket is with respect to the usual symplectic form on $T^*X^\circ = \bT^*_{X^\circ} X$ extended by continuity to $\bT^*X$, which also defines the b-Hamilton vector $\bhamvf_f$. \label{it:commutator}

\blue{In canonical coordinates given by \eqref{eq:bcotangentsection}, the
  symplectic form is
  $$
\omega= \frac{d\sigma \wedge dx}{x}+ d\eta \wedge dy
$$
while the Hamilton vector field of $f$ is
$$
\bhamvf_f = x (\pa_\sigma f) \pa_x - x (\pa_x f)\pa_\sigma +(\pa_\eta f)
\cdot \pa_y -(\pa_y f ) \cdot \pa_\eta.
$$}

\item Each $A \in \Psibhc^0(X)$ extends to a bounded operator on $L^2(X)$, and moreover there exists $A' \in \Psibhc^{-\infty}(X)$ such that
	\[
	\|Au\|_{L^2} \leq 2\sup|\bsymbol(A)|\|u\|_{L^2} + \mathcal{O}(h^\infty)\|A'u\|_{L^2}
	\] 
	for each $u\in L^2(X)$. Here $\bsymbol(A)$ is any representative of the principal symbol in $\symbbc^0(\bT^*X)/h\symbbc^{-1}(\bT^*X)$. \label{it:L2}
\end{enumerate}

The subspace of operators with smooth coefficients, $\Psibh^m(X) \subset \Psibhc^m(X)$, satisfies \ref{it:mapping}, \ref{it:symbol}, \ref{it:quantization}, \ref{it:adjoint}, \ref{it:commutator}, \ref{it:L2} above, simply dropping the subscript $\mathrm{c}$ throughout. Moreover, $\Psibh^m(X)$ enjoys better mapping properties, namely each element of $\Psibh^m(X)$ maps $\CI(X) \rightarrow \CI(X)$ and $\CdmI(X) \rightarrow \CdmI(X)$.

Suppose that $F \in I^{[-1-\alpha]}(\RR^n_+;\partial \RR^n_+)$ has compact support, where $\alpha > 0$. Then $F$ is continuous, smooth away from the boundary, and after a semiclassical rescaling the Schwartz kernel of multiplication by $F$ is 
\begin{equation} \label{eq:Fkernel}
\delta(x-\tilde x)\delta(y-\tilde y)F(x,y) = (2\pi h)^{-n}\int e^{\sfrac i h ((x-\tilde x) \sigma + \left<y-\tilde y, \eta\right>)} F(x,y) \, d\sigma d\eta.
\end{equation}
We can always insert a cutoff $\phi(x/\tilde{x})$ as in \eqref{eq:localquantization}, since the kernel is supported by the diagonal. In particular, \eqref{eq:Fkernel} can be written in the form \eqref{eq:localquantization}. The reason for introducing the algebra with conormal coeffients is that when viewed as a symbol (independent of $\sigma,\eta$),
\[
F \in \symbbch^0(\bT^*\RR^n_+),
\] 
namely multiplication by $F$ is in $\Psibhc^0(X)$ when $\alpha > 0$ (but not $\Psibh^0(X)$).

\subsection{Interaction with differential operators} We will also need to consider the interaction between $\Psibh(X)$ and the algebra of semiclassical differential operators $\Diffh(X)$, which of course is not a subalgebra of $\Psibh(X)$. The material in this section is not relevant for the class of conormal coefficient operators $\Psibhc(X)$.

The key consideration in what follows is the indicial operator family of $A \in \Psibh(X)$, defined for $\sigma \in \CC$ and $v \in \CI(\partial X)$ by
\[
\indicial(A)(\sigma)v = x^{-i\sigma}A (x^{i\sigma}u)|_{\partial X},
\]
where $u \in \CI(X)$ is an arbitrary extension of $v$; here $x$ is a
fixed, global boundary defining function. Thus $\indicial(A) = 0$ for
$A \in \Psibh^m(X)$ precisely when $A \in x\Psibh^m(X)$. Furthermore,
the indicial operator map is an algebra homomorphism \blue{to
  $\sigma$-dependent families of semiclassical pseudodifferential
  operators on $\pa X:$}
\[
\indicial(AB)(\sigma) = \indicial(A)(\sigma) \circ \indicial(B)(\sigma).
\]
Observe that $\indicial(hx D_x)(\sigma)$ is simply multiplication by $\sigma$, and $\indicial(x)(\sigma)$ vanishes identically. 

Assume that $A \in \Psibh^m(X)$ has compact support in a boundary coordinate patch $\mathcal{U}\subset X$, so that $(hD_x)A$ is a well defined operator. Applying $\indicial$, it follows that $[hxD_x,A]\in  xh \Psibh^{m}(X)$ and $[x,A] \in xh \Psibh^{m-1}(X)$. Therefore,
\[
(hD_x) A = x^{-1}[hxD_x,A] + x^{-1}Ax(hD_x). 
\]
This can be rephrased as in the following lemma:

\begin{lem} \label{lem:AhDxswap}
	Given $A \in \Psibh^m(X)$ with compact support in $\mathcal{U}$, there exist $A', \, A'' \in \Psibh^m(X)$ with compact support in $\mathcal{U}$ such that 
	\begin{equation} \label{eq:swap}
	(hD_x)A - A'(hD_x) = hA'',
	\end{equation}
	where $A' = x^{-1}Ax$ and $A'' = x^{-1}[hxD_x,A]$.
\end{lem}

Lemma \ref{lem:AhDxswap} allows us to give a reasonable definition of differential operators with b-pseudodifferential coefficients:

\begin{defi}
Let $\Diff_h^k \Psibh^m(X)$ denote the vector space of locally finite sums of the form $\sum P_j A_j$, where $P_j \in \Diffh^k(X)$ and $A_j \in \Psibh^m(X)$. 
\end{defi}

Using Lemma \ref{lem:AhDxswap}, it can shown that any $\sum P_j A_j \in \Diff^k_h \Psibh^m(X)$ can also be written in the form $\sum A'_j P'_j$, where $A'_j \in \Psibh^m(X)$ and $P'_j \in \Diff_h^k(X)$.

\blue{One can moreover show that the differential-b-pseudodifferential
operators form a graded algebra in the following sense.}
\begin{lem}[cf. {\cite[Lemma 2.5]{vasy2008propagation}}]
	If $B_1\in \Diff_h^{k_1} \Psibh^{m_1}(X)$ and $B_2\in \Diff_h^{k_2} \Psibh^{m_2}(X)$, then the composition satisfies \[
	B_1 B_2 \in \Diff_h^{k_1+k_2} \Psibh^{m_1+m_2}(X).
	\] 
Furthermore,
	\[
	[B_1,B_2] \in h\Diff_h^{k_1+k_2} \Psibh^{m_1+m_2-1}(X).
	\]
\end{lem}

We also have the following fundamental commutation result:
\begin{lem}[cf. {\cite[Lemma 2.8]{vasy2008propagation}}] \label{lem:hDxcommutator}
	If $A \in \Psibh^m(X)$ has compact support in a boundary coordinate patch $\mathcal{U}$, then there exist $A_1 \in \Psibh^m(X)$ and $A_0 \in \Psibh^{m-1}(X)$ satisfying 
	\begin{equation} \label{eq:D_xcommutator}
	i[hD_x,A] = hA_1 + hA_0 (hD_x).
	\end{equation}
	Here $\bsymbol(A_0) = \pa_{\sigma} a$ and $\bsymbol(A_1)= \pa_x a$.
\end{lem}
\begin{proof}
	The identity \eqref{eq:D_xcommutator} follows from \eqref{eq:swap}, since $A' - A = x^{-1}[A,x] \in h \Psibh^{m-1}(X)$. The computation of the principal symbol follows by continuity from $T^*X^\circ$ as in \cite[Lemma 2.8]{vasy2008propagation}
\end{proof}

For the next result we fix a Riemannian metric on $X$ with respect to which all adjoints are taken. In particular, $(hD_x)^* = hD_x + h\Diff^0_h(X)$.

\begin{lem} \label{lem:Dx2Acommutator}
	Let $A \in \Psibh^m(X)$  have compact support in $\mathcal{U}$, and suppose that $a = \bsymbol(A)$ is real valued. Then there exist 
	\[
	B_0 \in \Psibh^{m-1}(X), \quad B_1 \in \Psibh^m(X)
	\] 
	with $\bsymbol(B_0) = 2\partial_{\sigma} a$ and $\bsymbol(B_1) = 2\partial_xa$, such that
	\[
	(i/h)[(hD_x)^*hD_x,A] = (hD_x)^*B_0(hD_x) + (hD_x)^*B_1 + hR,
	\]
	where $R \in \Diff^1_h \Psibh^{m-1}(X)$.
\end{lem}
\begin{proof}
	First, compute
	\begin{align*}
	i[(hD_x)^*hD_x, A] &= i(hD_x)^* [hD_x, A] -i[hD_x, A^*]^* (h D_x)\\ &= h (h D_x)^* (A_0+A_0^*) (hD_x)+h \left( (hD_x)^* A_1 + A_1^*(hD_x) \right),
	\end{align*}
	modulo $h \Diff^1_h \Psibh^{m-1}(X)$, where according to Lemma \ref{lem:hDxcommutator},
	\[
	\bsymbol(A_0) = \partial_{\sigma}a, \quad \bsymbol(A_1) = \partial_xa.
	\] 
	Here we used that $A = A^* + h\Psibh^{m-1}(X)$. In particular, $\bsymbol(A_0+A_0^*) = 2 \partial_{\sigma} a$. We then write
	\[
	A_1^*(hD_x) = (hD_x)^*A_1 + h\Diff_h^1 \Psibh(X)
	\]
	according to Lemma \ref{lem:hDxcommutator}. Therefore,
	\[
	(i/h)[(hD_x^*)hD_x,A] = (hD_x)^*B_0(hD_x) + (hD_x)^*B_1 +  h\Diff^1_h \Psibh(X),
	\]
 with $B_0 = A_0 + A_0^*$ and $B_1 = A_1$. 
\end{proof}

\subsection{Wavefront set and ellipticity}
In this section $X$ continues to denote a smooth manifold with
boundary. There is an operator wavefront set for elements of
$\Psibhc(X)$, which is naturally a subset of the fiber-radial
compactification $\OL{\bT^*}X$. As usual, $\WFb(A)$ can be defined
locally as the essential support of the total symbol $a$ of
$A \in \Psibhc^m(X)$. Here the notion of essential support takes into
account the conormal behavior of $a$: $q_0 \notin \esssupp(a)$ if
there is a neighborhood of $q_0$ in $\OL{\bT^*}X$ where $a$ lies
in
$h^\infty \symbbc^{-\infty}(\bT^*X)$. If $a \in \symbb^m(\bT^*X)$, this
automatically implies that a is locally in $h^\infty \symbb^{-\infty}(\bT^*X)$
near $q$.  The operator wavefront set satisfies the usual relations
\begin{equation} \label{eq:WFcalculus}
\begin{gathered}
\WFb(AB) \subset \WFb(A) \cap \WFb(B), \\
 \WFb(A+B) \subset \WFb(A) \cup \WFb(B).\\
\end{gathered}
\end{equation} 
We write $\Psibhc^\COMP(X)$ for the subalgebra of operators whose wavefront sets are a compact subset of $\bT^*X \subset \OL{\bT^*}X$, and similarly for $\Psibh^\COMP(X)$.

Ellipticity is also defined as usual. For instance, fix a norm
$|\cdot|$ on the fibers on $\bT^*X$, and then set $\left<\zeta \right>
= (1+ |\zeta|^2)^{1/2}$. We say that $A \in \Psibhc^m(X)$ is elliptic
at $q_0 \in \OL{\bT^*} X$ if for some $h_0>0$
\[
\left<\zeta\right>^{-m}|\bsymbol(A)(z,\zeta)| > 0
\]
for $h \in (0, h_0)$ in a neighborhood of $q_0 = (z_0,\zeta_0)$. The set of elliptic points is denoted $\ELLb(A)$. 
The standard symbolic procedure for elliptic symbols allows one to construct microlocal elliptic parametrices: if $A\in \Psibhc^s(X)$ and $B \in \Psibhc^m(X)$ satisfy $\WFb(A) \subset \ELLb(B)$, then there is $Q \in \Psibhc^{s-m}(X)$ such that 
\begin{equation} \label{eq:parametrix}
A -QB \in h^\infty\Psibhc^{-\infty}(X), \quad A-BQ \in h^\infty \Psibhc^{-\infty}(X).
\end{equation}
Of course if $A,B \in \Psibh(X)$, then both $Q$ and the residual terms in \eqref{eq:parametrix} can be chosen in $\Psibh(X)$.

A simple adaptation of \cite[Lemmas 3.2, 3.4]{vasy2008propagation} shows that each $A \in \Psibh^0(X)$ defines a uniformly bounded map
\begin{equation}\label{sobolevbounded}
A : \OL{H}^1_h(X) \rightarrow \OL{H}^1_h(X),
\end{equation}
where $\OL{H}^1_h(X)$ is the space of extendible distributions in the
sense of \cite[Appendix B.2]{hormander1994analysis}. The same is true
if $\OL{H}^1_h(X)$ is replaced by $\dot{H}^1_h(X)$, the 
\blue{space of distributions supported on $X,$ again in the sense of
\cite[Appendix B.2]{hormander1994analysis}.} By duality, $A$ is uniformly bounded on $\OL{H}^{-1}_h(X)$ and $\dot{H}^{-1}_h(X)$ as well.

\begin{lem} \label{lem:H1L2bbounded}
	Each $A \in \Psibhc^1(X)$ is uniformly bounded $A: \OL{H}^1_h(X) \rightarrow L^2(X)$.
\end{lem}
\begin{proof}
By a microlocal partition of unity we can assume that
        $\WFb(A)$ is contained in the elliptic set of some vector
        field $B$ (we can take $B = hW$ for some $W
        \in \Vb(X)$). Thus $A = QB + R$ for a parametrix $Q \in
        \Psibhc^0(X),$ where $R \in h^\infty \Psibhc^{-\infty}(X).$
Hence
	\[
	\| Au \|_{L^2} \leq C\| Bu \|_{L^2} + \mathcal{O}(h^\infty)\|u\|_{L^2} \leq C\| u \|_{\OL{H}^1_h}
	\]
	since $B \in \Diff^1_h(X)$.
\end{proof}

It will also be convenient to have a wavefront set for operators 
\[
A \in \Diff_h^k \Psibh^m(X) + \Psibhc^l(X).
\]
For this, we define
\[
\WFb^k(A)^\complement = \bigcup \, \{\ELLb(B): B \in \Psibh^0(X) \text{ and } BA \in h^\infty \Diff_h^k\Psibh^{-\infty}(X)  + h^\infty\Psibhc^{-\infty}(X)\}.
\]
If $ A \in \Psibhc^m(X)$, then $\WFb^k(A) = \WFb(A)$ for all $k \in \mathbb{N}$. 
Consider a concrete representation
\[
A = \sum P_j A_j \in \Diff^k_h \Psibh^m(X)
\] 
where $P_j \in \Diffh^k(X).$ 
In that case, if $\WFb(A_j) \subset U$ for some $U$, then $\WFb^k(A) \subset U$ as well. In fact, the only reason we choose to introduce $\WFb^k(A)$ is to bound certain quadratic forms. For this, we use the following observation: if $F \in \Psibh(X)$ satisfies $\WFb(F) \cap \WFb^k(A) = \emptyset$ with $A$ as above, then $FA \in h^\infty \Diff_h^k\Psibh^{-\infty}(X)$.

\begin{lem} \label{lem:quadraticformWF}
If $A \in \Diff_h^2 \Psibh^0(X)$ and $G \in \Psibh^0(X)$ satisfy $\WFb^2(A) \subset \ELLb(G)$, then
\begin{equation} \label{eq:WFdbestimate}
\lvert \left< Au,u\right>\rvert \leq C \| Gu\|_{H^1_h}^2 + \mathcal{O}(h^\infty)\|u\|^2_{H^1_h}
\end{equation}
for each $ u\in H^1_h(X)$, where the left hand side of
\eqref{eq:WFdbestimate} is the pairing of $Au \in
H^{-1}_h(X)$ with $u \in H^1_h(X)$.
\end{lem}
\begin{proof}
	Choose $B \in \Psibh^0(X)$ such that 
	\[
	\WFb(B) \subset \ELLb(G), \quad \WFb(1-B) \cap \WFb^2(A) = \emptyset.
	\]
	Therefore $A = BA + h^\infty \Diff^2_h\Psibh^{-\infty}(X)$. We can then choose a decomposition
	\[
	BA = \sum_{i,j} B Q_j Q_j' A_{ij} + h^\infty \Diff^2_h\Psibh^{-\infty}(X),
	\]
	where $Q_{i},Q_j' \in \Diff^1_h(X)$, and $A_{ij} \in \Psibh^0(X)$ satisfies $\WFb(A_{ij}) \subset \ELLb(G)$. Therefore 
	\[
	\lvert\left< Au,u\right>\rvert \leq \sum_{ij}\lvert\left<Q_j' A_{ij}u,Q_j^* B^*u\right>\rvert + \mathcal{O}(h^\infty)\|u\|_{H^1_h}^2 \leq C \| Gu\|_{H^1_h}^2 + \mathcal{O}(h^\infty)\|u\|^2_{H^1_h}
	\]
	as desired.
\end{proof}

\subsection{b-Calculus relative to an interior hypersurface}
In this section we depart from the setting of manifolds with boundary,
and instead consider a boundaryless manifold $X$ with a distinguished
hypersurface $Y\subset X$. For simplicity of exposition, we will work under the geometric assumption
that $Y$ is oriented, and that $Y$ divides $X$ into two manifolds with
boundaries,
\[
X = X_+ \cup X_-,
\] 
each of which satisfies $Y = \partial X_\pm$; the orientation is
chosen so that $X_+$ is the positive side.   In fact, all of our uses
of this calculus will be local near a single point in $Y,$ so
neither the hypothesis of orientation nor that of bounding two components
plays any role here: both are always true locally.

The space $\Psibh^m(X,Y)$ of b-pseudodifferential operators (or
$\Psibhc^m(X,Y)$, with conormal coefficients) relative to $Y$ is
defined in analogy with boundary case discussed in Section
\ref{subsect:b-pseudos}. For instance, to define residual operators
$h^\infty \Psibhc^{-\infty}(X,Y)$, the stretched product $X^2_\B$ is replaced
by the blow-up $[X^2;Y^2]$. The condition of vanishing to infinite
order at the side faces is then replaced by requiring the kernel to be
supported on the lift of $X_+^2 \cup X_-^2$.

In the case of smooth coefficients, we must impose an additional condition to ensure that the residual operators preserve $H^1_h(X)$. If $R \in h^\infty \Psibh^{-\infty}(X,Y)$, then by restriction $R$ defines two operators $R_\pm \in h^\infty \Psibh^{-\infty}(X_\pm)$, and the action of $R$ on $\CI(X)$ is given by
\[
R = e_+ R_+ r_+ + e_- R_- r_-,
\]
where $r_\pm : \CI(X) \rightarrow \CI(X_\pm)$ are the restriction maps, and $e_\pm$ is extension by zero from $X_\pm$ to $X$. A priori $R$ does not preserve $\CI(X)$. On the other hand, if we further require that the normal operators $\widehat{N}(R_\pm)(0)$ agree along $Y$, then $R$ maps $\CI(X)$ into piecewise continuous functions with smooth restrictions to $X_\pm$; this implies that $R$ is uniformly bounded on $H^1_h(X)$ and $H^{-1}_h(X)$ by duality (cf.\ the discussion preceding \cite[Lemma 4.1]{de2014diffraction}). \emph{We thus always assume this matching condition for residual operators with smooth coefficients} (observe that this is meaningless for operators with conormal coefficients).

The symbol classes
$\symbbc^m(\bT^*(X,Y))$ and $\symbb^m(\bT^*(X,Y))$ are defined in the
obvious way, replacing the usual b-cotangent bundle by the relative
space $\bT^*(X,Y)$ discussed in Section
\ref{subsect:btangentcotangent}. The quantization procedure
\eqref{eq:localquantization} does not need modification, and hence
Definition \ref{defi:b-pseudosbylocalization} goes through
verbatim. In particular, if $a \in \symbb^m(\bT^*(X,Y))$ is a smooth b-symbol, then $\Opbh(a)$ automatically has matching normal operators.

Properties of $\Psibhc^m(X,Y)$ are largely analogous to those in the boundary case. If $X$ is compact then each $\Psibh^0(X,Y)$ is uniformly bounded on $H^s_h(X)$ for $s \in \{-1,0,1\}$, cf.\ \cite[Lemma 4.1]{de2014diffraction}. In the case of conormal coefficients, we still have uniform boundedness on $L^2(X)$.

 Similarly, we can define $\Diff_h^k \Psibh(X,Y)$ to consist of locally finite sums $\sum P_j A_j$, where $P_j \in \Diff_h^k(X)$ and $A_j \in \Psibh(X,Y)$.

Finally, we define the wavefront set of a family $u = u(h)$ which is $h$-tempered in $H^s_h(X)$. Here, we will only consider the cases $s \in \{-1,0,1\}$. We say that $q_0 \notin \WFb^{s,r}(u)$ if there exists $A \in \Psibh^0(X)$ which is elliptic at $q_0$ and
\[
\|Au \|_{H^s_h} \leq Ch^r.
\]
When $s = 0$ it suffices to test within the larger class of operators $A \in \Psibhc^0(X)$, and we also abbreviate $\WFb^{r}(u) = \WFb^{0,r}(u)$.
The action of b-pseudodifferential operators is then semiclassically pseudolocal in the sense that  
\[
\WFb^{s,r}(Au) \subset \WFb^{s,r}(u) \cap \WFb(A).
\]
In fact, the following result shows that for our purposes, the distinction between $\WFb^{1,r}(u)$ and $\WFb^{r}$ is irrelevant; the operator $P$ is as in Section \ref{subsect:statementofresults}.
\begin{lem} \label{lem:equivalentWF}
	If $u$ is $h$-tempered in $H^1_h(X)$, then 
	\[
	\WFb^{1,r}(u) = \WFb^{r}(u) \cup \WFb^{-1,r}(Pu).
	\]
\end{lem}
\begin{proof}
	The inclusion $\WFb^{r}(u) \cup \WFb^{-1,r}(Pu) \subset \WFb^{1,r}(u)$ is obvious. The converse inclusion follows directly from Lemma \ref{lem:L2H1}, proved in Section \ref{subsect:elliptic} below.
\end{proof}

\section{Bicharacteristics}
\label{sec:bichars}

 \subsection{The characteristic set}

We return to the setting of Section \ref{subsect:statementofresults}: $(X,g)$ is a smooth $n$ dimensional Riemannian manifold with a distinguished hypersurface $Y\subset X$, and 
\[
P = h^2 \Delta_g + V
\]
where $V \in I^{[-1-\alpha]}(Y)$ for some $\alpha > 0$. In particular, we can consider multiplication by $V$ as a b-pseudodifferential operator
\[
V \in \Psibhc^0(X,Y).
\] 
Since $Y$ is fixed, for ease of notation we write $\bT^*X$ instead of the more precise $\bT^*(X,Y)$.

Given a point $y_0 \in Y$, we can find a coordinate patch $\mathcal{U} \ni y_0$ equipped with geodesic normal coordinates $(x,y)$ with respect to $g$. In particular, $\mathcal{U} \cap Y = \{x=0\}$.  In these coordinates the metric is given by 
\[
g=dx^2 + k(x,y,dy),
\]
where $x\mapsto k(x,\cdot)$ is family of metrics on $Y$ depending smoothly on the parameter $x$. Therefore
\[
P = (hD_x)^*(hD_x) + h^2\Lap_k + V,
\]
where $(hD_x)^*$ is the adjoint of $hD_x$ with respect to the metric density.  If $(x,y,\xi,\eta)$ are the corresponding canonical coordinates on $T^*X$, then the principal symbol is given by
\[
p = \xi^2 + k^{ij}\eta_i \eta_j  + V.
\] 
We also set \begin{equation}\label{Ptilde}\tilde{P} = h^2\Delta_k + V\end{equation} with principal symbol \[\tilde p = k^{ij}\eta_i \eta_j + V.\] 
Denote the characteristic set of $P$ by $\Sigma = \{p=0\}\subset T^*X$. The compressed characteristic set is then defined by
\[
\cchare = \pi (\Sigma)\subset \bdotT^*X,
\]
where $\pi : T^*X \rightarrow \bdotT^*X$ is the usual map. We equip $\cchare$ with the subspace topology inherited as a subset of $\bT^*X$ (in particular, $\cchare$ is locally compact and metrizable). Note that $\chare$ is compact in the fiber variables: if $K \subset X$ is compact, then so is $\chare \cap T^*_K X$. In particular, the restriction of $\pi$ to $\chare$ is proper.

We decompose the fiber-radial compactification $\OL{\bdotT^*}X$ into the elliptic, hyperbolic, and glancing regions, denoted by $\ellip, \hyp, \gl$, respectively:
\begin{equation}\label{ellhypgl}
\begin{aligned}
\ellip &= \{q \in \OL{\bdotT^*}X : \pi^{-1}(q) \cap \chare = \emptyset\},\\
\gl &= \{q \in \bdotT X:  |\pi^{-1}(q) \cap \chare| = 1\},\\
\hyp &= \{q \in \bdotT X: |\pi^{-1}(q) \cap \chare| \geq 2\}.
\end{aligned}
\end{equation}
Here $|\cdot|$ refers to the cardinality of a set. Since the restriction of $\pi$ to $T^*(X\setminus Y)$ is \blue{$1-1$}, it is clear that $\hyp \subset \bT^*_Y X \cap \dot \chare$. Furthermore, if $T^*(X\setminus Y)$ is identified with its image under $\pi$, any point $q\in T^*(X\setminus Y)$ is either in $\ellip$ or $\gl$, depending on whether $q \notin \Sigma$ or $q\in\Sigma$, respectively. Over a normal coordinate patch $\mathcal{U}$, the glancing region is given by
\blue{\[
\gl \cap   {\bT^*_\mathcal{U} X} = \{x=0,\ \tilde p= 0\} \subset T^*Y
\subset \bT_Y^* X.
\]
Likewise $\hyp \cap   {\bT^*_\mathcal{U} X}$ consists of those points
$q \in T^*Y\subset \bT_Y^*X$ for which $\tilde p(q) < 0$.}

\subsection{Hamilton flow} \label{subsect:hamiltonflow} Formally, the Hamilton vector field of $p$ on $T^*X$ in normal coordinates is given by
\[
\hamvf_p = 2\xi \partial_x + 2k^{ij}\eta_j \partial_{y_i} - \left(\left(\partial_x k^{ij}\right)\eta_i\eta_j + \partial_x V \right)\partial_{\xi} - \left(\left(\partial_{y_i} k^{jk}\right)\eta_j\eta_k +  \partial_{y_i} V \right) \partial_{\eta_i},
\]
where Einstein summation is implied. This is  a smooth vector field away from $T^*_Y X$, but in general only possesses $\C_*^{\alpha -1}$ coefficients due to the $\partial_\xi$ component. Of course if $\alpha > 2$, then $\hamvf_p$ has $\C^1$ (hence Lipschitz continuous) components, where the existence and uniqueness of solutions to Hamilton's equations are classical. Under the assumption that $\alpha > 0$, we define integral curves in the following sense:

\begin{defi} \label{defi:integralcurve}
If $I \subset \RR$ is an interval, we say that an absolutely continuous map $\gamma : I \rightarrow T^*X$ is an integral curve of $\hamvf_p$ if
\begin{equation} \label{eq:caratheodory}
\frac{d}{ds}{\gamma}(s) = \hamvf_p(\gamma(s))
\end{equation}
for almost every $s \in I$.
\blue{Such a curve is called a bicharacteristic.}
\end{defi}

Implicit in this definition is that $\hamvf_p\circ \gamma$ itself has measurable, locally integrable components. For general $\alpha > 0$, there is no reason to expect existence, let alone uniqueness, of integral curves through an arbitrary point $q_0 \in T^*_Y X$. 

On the other hand, near a point $q_0 = (0,y_0,\xi_0,\eta_0)$ with $\xi_0 \neq 0$, we can convert \eqref{eq:caratheodory} into an equation to which the Carath\'eodory existence and uniqueness theorem applies. More generally, consider a vector field 
\[
F = \sum F_j \partial_{z_j}
\]
on an open set $D\subset \RR^m_z$ with arbitrary real coefficients. Generalizing Definition \ref{defi:integralcurve}, we say that an absolutely continuous map $\gamma : I \rightarrow D$ is an integral curve of $F$ if 
\begin{equation} \label{eq:measurableode}
\frac{d}{ds}\gamma(s) = F(\gamma(s))
\end{equation}
for almost every $s \in I$. The following lemma is a variation of \cite[Lemma 3.1]{de2014diffraction}; when applied to $F = \hamvf_p$, it allows us to treat the whole range of parameters $\alpha > 0$, whereas the given reference would only be valid for $\alpha >1$.

\begin{lem} \label{lem:caratheodory}
Let $z = (z_1,z') \in \RR \times \RR^{m-1}$, with a corresponding decomposition $F = (F_1,F') : \RR^m \rightarrow \RR \times \RR^{m-1}$. Assume that 
\[
D = J_{z_1} \times O_{z'},
\] 
where $J \subset \RR$ is an interval and $O \subset \RR^{m-1}$. Suppose that $F_1$ is continuous and nonvanishing, and $F'$ satisfies the following properties on $D$.
\begin{enumerate} \itemsep6pt
	\item $F'$ is measurable in $z_1$ for all $z'$, and continuous in $z'$ for almost every $z_1$.
	\item There exists $m \in L^1(J; \RR_+)$ such that $|F'(z)| \leq m(z_1)$.
	\item There exists $k \in L^1(J; \RR_+)$ such that  $|F'(z_1,x') - F'(z_1,y')| \leq k(z_1)|x' - y'|$.
\end{enumerate}
Given $z_0 \in D$, there exists $\varepsilon >0$ and a unique integral curve $\gamma: [-\varepsilon,\varepsilon] \rightarrow D$ such that $\gamma(0) = z_0$.

Furthermore, suppose that $F'$ is continuous. If $\delta > 0$ is sufficiently small and $|z-z_0| \leq \delta$, then there is  a unique integral curve 
\[
\gamma^{(z)}: [-\varepsilon,\varepsilon] \rightarrow D
\] 
satisfying $\gamma^{(z)}(0) = z$, and $\gamma^{(z)} \rightarrow \gamma^{(z_0)}$ uniformly on $[-\varepsilon,\varepsilon]$ as $z \rightarrow z_0$.
\end{lem}
\begin{proof}
To avoid notational confusion, we reserve 
\[
\pi_1 : \RR \times \RR^{m-1} \rightarrow \RR, \quad \pi' : \RR \times \RR^{m-1} \rightarrow \RR^{m-1}
\] 
for projections onto the first and second factors, respectively. Suppose that $\gamma$ is an integral curve of $F$. Since $F_1$ is continuous and nonvanishing, the map $s \mapsto (\pi_1\circ \gamma)(s)$ has an absolutely continuous (and in fact $\C^1$) inverse $S = S(t)$. Define the time dependent vector field $G = (G_1,G')$ by 
\[
G_1(t,s,z') = 1/F_1(t,z'), \quad G'(t,s,z) = F'(t,z')/F_1(t,z').
\]
Then the curve $\Gamma(t) = (S(t),(\pi' \circ \gamma)(S(t)))$
satisfies the equation
\begin{equation} \label{eq:reparametrizedcurve}
\frac{d}{dt} \Gamma(t) = G(t,\Gamma(t)).
\end{equation}
This process can be reversed as well, in the sense that from an
absolutely continuous solution $\Gamma(t)$ of
\eqref{eq:reparametrizedcurve} we can recover a solution $\gamma(s)$
of \eqref{eq:measurableode} by setting
\begin{equation} \label{eq:Gammatogamma}
\gamma(s) = (T(s),(\pi'\circ \Gamma)(T(s))),
\end{equation}
where $T = T(s)$ is the inverse of $t \mapsto (\pi_1 \circ \Gamma)(t)$.

The equation \eqref{eq:reparametrizedcurve} is well-posed in the sense of Carath\'eodory  \cite[Theorems 1.1]{coddington1955theory}. Thus, given $(z_1,z') \in D$, there exists $\varepsilon_0 >0$ and a unique integral curve
\[
\Gamma : [z_1-\varepsilon_0, z_1 + \varepsilon_0] \rightarrow D
\]
such that $\Gamma(z_0) = (0,z')$. Passing to a curve $\gamma$ as in \eqref{eq:Gammatogamma}, we obtain a unique integral curve of $F$ satisfying $\gamma(0) = (z_1,z')$ on a suitable interval $[-\varepsilon,\varepsilon]$.

If $F'$ is continuous in its arguments, then solutions to
\eqref{eq:measurableode} (which are unique by the argument above)
depend continuously on the initial data \cite[Theorem
4.2]{coddington1955theory}, which implies the second
point. \end{proof}

Lemma \ref{lem:caratheodory} applies directly to the equation \eqref{eq:caratheodory} in a neighborhood of the hyperbolic region. 

\begin{lem} \label{lem:bicharacteristicsexist} Let $\alpha>0.$ Given
  $ \varpi_0\in \pi^{-1}(\hyp)$, there exists $\varepsilon>0$ and a
  unique integral curve
  $\gamma: (-\varepsilon,\varepsilon) \rightarrow \chare$ of $\hamvf_p$
  such that $\gamma(0) = \varpi_0$. Furthermore, if $\alpha > 1$, then
  the flow
	\[
	(s,\varpi) \mapsto \exp(sH_p)(\varpi)
	\]
	exists and is continuous in a neighborhood of $(0,\varpi_0)$
\end{lem}
\begin{proof}
Apply Lemma \ref{lem:caratheodory} to $F = \hamvf_p$ with the splitting of variables $z_1 = x$ and $z' = (y,\xi,\eta)$. Since $F_1 = 2\xi$, it is continuous and nonvanishing in a small neighborhood of $\tilde q_0$. The hypotheses on the remaining components of $F$ follows from Lemma \ref{lem:conormalintegrable} and \eqref{eq:L1lipschitz}. It remains to shows that $\gamma((-\varepsilon,\varepsilon)) \subset \Sigma$. If $x\neq 0$, then $\hamvf_p p =0$. On the other hand, 
\[
x(\gamma(s)) \neq 0 \text{ for } s \in (-\varepsilon,\varepsilon)\setminus 0,
\]
since, \blue{by our assumption that $\varpi_0 \in \pi^{-1}(\hyp),$} $F_1 = \hamvf_p x\neq 0$. Thus $p\circ \gamma$ is locally constant on $(-\varepsilon,\varepsilon)\setminus 0$, which completes the proof since $p\circ \gamma$ is continuous and $p(\gamma(0)) = 0$.

Now suppose that $\alpha > 1$, in which case the properties of the flow in $(s,{q})$ follow from the second part of Lemma \ref{lem:caratheodory}. \end{proof}

\begin{rem} \label{rem:nonunique}
For $\alpha<2,$ uniqueness of bicharacteristics can certainly fail \blue{in
the glancing region},
notwithstanding the special structure of Hamilton's equations.
Consider for instance the symbol
\[p= \big( \xi^2+\eta^2\big) -1-4 \lvert x\rvert^{3/2}\] on
$T^*\RR^2.$
The Hamilton vector field is
\[ 
\hamvf_p = 2 \xi \pa_x+ 6 (\sgn x) \lvert x\rvert^{1/2} \pa_\xi+2\eta \pa_y.\]
Clearly $(x=0, \xi=0, y=2s, \eta=1)$ is a null bicharacteristic.  But on the other hand, so is
  $(x= s_+^4,  \xi= 2 s_+^3,  y=2s, \eta=1).$
This example exhibits the possibility of bicharacteristics sticking to the
interface $Y$ for arbitrarily long times before detaching (cf.\
\cite{hartman1951problems} for further related examples of
non-uniqueness of geodesics).
\end{rem}

\subsection{Generalized broken bicharacteristics} \label{subsect:GBB}

We now define the generalized broken bicharacteristic flow as initially introduced by Melrose--Sj\"ostrand \cite{melrose1978singularities}; cf.\ \cite{lebeau1997propagation}, \cite{vasy2008propagation}.

\begin{defi}
	A function $f$ on $T^*X$ is $\pi$-invariant if $f(\varpi_1) = f(\varpi_2)$ whenever $\pi(\varpi_1) = \pi(\varpi_2)$.
\end{defi}

Any $\pi$-invariant function $f$ induces a function on $\bdotT^*X$, denoted by $f_\pi$.  A rich class of $\pi$-invariant functions are those of the form $\pi^*F$, where $F$ is a function on $\bdotT^* X$. In that case $F = (\pi^*F)_\pi$. If $f$ is $\pi$-invariant, then in local coordinates $(x,y,\xi,\eta)$ on $T^*X$,
\[
\xi \mapsto f(0,y,\xi,\eta)
\] 
is constant for every fixed $(y,\eta)$.

\begin{lem} \label{lem:Hpextends}
Let $\alpha>0.$  If $f \in \C^1(T^*X)$ is $\pi$-invariant, then $\hamvf_p f$ admits a continuous extension to $T^*X$. 
\end{lem}
\begin{proof}
	The only obstruction to proving the lemma is the term $-(\partial_x V) \partial_\xi f$. On the other hand, since $f$ is $\pi$-invariant, $\xi \mapsto f(0,y,\xi,\eta)$ is constant. Now $\partial_{\xi} f$ exists and vanishes along $T^*_Y X$, and hence $\partial_{\xi} f \in x\C^0(T^*X)$. Therefore $(\partial_x V) \partial_{\xi} f = (x\partial_x V)F$, where $F \in \C^0(T^*X)$, and this latter term vanishes along $T^*_Y X$ by Lemma \ref{lem:vanishes}.
\end{proof}

We now recall the definition of generalized broken bicharacteristics
as given in \cite{vasy2004propagation}.
\begin{defi} \label{defi:GBB}
	If $I \subset \RR$ is an interval, we say that a continuous map $\gamma : I \rightarrow \dot \chare$ is a generalized broken bicharacteristic ($\GBB$) if for each $s_0 \in I$ and $f \in \CI(T^*X)$ which is $\pi$-invariant,
	\begin{equation} \label{eq:GBBdefi}
	\liminf_{s\rightarrow s_0} \frac{ f_\pi(\gamma(s)) -
          f_\pi(\gamma(s_0))}{s-s_0} \geq \inf\{ (\hamvf_p f)(\varpi):
        \pi(\varpi) = \gamma(s_0),\ \varpi \in \Sigma\}
	\end{equation}
	If $s_0$ is an endpoint of $I$, the left hand side of \eqref{eq:GBBdefi} is meant in the one-sided sense.
      \end{defi}
Note that in the case at hand, the infimum on the right hand side is in fact a minimum over at
most two values.

This is of course the same as saying that both lower Dini derivatives
$D_\pm (f_\pi\circ \gamma)(s_0)$ are no smaller than the right hand side
of \eqref{eq:GBBdefi}. Definition \ref{defi:GBB} makes it clear that
$\GBB$s can be concatenated: if $\gamma : (s_0,s_1] \rightarrow
\cchare$ and $\gamma' : [s_1,s_2) \rightarrow \cchare$ are two $\GBB$s
with $\gamma(s_1) = \gamma'(s_1)$, then we can define a $\GBB$ on
$(s_0,s_2)$ that restricts to $\gamma$ on $(s_0,s_1]$ and $\gamma'$ on
$[s_1,s_2)$. This concise definition can be recast more concretely, as
in work of Lebeau \cite{lebeau1997propagation}:
\begin{lem} \label{lem:GBBequiv}
	If $I \subset \RR$ and $\gamma : I \rightarrow \dot \chare$ is a continuous map, then the following are equivalent.
	
	\begin{enumerate} \itemsep6pt 
		\item \label{it:GBB1} $\gamma$ is a $\GBB$ in the sense of Definition \ref{defi:GBB}.
		\item \label{it:GBB2} The following two conditions are satisfied for each $s_0 \in I$.
		\begin{enumerate}[itemsep = 6pt,topsep = 6pt] 
			\item If $q_0 = \gamma(s_0) \in \gl$, then for each $f \in \CI(T^*X)$ which is $\pi$-invariant,
		\begin{equation} \label{eq:glancingderivative}
		\frac{d}{ds}(f_\pi\circ \gamma)(s_0) = (\hamvf_p f)(\varpi_0),
		\end{equation}
		where $\varpi_0 \in \Sigma$ is the unique point for which $\pi(\varpi_0) = q_0$.
		\item If $q_0 = \gamma(s_0) \in \hyp$, then there exists $\varepsilon>0$ such that $0 < |s-s_0| < \varepsilon$ implies that $x(\gamma(s)) \neq 0$.
		\end{enumerate}
	\item \label{it:GBB3} For each $s_0 \in I$ there exist unique
          $\varpi_\pm \in \Sigma$ such that $\pi(\varpi_\pm) =
          \gamma(s_0)$ and for all $\pi$-invariant $f,$
	\begin{equation} \label{eq:onesidedderivative}
	\frac{d}{ds}(f_\pi \circ \gamma)_\pm(s_0) = (\hamvf_p f)(\varpi_\pm).
	\end{equation}
	\end{enumerate}

\end{lem} 
\begin{proof}
\eqref{it:GBB1} $\Longrightarrow$  \eqref{it:GBB2}: Let $\gamma : I \rightarrow \cchare$ be a $\GBB$ and $s_0 \in I$. First assume that $q_0 = \gamma(s_0) \in \gl$, in which case $\pi^{-1}(\{q_0\})$ consists of a single point $\varpi_0$. Applying \eqref{eq:GBBdefi} to $f$ and $-f$ shows that \eqref{eq:glancingderivative} holds. If $q_0 \in \hyp$ instead, apply \eqref{eq:GBBdefi} to the $\pi$-invariant function 
\[
f = x\xi = \pi^*\sigma.
\]
Then $\hamvf_p f = 2\xi^2$ along $\pi^{-1}(\hyp)$, so the infimum on the right hand side of \eqref{eq:GBBdefi} is positive for $q_0 \in \hyp$. On the other hand $f_\pi(\gamma(s_0)) = 0$, so $\sigma(\gamma(s)) \neq 0$ for small but nonzero values of $|s-s_0|$. Since $\gamma$ takes values in $\cchare$, this implies that $x(\gamma(s)) \neq 0$ as well.

\eqref{it:GBB2} $\Longrightarrow$ \eqref{it:GBB3}: By definition this implication is clear for $\gamma(s_0) \in \gl$, so we may assume that $\gamma(s_0) \in \hyp$, and that $s_0=0$. By hypothesis,
\[
x(\gamma(s))\neq 0 \text{ and } \gamma(s) \in \gl \text{ for } s \in [-\varepsilon,\varepsilon]\setminus 0,
\] 
thus we can view $\gamma : (0,\varepsilon] \rightarrow \chare$. In particular, $\xi(\gamma(s)) = \pm (\tilde p (\gamma(s)))^{1/2}$ for one choice of sign, and since $\tilde p$ is $\pi$-invariant, the limit $\xi_+ = \lim_{s\rightarrow 0^+}\xi(\gamma(s))$ exists. We then set 
\[
\varpi_+ = (0,y(\gamma(0)),\xi_+,\eta(\gamma(0))).
\]
Similarly, we can construct $\varpi_-$, and it is easy to check that \eqref{eq:onesidedderivative} holds. The choices of $\xi_\pm$ are unique, since they are recovered by applying \eqref{eq:onesidedderivative} to the $\pi$-invariant function $x$.

\eqref{it:GBB3} $\Longrightarrow$ \eqref{it:GBB1}: The condition \eqref{eq:onesidedderivative} shows that the left hand side of \eqref{eq:GBBdefi} is equal to the minimum of $(\hamvf_p f)(\varpi_\pm)$, which is clearly bigger than or equal to the infimum on the right hand side of \eqref{eq:GBBdefi}. 
\end{proof}

Suppose that $q_0 \in \hyp$. In view of Lemma \ref{lem:GBBequiv}, we can construct a backward $\GBB$ on some $(-\varepsilon,0]$ by solving \eqref{eq:caratheodory} with a choice of initial data in $\pi^{-1}(\{q_0\})$ and then projecting to $\cchare$ by $\pi$;  the same construction works in the forward direction. Conversely, any $\GBB$ through a point $q_0 \in \hyp$ is locally obtained by concatenating two solutions of \eqref{eq:caratheodory} projected to $\cchare$.

If $f$ is $\pi$-invariant, then $f_\pi \circ \gamma$ is Lipschitz on $I$. This follows from the fact that $f_\pi \circ \gamma$ has uniformly bounded one sided derivatives at each $s \in I$ by Lemma \ref{lem:GBBequiv}. Therefore
\[
|(f_\pi\circ\gamma)(s_1) - (f_\pi \circ \gamma)(s_2)| \leq \sup\{ |{\hamvf_p f( \varpi)|: \pi( \varpi) \in \gamma(I) \}} \cdot |s_1-s_2|,
\]
so $f_\pi \circ \gamma$ is in fact Lipschitz on $I$ with a constant independent of $\gamma$ provided we assume that $\gamma$ takes values in a fixed compact set $K$ (cf.\ \cite[Corollary 5.3]{vasy2008propagation} and \cite[Corollary 2]{lebeau1997propagation}).

Furthermore, suppose that $\mathcal{U}$ is an adapted coordinate patch, so that $\bT^*_{\mathcal{U}}X$ is equipped with the Euclidean distance 
induced by the coordinates $(x,y,\sigma,\eta)$. If $\gamma$ is any $\GBB$ with values in a fixed compact set $K \subset \bT^*_{\mathcal{U}}X$, we conclude that
\begin{equation} \label{eq:gammalipshitz}
|\gamma(s_1) - \gamma(s_2)| \leq L|s_1-s_2|
\end{equation}
for a constant $L > 0$ depending only on $K$. Using these observations, one deduces some important topological information about the set of all $\GBB$s. For a proof of the following proposition, the reader is referred to \cite[Proposition 5.4 and Corollary 5.6]{vasy2001propagation}.

\begin{prop} \label{prop:equicontinuous}
	Given a compact set $K \subset \cchare$ and a compact interval $[a,b] \subset \RR$, let 
	\[
	\mathcal{R} = \{\gamma: [a,b]\rightarrow K: \gamma \text{ is a } \GBB\}.
	\] 
	If $\mathcal{R} \neq \emptyset$, then $\mathcal{R}$ is compact with respect to the topology of uniform convergence. Furthermore, if  $\gamma : (a,b) \rightarrow \cchare$ is a $\GBB$, then $\gamma$ extends to a $\GBB$ on $[a,b]$.
\end{prop}

 For a closely related result, see Lemma \ref{lem:cauchypeano}.
\begin{lem} \label{lem:GBBgrowth}
	Let $U \subset \cchare$ be open and precompact, and $K \subset U$ be compact. There exists $\varepsilon_0 >0$ such that if $\gamma$ is any $\GBB$ defined on $[-\varepsilon,\varepsilon]$ with $\varepsilon \in (0,\varepsilon_0)$ and $\gamma(0) \in K$, then $\gamma([\varepsilon,\varepsilon]) \subset U$.
\end{lem}
\begin{proof}
	 It suffices to prove the result with $[0,\varepsilon]$ and $[-\varepsilon,0]$ replacing $[-\varepsilon,\varepsilon]$. We argue by contradiction. Fix $U' \supset K$ open with closure in $U$; we may thus assume that $d(U',\partial U) > c_0$ for some $c_0$. If the result does not hold, then we may choose a positive decreasing sequence $s_n \rightarrow 0$ and $\GBB$s $\gamma_n : [0,s_n] \rightarrow \overline U$ such that $\gamma_n(s_n) \in \partial U$ and $\gamma_n(s_{n+1}) \in U'$. In particular,
	 \[
	 d(\gamma_n(s_n),\gamma_n(s_{n+1})) > c_0
	 \]
	 uniformly in $n$. Let $q,q'$ denote subsequential limits of $\gamma_n(s_n), \gamma_n(s_{n+1})$, respectively; it follows that $q \neq q'$.	On the other hand, if $f \in \CI$ is $\pi$-invariant, then
	 \[
	 |f_\pi(\gamma_n(s_n)) - f_\pi(\gamma_n(s_{n+1}))| \leq L(|s_n| + |s_{n+1}|)
	 \]
	 where $L$ is independent of $n$. Since functions of the form $f_\pi$ separate points, this implies that $q = q'$, which is a contradiction.   
\end{proof}

\blue{We close this section with a brief description of the
  phenomenology allowed by the results above.  Fix $\alpha>1,$ so
  that solutions to Hamilton's equations exist.  A bicharacteristic curve arriving transversely at
  $Y$ (hence at a point in $\hyp$)
  can be continued in just one way across the interface as a
  bicharacteristic curve.  By contrast, the continuous trajectory in $X$
  obtained by flipping the sign of the normal momentum at the moment
  of impact is also the image of a GBB; these two curves are the only
  possible continuations of the incident bicharacteristic as a GBB,
  with the latter being ``diffractive'' in the heuristic terminology
  of the introduction.  A bicharacteristic arriving tangent
  to $Y,$ hence in $\gl,$ may, if $\alpha<2,$ stick to $Y$ thereafter (and possibly
  re-release in a tangent direction at some later point).   If
  $\alpha>2,$ uniqueness of bicharacteristics rules out this sticking
  at a point of simple tangency: the bicharacteristic brushes past $Y$
  and continues on its way.  By contrast,  the sticking behavior is
  always possible for a GBB.}

\section{Propagation of singularities along $\GBB$s}
\label{sec:GBB}

Throughout this section we assume that $\alpha >0$. We continue to write $\bT^*X$ instead of $\bT^*(X,Y)$, and also abbreviate $\Psibh^m = \Psibh^m(X,Y)$. To simplify various statements, assume that $X$ is compact; as usual this is inessential. 

\subsection{The elliptic region}
\label{subsect:elliptic}

We will begin by studying the elliptic region. The main result here is the following:

\begin{prop} \label{prop:elliptic}
	If $A,G \in \Psibh^0$ satisfy $\WFb(A) \subset \ELLb(G)$ and $\WFb(A) \cap \cchare = \emptyset$, then 
\[
\| Au \|_{H^1_h} \leq C\|GPu \|_{H^{-1}_h} + \mathcal{O}(h^\infty)\|u \|_{H^1_h}
\]
for each $u \in H^1_h(X)$. 
\end{prop}

An immediate consequence of Proposition \ref{prop:elliptic} is microlocal b-elliptic regularity, in the semiclassical sense.

\begin{prop}
	If $u$ is $h$-tempered in $H^1_h(X)$, then $\WFb^{1,r}(u) \subset \WFb^{-1,r}(Pu) \cup \cchare$ for each $r \in \RR \cup \{+\infty\}$.
\end{prop}

Since this is just ordinary elliptic regularity away from $Y$, we will henceforth assume that all pseudodifferential operators have compact support in a normal coordinate chart $\mathcal{U}$. 
We begin by giving a simple microlocal estimate for the Dirichlet form associated with the operator $P$.

\begin{lem} \label{lem:greens}
	If $A, G \in \Psibh^0$ satisfy $\WFb(A) \subset \ELLb(G)$, then
	\begin{align*}
	\int_X h^2|dAu|_g^2 + V|Au|^2 \, dg &\leq C \varepsilon^{-1}\|GPu\|_{H^{-1}_h}^2 \\&+ \varepsilon \|Au\|^2_{H^1_h} +Ch\|Gu\|^2_{H^1_h} + \mathcal{O}(h^\infty)\| u \|^2_{H^1_h}
	\end{align*}
	for each $u \in H^1_h(X)$ and $\varepsilon > 0$.
\end{lem}
\begin{proof}
	By Green's formula, if $v \in H^1_h(X)$, then
	\[
	\int_X h^2|dv|^2_g + V|v|^2 \, dg = \left<Pv, v\right>,
	\]	
	where the right hand side is the pairing of $H^{-1}_h(X)$ with $H^1_h(X)$ induced by the volume density. Applying this to $v = Au \in H^1_h(X)$, it remains to estimate 
	\begin{equation} \label{eq:PAu}
	\left<PAu,Au\right> = \left<APu,Au\right> + \left<[P,A]u,Au\right>.
	\end{equation}
	The first term on the right hand side of \eqref{eq:PAu} is simply bounded by Cauchy--Schwarz,
	\begin{equation} \label{eq:ellipticAP}
	|\left<APu,Au\right>| \leq (1/4)\varepsilon^{-1}\| APu \|_{H^{-1}_h}^2 + \varepsilon\|Au\|_{H^1_h}^2.
	\end{equation}
	Since $\WFb(A) \subset \ELLb(G)$, we can use
        microlocal
        ellipticity to estimate
        $\|APu\|_{H^{-1}_h}$ by $\|GPu\|_{H^{-1}_h} +
        \mathcal{O}(h^\infty) \| P u\|_{H^{-1}_h}$ on the right hand
        side of \eqref{eq:ellipticAP}; we may of course further
        estimate $\| P u\|_{H^{-1}_h} \leq C\| u \|_{H^1_h}$. As for the commutator, $\WFb^2([P,A]) \subset \WFb(A)$, and therefore by Lemma \ref{lem:quadraticformWF},
\[
	|\left<[P,A]u,Au \right>| \leq Ch\| Gu \|^2_{H^1_h} + \mathcal{O}(h^\infty)\|u\|^2_{H^1_h}. 
	\]
	This completes the proof.
\end{proof}

Before proving Proposition \ref{prop:elliptic} we record a corollary of Lemma \ref{lem:greens} that will be important when studying the hyperbolic region. Since $V \in L^\infty(X)$, by choosing $\varepsilon > 0$ sufficiently small in Lemma \ref{lem:greens} we can estimate
\begin{equation} \label{eq:L2H1constant}
\| Au \|_{H^1_h} \leq C \| GPu\|_{H^{-1}_h} + Ch \|Gu\|_{H^1_h} +  C_0 \|Au \|_{L^2} + \mathcal{O}(h^\infty)\| u \|_{H^1_h},
\end{equation}
where crucially $C_0 > 0$ is independent of $A$. The remainder can also be improved, at the cost of losing control of $C_0$:

\begin{lem} \label{lem:L2H1}
If $A,G \in \Psibh^0$ satisfy $\WFb(A) \subset \ELLb(G)$, then
\[
\|Au\|_{H^1_h} \leq C\|GPu\|_{H^{-1}_h} + C\|Gu\|_{L^2}  + \mathcal{O}(h^\infty)\|u\|_{H^1_h}.
\]
for each $u \in H^1_h(X)$.
\end{lem}
\begin{proof}
The proof follows by inductively showing that for each $k \in
\mathbb{N}$ and $u \in H^1_h(X)$, \blue{and for every $A,G$ satisfying
  the hypotheses of the lemma,}
\begin{equation} \label{eq:ellipticinduction}
\|Au\|_{H^1_h} \leq   C\|GPu\|_{H^{-1}_h} + C\|Gu\|_{L^2} +  Ch^k\|Gu\|_{H^1_h} + \mathcal{O}(h^\infty)\|u\|^2_{H^1_h}.
\end{equation}
Now \eqref{eq:ellipticinduction} holds for $k=1$ by using \eqref{eq:L2H1constant} and then estimating 
\begin{equation}\label{ellreg4.1}
\|Au\|_{L^2} \leq C\|Gu\|_{L^2} + \mathcal{O}(h^\infty)\| u \|_{L^2}.
\end{equation}
In the inductive step, assume that \eqref{eq:ellipticinduction} holds
for $k=s$. \blue{Apply \eqref{eq:ellipticinduction}, replacing $G$
  with $A'$ satisfying $\WFb(A)\subset \ELLb(A')$ and $\WF(A') \subset
  \ELLb(G)$ to obtain
\begin{equation} \label{eq:ellipticinduction2}
\|Au\|_{H^1_h} \leq   C\|A'Pu\|_{H^{-1}_h} + C\|A'u\|_{L^2} +  Ch^s\|A'u\|_{H^1_h} + \mathcal{O}(h^\infty)\|u\|^2_{H^1_h}.
\end{equation}
Likewise, replacing $A$ with $A'$ in our inductive assumption gives
\begin{equation} \label{eq:ellipticinduction3}
\|A'u\|_{H^1_h} \leq   C\|GPu\|_{H^{-1}_h} + C\|Gu\|_{L^2} +  Ch^s\|Gu\|_{H^1_h} + \mathcal{O}(h^\infty)\|u\|^2_{H^1_h}.
\end{equation}
Then substitute \eqref{eq:ellipticinduction3} into the $\|A' u\|_{H^1_h}$ term on the
right hand side of \eqref{eq:ellipticinduction2}; the remaining $A'$
terms on the right are estimated by the corresponding terms with $G$
by elliptic regularity as in \eqref{ellreg4.1} (recall that the
b-calculus is bounded on $H^{\pm 1}_h$ as well as $L^2$); this completes the
inductive step.}
\end{proof}

Note that the complement of $\cchare$ within $\overline{\bT^*}X$ is
the union of $\overline{\bT^*X}\setminus \OL{\bdotT^*}X$ with
$\ellip$. We begin by studying regularity on the former of these sets.

\begin{lem} \label{lem:sigmanotzero}
If $A \in \Psibh^0$ has compact support in $\{|x| < \delta/\sqrt{2}\}$, where $\delta>0$ satisfies 
			\begin{equation} \label{eq:Vellipticbound}
|V| < \tfrac12 \delta^{-2}\sigma^2
\end{equation}
in a neighborhood of $\WF(A)$, and $G \in \Psibh^0$ satisfies $\ELLb(A) \subset \WFb(G)$, then
\[
\| Au \|_{H^1_h} \leq C\|GPu \|_{H^{-1}_h} + \mathcal{O}(h^\infty)\|u \|_{H^1_h}
\]
for each $u \in H^1_h(X)$. 
\end{lem}
\begin{proof}

  Since $A$ is assumed to have compact support in $\{|x| < \delta/\sqrt{2}\}$,
	\begin{equation} \label{eq:xDxbound}
	\int_X \delta^{-2} |(xhD_x) Au|^2 + V|Au|^2 \,dg \leq \int_X \tfrac{1}{2} h^2|dAu|^2 + V|Au|^2 \, dg.
	\end{equation}
	In view of \eqref{eq:Vellipticbound}, we can choose $B,F \in \Psibhc^1(X,Y)$, where $\WFb(A) \subset \ELLb(B)$, such that
	\[
	\WFb((\delta^{-2}(hxD_x)^*(hxD_x) + V) -  (B^* B + hF)) \cap \WFb(A) = \emptyset.
	\]
	Now integrate by parts in $x$ to write the left hand side of \eqref{eq:xDxbound} as
	\[
	\int_X \delta^{-2} |hxD_x Au|^2 + V|Au|^2 \,dg = \|BAu\|^2_{L^2} + h\left<FAu,Au\right> + \mathcal{O}(h^\infty)\|u \|^2_{L^2}.
	\]
	In particular, this implies that
	\begin{multline*}
	\int_X \tfrac{1}{2} h^2|dAu|^2 \, dg + \|BAu\|^2_{L^2} \leq \int_X h^2|dAu|^2 + V|Au|^2\, dg \\ + h\|FAu\|_{L^2}\|Au\|_{L^2} + \mathcal{O}(h^\infty)\|u \|^2_{L^2}.
	\end{multline*}
	Since $B$ is elliptic on $\WFb(A)$, the left hand side of the inequality above controls $\|Au\|^2_{H^1_h}$, whereas by Lemma \ref{lem:greens} the right hand side is controlled by
	\[
	C\varepsilon\|GPu\|^2_{H^{-1}_h} + (C
        \varepsilon^{-1}+Ch)\|Au\|^2_{H^1_h}+C h\|G u\|^2_{H^1_h} + \mathcal{O}(h^\infty)\|u\|^2_{H^1_h}.
	\]
	Here we used Lemma \ref{lem:H1L2bbounded} to bound the
        operator norm of $F \in \Psibhc^1(X,Y).$	Thus for $\varepsilon > 0$ sufficiently small we can absorb
        the second term on the right hand side into the left hand
        side.

        \blue{This establishes the result but with an extra term $C h\|G
        u\|^2_{H^1_h} $ on the right hand side.  We now eliminate this
      term iteratively, just as in the proof of Lemma~\ref{lem:L2H1}.}
\end{proof}

Lemma \ref{lem:sigmanotzero} will also prove useful later in Section \ref{subsect:truepositivecommutator}. The next step is to consider $A \in \Psibh^0$ with wavefront set in a neighborhood of $q_0 \in \ellip$.

\begin{lem} \label{lem:nearelliptic}
	Let $q_0 \in \ellip$. There exists $A \in \Psibh^0$ with $q_0 \in \ELLb(A)$, such that if $G \in \Psibh^0$ satisfies $\WFb(A) \subset \ELL(G)$, then
\[
\| Au \|_{H^1_h} \leq C\|GPu \|_{H^{-1}_h} + \mathcal{O}(h^\infty)\|u \|_{H^1_h}
\]
for each $u \in H^1_h(X)$. 	
\end{lem}
\begin{proof}
If $\WFb(A)$ is a sufficiently small neighborhood of $q_0$, then there exists $c_0>0$ such that
	\begin{equation} \label{eq:ellipticpositive}
	(1-c_0)k^{ij}\eta_i \eta_j +V > 0
	\end{equation}
	near $\WFb(A)$. As in the proof of Lemma \ref{lem:sigmanotzero}, we can choose $B,F \in \Psibhc^1(X,Y)$, where $\WFb(A) \subset \ELLb(B)$, such that
	\[
	\WFb((1-c_0)k^{ij}(hD_{y^i})(hD_{y^j}) + V ) -  (B^* B + hF)) \cap \WFb(A) = \emptyset.
	\] Integrating by parts in $y$, it follows that 
	\begin{multline*}
	\int_X |(hD_x)Au|^2 + c_0 k^{ij}(hD_{y^i} Au)(\overline{hD_{y^j}Au}) + |BAu|^2 \, dg \\ \leq C \int_X h^2 |dAu|^2 + V|Au|^2 \, dg + h\left<FAu,Au\right> + \mathcal{O}(h^\infty)\|u \|^2_{L^2}.
	\end{multline*}
	which completes the proof as above, since the left hand side controls a multiple of $\|Au \|_{H^1_h}^2$. 
\end{proof}

Proposition \ref{prop:elliptic} follows by combining Lemmas \ref{lem:sigmanotzero}, \ref{lem:nearelliptic} with a microlocal partition of unity argument.

\subsection{The hyperbolic region} \label{subsect:hyperbolicregion}
 Since $\hyp$ is  a compact subset of $\bT^*X$, it suffices to work with pseudodifferential operators that are both compactly supported in a normal coordinate patch $\mathcal{U}$ and compactly microlocalized. Let $q_0 \in \mathcal{H}$. If  $(x,y,\sigma,\eta)$ are local coordinates near $q_0$, then
\[
q_0 = (0,y_0,0,\eta_0),
\]
where $\tilde p(q_0)  < 0$. 

\begin{prop}\label{prop:qualitativehyp}
	Suppose that $u$ is $h$-tempered in $H^1_h(X)$ and $q_0 \notin \WFb^{-1,r+1}(Pu)$, where $r \in \RR \cup \{+\infty\}$. If $q_0$ has a neighborhood $U \subset \cchare$ such that 
	\[
	U\cap \WFb^{1,r}(u) \cap \{\sigma < 0\} = \emptyset,
	\] 
	then $q_0 \notin \WFb^{1,r}(u)$.
\end{prop}

Combined with b-elliptic regularity, this proposition implies that if
\blue{$$q_0 \in \WFb^{1,r}(u)\backslash
  \WFb^{-1,r+1}(Pu),$$ then $q_0$ is a limit point} of $\WFb^{1,r}(u) \cap T^*(X\setminus Y)$. This in turns suffices to prove propagation of singularities; see Section \ref{subsect:propagationofsingularities}. The proposition is a restatement of the following quantitative result.

\begin{prop} \label{prop:hyperbolic} If $G \in \Psibh^\COMP$ is elliptic at $q_0$, then there exist $Q, Q_1\in \Psibh^\COMP$, where
\begin{gather*}
\WFb(Q) \subset \ELLb(G) \text{ and } q_0 \in \ELLb(Q),\\
\WFb(Q_1) \subset \ELLb(G) \cap \{\sigma < 0\},
\end{gather*}
such that
\[
\| Qu \|_{H^1_h} \leq  C h^{-1}\| GPu\|_{H^{-1}_h} + C\| Q_1u \|_{H^1_h} + \mathcal{O}(h^\infty)\|u\|_{H^1_h},
\]
for each $u\in H^1_h(X)$. 
\end{prop}

Proposition \ref{prop:hyperbolic} holds verbatim if we replace $\sigma$ with $-\sigma$ (corresponding to propagation in the backwards direction). 
We prove Proposition \ref{prop:hyperbolic} by a positive commutator argument, closely following \cite[Section 6]{vasy2008propagation}. Define the functions
\[
\omega = |x|^2 + |y-y_0|^2 + |\eta - \eta_0|^2, \quad \phi = \sigma+\frac{1}{\beta^2\delta}\omega.
\]
Here the parameters $\delta,\beta \in (0,\infty)$ will be chosen
later;  $\delta$ will be chosen small, while in this
argument $\beta$ will ultimately be taken to be large.

Observe that $|W\phi| \leq C(1+\beta^{-2}\delta^{-1})\omega^{1/2}$,
where $W \in \{\partial_x, x\partial_{\sigma}, \partial_{y_i},
\partial_{\eta_i}\}$. In particular, if $f \in \CI(\bT^*X)$, and $U$
is a neighborhood of $q_0$ with compact closure in $\bT^*X$, then
using Lemma~\ref{lemma:holder} we find
\begin{equation} \label{eq:phiderivatives}
|\partial_x \phi| + |H_f \phi| \leq C_0(1+\beta^{-2}\delta^{-1})\omega^{1/2}
\end{equation}
on $U$, where $C_0 > 0$ does not depend on $\beta,\delta$. Choose cutoff functions $\chi_0, \chi_1$ with the following properties:
\begin{itemize} \itemsep6pt
	\item $\chi_0$ is supported in $[0, \infty)$, with $\chi_0(s)
	=\exp(-1/s)$ for $s>0$.
	\item $\chi_1$ is supported in $[0,\infty)$, with $\chi_1(s) =1$ for
	$s\geq 1$, and $\chi_1'\geq 0$.
\end{itemize}
Now set
\begin{equation} \label{eq:hypcommutant}
a=\chi_0(2-\phi/\delta) \chi_1(2+\sigma/\delta).
\end{equation}
For each fixed $\beta > 0$, the support of $a$ is controlled by the
parameter $\delta > 0$ as follows.

\begin{lem} \label{lem:asupport}
	Given a neighborhood $U\subset \bT^* X$ of $q_0 \in \hyp$ and $\beta>0$, there exists $\delta_0>0$ such that $\supp a \subset U$ for each $\delta \in (0,\delta_0)$.
\end{lem}
\begin{proof}
	Necessary conditions to lie in the support of $a$ are $\phi \leq 2\delta$ and  $-2\delta  \leq \sigma$. From the definition of $\phi$, 
	\[
	|\sigma| \leq 2\delta, \quad 0 \leq \omega \leq \beta^2\delta(2 \delta-\sigma) \leq 4\beta^2 \delta^2
	\] 
	on $\supp a$, i.e.,\
        \begin{equation}
          \label{suppa}
          \supp a \subset \{|\sigma| \leq 2\delta, \, \omega^{1/2} \leq 2\beta \delta \}.
        \end{equation}
        Finally, observe that any neighborhood of $U$ of $q_0$ contains a set of the form $\{|\sigma| \leq 2\delta, \, \omega^{1/2} \leq 2\beta \delta \}$ provided $\delta$ is sufficiently small.
\end{proof}

If $A \in \Psibh^\COMP$ has principal symbol $a$, the goal is to obtain negativity of the
commutator $(i/h)[P,A^*A]$. This cannot be done symbolically within
the b-calculus, since $P$ is merely an element of $\Diff^2_h$ (for
more motivational material, see \cite[Section
6]{vasy2008propagation}). Using the expression for $P$ and the
notation of Lemma \ref{lem:Dx2Acommutator} and \eqref{Ptilde},
\begin{align} \label{eq:hypcommutator}
(i/h)[P,A^*A] &= B_0(hD_x)^*(hD_x) + B_1(hD_x) + (i/h)[\tilde P,A^*A]  + h \Diff_h^2 \Psibh^\COMP                \notag \\ &= B_0P - B_0\tilde P + B_1(hD_x) +
                            (i/h)[\tilde{P},A^*A] + h \Diff_h^2 \Psibh^\COMP,
\end{align}
where $\bsymbol(B_0) = 2\partial_{\sigma}(a^2)$ and $\bsymbol(B_1) = 2\partial_x(a^2)$. The last term is a b-pseudodifferential operator (with conormal coefficients) with principal symbol $\bhamvf_{\tilde p}a^2$.

The symbols of the operators in \eqref{eq:hypcommutator} can be further decomposed, depending on whether the various derivatives fall onto $\chi_0$ or $\chi_1$ when $a^2$ is differentiated. Those terms differentiating $\chi_1$ give rise to an term error supported on $\{-2\delta \leq \sigma \leq -\delta, \, \omega^{1/2} \leq 2\beta\delta\}$, whereas derivatives of $\chi_0$ will yield positivity. To this end, define
\begin{equation}\label{bdef}
b = 2\delta^{-1/2} (\chi'_0 \chi_0)^{1/2} \chi_1, \quad B = \Op_h(b).
\end{equation}
Here we have suppressed the arguments of $\chi_0,\chi_1$ as in \eqref{eq:hypcommutant}. 

Next, fix a neighborhood $U_0$ of $q_0$ with compact closure in $\bT^*X$ such that $\tilde p <0$ near $U_0$. Thus we can choose $\tilde{B} \in \Psibhc^\COMP$ such that
\[
U_0 \subset \ELLb(\tilde B), \quad 	\WFb(\tilde B^*\tilde B + \tilde P) \cap U_0 = \emptyset. 
\]
The operators $A,B$ depend on $\delta,\beta$, whereas $\tilde B$ does not. Finally, fix $\alpha_0 \in (0,\alpha)$ and let $\theta = \min(1,\alpha_0) \in (0,1]$. According to Lemma \ref{lem:vanishes} and \eqref{eq:vanishingholder},
\[
|x\partial_x V| \leq C|x|^{\theta}
\]
on $\mathcal{U}$. We then have the following decomposition of $[P,A^*A]$:

\begin{lem} \label{lem:hyperboliccommutator}
Given $\beta >0$, there exists $\delta_0 >0$ such that for each $\delta \in (0,\delta_0)$,
\begin{equation} \label{eq:hyperboliccommutator}
(i/h)[P,A^*A] = B_0 P - B^*(\tilde B^*\tilde B + R_0 + (hD_x)^* R_1)B + E + hR,
\end{equation}
where $A,B, B_0,C$ are as above, and remaining operators in \eqref{eq:hyperboliccommutator} have the following properties:
\begin{itemize} \itemsep6pt 
	\item $R_0 \in \Psibhc^\COMP$ and $R_1 \in \Psibh^\COMP$ satisfy 
	\[
	|\bsymbol(R_i) | \leq C_1((\delta \beta)^{\theta} + \beta^{-1}),
\]
 where $C_1 > 0$ does not depend on $\beta,\delta$.
\item $E,R\in \Diff^2_h \Psibh^\COMP + \Psibhc^\COMP$, and $\WFb^2(E) \subset \{-2\delta \leq \sigma \leq -\delta,\, \omega^{1/2} \leq 2\beta\delta\}$. 

\end{itemize}
The b-wavefront sets of $R_0,R_1, R$ are contained in $ \{|\sigma| \leq 2\delta, \, \omega^{1/2} \leq 2\beta\delta\}$.
\end{lem}
\begin{proof}
		Throughout the proof, we will use the notation $E,R$ to denote any operators satisfying the hypotheses of the lemma; these may change from line to line. Fix a cutoff $\psi \in \CI(\bT^*X;[0,1])$ such that $\psi = 1$ near $\{|\sigma| \leq 2\delta,\, \omega^{1/2} \leq 2\beta\delta\}$ with support in $\{|\sigma| < 3\delta, \, \omega^{1/2} < 3\beta\delta\}$.
		
			\begin{inparaenum}
		\item As in Lemma \ref{lem:asupport}, given $\beta > 0$ we can choose $\delta_0 > 0$ so that $\WFb(B) \subset U_0$ for $\delta \in (0,\delta_0)$; without loss we can assume that $\delta\beta \leq 1$. On the other hand, 
		\begin{align} \label{eq:aderivative}
		\bsymbol(B_0) &= -4\delta^{-1} (\chi'_0 \chi_0) \chi_1^2  + 4\delta^{-1}\chi_0^2 (\chi_1' \chi_1) \notag \\&= -b^2 + e.
		\end{align}
		Since $e$ is supported in $\{-2\delta \leq \sigma \leq -\delta,\, \omega^{1/2} \leq 2\beta\delta\}$, if we denote its quantization by $E$, then
		\[
		-B_0 \tilde P = -B^*\tilde B^*\tilde B B + E + hR.
		\]
	Here the error $R$ arises since we have arranged equality at the level of principal symbols.
	
	\item 
	Next, consider the term $B_1(hD_x)$. Since $\bsymbol(B_1) = -(\partial_x \phi)\psi b^2$, we can write 
	\[
	B_1(hD_x) = B^*R_1 (hD_x)B + hR,
	\]
	where according to \eqref{eq:phiderivatives} we can bound $|\bsymbol(R_1)| \leq C_0(1+\beta^{-2}\delta^{-1})\omega^{1/2}$ on $U_0$ for some $C_0 > 0$ independent of $\beta,\delta$ (recall that $U_0$ is chosen in the paragraph preceding the lemma). Moreover, $\omega^{1/2} \leq 3\beta\delta$ on its support, so
	\[
	|\bsymbol(R_1)| \leq 3C_0(\delta \beta + \beta^{-1})
	\]
	as desired.
	
	\item We split up $(i/h)[\tilde P, A^*A] = (i/h)[h^2\Delta_k,A^*A] + (i/h)[V,A^*A]$. Temporarily writing $f = k^{ij}\eta_i \eta_j$, the first term has principal symbol 
	\[
	\bhamvf_fa^2 = -(\bhamvf_f \phi)\psi b^2 + e,
	\]
	where $\supp e \subset \{-2\delta < \sigma < -\delta,\ \omega^{1/2} \leq 2\beta\delta\}$.  As above, we can write
	\[
	(i/h)[h^2 \Delta_k, A^*A] = B^*R_0'B + E + hR, 
	\]
	where according to \eqref{eq:phiderivatives} we can bound
        $|\bsymbol(R'_0)| \leq 3C_0(\delta\beta + \beta^{-1})$.

        \item Finally, consider $(i/h)[V,A^*A]$ with principal symbol 
	\[
	\bhamvf_V a^2 = -(\bhamvf_V \phi)\psi b^2 + e.
	\]
	Now $\bhamvf_V = (x\partial_x V) \partial_{\sigma} + (\partial_{y_i}V) \partial_{\eta_i}$ (with Einstein summation), so when bounding $|\bhamvf_V \phi|$ we certainly have
	\[
	|(\partial_{y_i}V) \partial_{\eta_i} \phi| \leq C_0(1+\beta^{-2}\delta^{-1})\omega^{1/2}
	\]
	by \eqref{eq:phiderivatives}. 
	
	This does not hold when $\phi$ is differentiated in $\sigma$. Instead, we bound $|x\partial_x V| \leq C_0 '|x|^{\theta} \leq C_0'\omega^{\theta/2}$. Thus we can write
	\[
	(i/h)[V, A^*A] = B^*R_0''B + E + hR, 
	\]
	where $|\bsymbol(R_0'')| \leq 3C_0((\beta\delta ) + \beta^{-1}) + 3^\theta C_0'(\beta\delta)^\theta$ by the support properties of $\psi$. Letting $R_0 = R_0' + R_0''$ completes the proof of the lemma.
	\end{inparaenum}
\end{proof}

Given $u \in H^1(X)$, apply Lemma \ref{lem:hyperboliccommutator} to write
\begin{align*}
-(2/h)\Im \left<APu,Au\right>  &= (i/h)\left<[A^*A,P]u,u\right> \\  &= \|\tilde B Bu\|^2_{L^2} + \left<R_0Bu,Bu\right> + \left< R_1Bu, (hD_x)Bu\right>\\ &- \left< Eu,u\right> + h\left<Ru,u\right>  -\left<B_0 Pu,u\right>,
\end{align*}
noting that $A,B,\tilde B$ preserve $H^1_h(X)$ and $B_0$ preserves $H^{-1}_h(X)$ (these operators all have smooth coefficients).

 First, we use the ellipticity of $\tilde B$ on $\WFb(B)$ and \eqref{eq:L2H1constant} to estimate
\begin{equation} \label{eq:CBbound}
c_0\|Bu\|^2_{H^1_h} \leq  \|\tilde B Bu\|^2_{L^2} + C\|GPu\|_{H^{-1}_h}^2 + Ch\| Gu \|_{H^1_h}^2 + \mathcal{O}(h^\infty)\|u\|_{H^1_h}^2,
\end{equation}
where $c_0>0$ independent of $\beta,\delta$ so long as $\delta \in (0,\delta_0)$, and where $G$ is elliptic on $\WFb(B)$. We fix $\beta$ once and for all using the following lemma:

\begin{lem} \label{lem:hyperbolicerrorbound}
	Given $\varepsilon>0$, there exists $\beta>0$ and $\delta_1  \in (0,\delta_0)$ such that
	\[
	|\left<R_0Bu,Bu\right>| + |\left< R_1Bu, (hD_x)Bu\right>| \leq \varepsilon\| Bu\|_{H^1_h}^2 + \mathcal{O}(h^\infty)\|u\|^2_{H^1_h}.
	\]
	for each $\delta \in (0,\delta_1)$ and $u \in H^1_h(X)$. 
\end{lem}
\begin{proof}
We bound 
\begin{align*}
\|R_i v\|_{L^2} &\leq 2\sup |\bsymbol(R_i)| \| v\|_{L^2} + \mathcal{O}(h^\infty) \|v\|_{L^2} \\ 
&\leq 2C_1( (\delta \beta)^{\theta} + \beta^{-1} )\|v\|_{L^2} + \mathcal{O}(h^\infty) \| v\|_{L^2},
\end{align*}
where $C_1 > 0$ does not depend on $\beta,\delta$. It suffices to first fix $\beta > 0$ sufficiently large, and then take $\delta_1 \in (0,\delta_0)$ sufficiently small. Applying this to $v=Bu$, along with Cauchy--Schwarz, finishes the proof.
\end{proof} 

Now suppose that $G \in \Psibh^\COMP$ is elliptic on $\WFb(B)$, and
$Q_1 \in \Psibh^\COMP$ is elliptic on $\WFb(E)$ with $\WFb(Q_1)
\subset \ELLb(G)\cap \{\sigma<0\}$ as in the statement of
Proposition~\ref{prop:hyperbolic}. Apply Lemma
\ref{lem:hyperbolicerrorbound} by taking $\varepsilon = c_0/2$
\blue{(with $c_0$ defined by \eqref{eq:CBbound})}. Combined with \eqref{eq:CBbound},
\begin{multline*}
(c_0/2)\|Bu\|^2_{H^1_h} \leq (2/h)\lvert\left<APu,Au\right>\rvert + C\|GPu\|_{H^{-1}_h}^2 + Ch\| Gu \|_{H^1_h}^2  \\ +\lvert\left< Eu,u\right>\rvert + h\lvert\left<Ru,u\right>\rvert + \lvert\left<B_0Pu,u\right>\rvert + \mathcal{O}(h^\infty)\|u\|^2_{H^1_h}
\end{multline*}
for $\delta \in (0,\delta_1)$.  Using Cauchy--Schwarz on the $B_0$ term
and estimating the $E$ term by $Q_1$ using microlocal elliptic
regularity bounds the second line by
\begin{multline*}
\lvert\left< Eu,u\right>\rvert + h\lvert\left<Ru,u\right>\rvert + \lvert\left<B_0Pu,u\right>\rvert \\ \leq  Ch^{-1}\|GPu\|^2_{H^{-1}_h} + Ch\|Gu\|_{H^1_h}^2+ C\|Q_1u \|_{H^1_h}^2 + \mathcal{O}(h^\infty)\|u\|_{H^1_h}^2.
\end{multline*}
Since $\WFb(A) \subset \ELLb(G)$ as well, we can also estimate
\[
(2/h)\lvert\left<APu,Au\right>\rvert \leq C\varepsilon^{-1} h^{-2}\| GPu\|_{H^{-1}_h}^2 + C\varepsilon \|Au\|^2_{H^1_h}+ \mathcal{O}(h^\infty)\| u \|_{H^1_h}^2.
\]
Hence overall we obtain
\begin{align*}
(c_0/2)\|Bu\|^2_{H^1_h} &\leq C\varepsilon^{-1} h^{-2}\| GPu\|_{H^{-1}_h}^2 + Ch\|Gu\|_{H^1_h}^2 + C\|Q_1u \|_{H^1_h}^2 \\ &+ C\varepsilon \|Au\|^2_{H^1_h} + \mathcal{O}(h^\infty)\|u\|_{H^1_h}^2.
\end{align*}
By construction $\chi_0(s) = s^2 \chi_0'(s)$ for $s>0$, so  
\[
a = (2-\phi/\delta) (\chi_0'\chi_0)^{1/2}\chi_1 = \frac 12 \delta^{1/2}(2-\phi/\delta)b.
\] 
Thus we can write $A = FB + hF'$ for some $F,F' \in \Psibh^\COMP$. Choosing $\varepsilon>0$ sufficiently small gives the estimate
\[
\|Bu\|_{H^1_h} \leq C h^{-1}\|GPu\|_{H^{-1}_h} + C\| Q_1u \|_{H^1_h} + C h^{1/2} \|G u \|_{H^1_h} + \mathcal{O}(h^\infty)\| u \|_{H^1_h}.
\]
We now finish the proof of Proposition \ref{prop:hyperbolic}.

\begin{proof} [Proof of Proposition \ref{prop:hyperbolic}]
	Let $G$ be as in the statement of the proposition. Since $\ELLb(G)$ is open, choose $\delta_\star \in (0,\delta_1)$ such that \[
	\{|\sigma| \leq 2\delta_\star, \, \omega^{1/2} \leq 2\beta \delta_\star \} \subset \ELLb(G).
	\] 
	Recall that $\delta_1,\beta$ are fixed in Lemma
        \ref{lem:hyperbolicerrorbound}. Then, choose $Q_1 \in
        \Psibh^\COMP$ such that
	\[
	\{-2\delta_\star \leq \sigma \leq -\delta_\star, \, \omega^{1/2} \leq 2\beta \delta_\star \} \subset \ELLb(Q_1), \quad \WFb(Q_1) \subset \WFb(G).
	\]
	Take a sequence of operators $B_k \in \Psibh^\COMP$ corresponding to decreasing sequence of $\delta_k$ in $(\delta_\star/2,\delta_\star)$. Then $B_{k}$ is elliptic on $\WFb(B_{k+1})$, so
	\[
	\| B_{k+1} u\|_{H^1_h} \leq Ch^{-1} \|B_k Pu \|_{H^{-1}_h} + C\|Q_1 u\|_{H^1_h} + Ch^{1/2}\|B_{k} u \|_{H^1_h} + \mathcal{O}(h^\infty)\|u\|_{H^1_h}
	\]
	for each $k$. Fix $Q \in \Psibh^\COMP$, elliptic at $q_0$, such that each $B_k$ is elliptic on $\WFb(Q)$. By induction, we conclude that 
	\[
	\| Q u\|_{H^1_h} \leq Ch^{-1} \|G Pu \|_{H^{-1}_h} + C\|Q_1 u\|_{H^1_h} + Ch^{k/2}\|G u \|_{H^1_h} + \mathcal{O}(h^\infty)\|u\|_{H^1_h}
	\]
	for each $k \in \mathbb{N}$, which completes the proof.
\end{proof}

\subsection{The glancing region} \label{subsect:glancing}
As before, we assume that all b-pseudodifferential operators are supported in a fixed normal coordinate patch $\mathcal{U}$, and are compactly microlocalized.  
Before proceeding to the commutator argument, we need a variant of Lemma \ref{lem:L2H1}.

\begin{lem} \label{lem:Dxbound}
Given $\delta > 0$, let $U_\delta = \{q \in \bT^*_{\mathcal{U}}X: |\tilde p| < \delta\}$. If $A,G \in \Psibh^\COMP$ satisfy $\WFb(A) \subset \ELLb(G) \cap  U_\delta$,
then
\[
\int_X |hD_x Au|^2 \, dg \leq   C h^{-1}\|GPu \|^2_{H^{-1}_h} +Ch\|Gu\|^2_{H^1_h} + 2\delta \| Au\|^2_{L^2} + \mathcal{O}(h^\infty)\| u\|_{H^1_h}  
\] 
for each $u \in H^1_h(X)$.
\end{lem}
\begin{proof}
	Write $|hd v|_g^2 = |hD_x v|^2 +
        k^{ij}(hD_{y_i}v)(\overline{hD_{y_j}v})$; now let $v =Au$ and
        apply Lemma \ref{lem:greens} with $\varepsilon=h$ to see that
	\begin{multline*}
	\int_X |hD_x Au|^2 \, dg \leq -\int_X (\tilde PAu) \overline{Au} \, dg \\ +C h^{-1}\|GPu\|_{H^{-1}_h}^2 +Ch\|Gu\|^2_{H^1_h} + \mathcal{O}(h^\infty)\| u \|^2_{H^1_h}
	\end{multline*}
	after integrating by parts in $y.$  Choose $F \in \Psibhc^\COMP$ such that 
	\[
		\WFb(F + \tilde {P}) \cap \WFb(A) = \emptyset, \quad \WFb(F) \subset U_\delta
	\]
	One can always choose $F$ such that with $f = \bsymbol(F)$,
	\[
	\sup |f| \leq \delta.
	\]
Therefore we can bound
	\[
	\left<Fv,v\right> \leq 2\sup|f|\|v\|^2_{L^2} + \mathcal{O}(h^\infty)\|v\|^2_{L^2} \leq 2\delta  \|v\|_{L^2} + \mathcal{O}(h^\infty)\|v\|_{L^2}.
	\]
Applying this to $v= Au$ and using that $\WFb(A) \subset \ELLb(G)$, we find that
	\[
	\int_X |hD_x Au|^2 \, dg \leq Ch^{-1}\|GPu\|^2_{L^2} + 2\delta \|Au\|_{L^2}^2 + Ch\|Gu\|^2_{H^1_h} + \mathcal{O}(h^\infty)\| u \|_{H^1_h}
	\]
	for each $u \in H^1_h(X)$.
\end{proof}

Define $\tilde{p}_0 \in \C^\infty(T^*X)$ over the normal coordinate patch $\mathcal{U}$ by 
\[
\tilde{p}_0(x,y,\xi,\eta) = k^{ij}(0,y)\eta_i \eta_j + V(0,y). 
\]
Given $q_0 \in \gl \cap T^*Y$, let $\varpi_0$ denote the unique point in $\Sigma$ such that $\pi(\varpi_0) = q_0$. Recall that a $\GBB$ passing through $q_0$ at $s=s_0$ is characterized by the equality
\[
\frac{d}{ds}(f_\pi\circ \gamma)(s_0) = (\hamvf_p f)(\varpi_0)
\]
for each $f\in \CI(T^*X)$ which is $\pi$-invariant. On the other hand, since $\xi(\varpi_0) = 0$, it follows that 
\begin{equation} \label{eq:equalityatglancing}
(\hamvf_p f)(\varpi_0) = (\hamvf_{\tilde{p}_0}f)(\varpi_0).
\end{equation}
Via the local coordinates $(x,y,\sigma,\eta)$, we can also view $\tilde{p}_0$ as a function on $\bT^*_{\mathcal{U}}X$. With this identification, $\tilde p_0$ can be considered as a function on $\bT^*X$, and the flow $\exp(s\bhamvf_{\tilde p_0})$ on $\bT^*X$ makes sense.

As in Section \ref{subsect:hyperbolicregion}, choose $\alpha_0 \in (0,\alpha)$ and let $\theta = \min(1,\alpha_0) \in (0,1]$. Denote by $|\cdot|$ the Euclidean distance on $\bT^*_{\mathcal{U}}X$ in local coordinates, and write $\mathsf{B}(q_0,\varepsilon)$ for the corresponding ball of radius $\varepsilon>0$.

\begin{prop} \label{prop:glancing}
	Suppose that $u$ is $h$-tempered in $H^1_h(X)$ and $q_0 \notin \WFb^{-1,r+1}(Pu)$, where $r \in \RR \cup \{+\infty\}$. Let $K \subset \gl \cap T^*_{\mathcal{U}\cap Y} Y$ be compact. There exists $C_0, \delta_0>0$ such that for each $\delta \in (0,\delta_0)$ and $q_0 \in K$, if
\[
\mathsf{B}(\exp(-\delta \bhamvf_{\tilde p_0})(q_0),C_0\delta^{2/(2-\theta)}) \cap \WF^{1,r}(u) = \emptyset,
\]
then $q_0 \notin \WFb^{1,r}(u)$.	
\end{prop}

Following \cite[Section 7]{vasy2008propagation}, define the set 
\[
\mathsf{D}(q_0,\varepsilon) = \{q \in \bT^*X: |x(q)-x(q_0)| + |y(q)-y(q_0)| + |\eta(q) -\eta(q_0)| \leq \varepsilon\}.
\]
In order to prove Proposition \ref{prop:glancing}, it suffices to replace $\mathsf{B}$ with $\mathsf{D}$, possibly modifying $C_0$.
Indeed, $\WFb^{1,r}(u) \subset \cchare$, and on the compressed characteristic set $|\sigma| \leq C_1|x|$, where $C_1>0$ is uniform over compact subsets of $X$. Proposition \ref{prop:glancing} is then just a restatement of the following result:

\begin{prop}\label{prop:restated}
Let $K \subset \gl \cap T^*_{\mathcal{U}\cap Y} Y$ be compact. There exist $C_0,\delta_0>0$ such that the following property holds for each $\delta \in (0,\delta_0)$ and $q_0 \in K$. If $G \in \Psibh^\COMP$ is elliptic at $q_0$, then there exist $Q,Q_1 \in \Psibh^\COMP$, where
\begin{gather*}
\WFb(Q) \subset \ELLb(G) \text{ and } q_0 \in \ELLb(Q),\\
\WFb(Q_1) \subset \ELLb(G)\cap  \mathsf{D}(\exp(-\delta \bhamvf_{\tilde p_0})(q_0),C_0\delta^{2/(2-\theta)}) 
\end{gather*}
such that 
\[
\| Qu \|_{H^1_h} \leq  C h^{-1}\| GPu\|_{H^{-1}_h} + C\| Q_1u \|_{H^1_h} + \mathcal{O}(h^\infty)\|u\|_{H^1_h}
\]
for each $u\in H^1_h(X)$. 
\end{prop}

Just as with the hyperbolic estimate, Proposition
\ref{prop:hyperbolic}, we can also reverse the direction of
propagation here. Thus the same result holds verbatim if we replace
$\exp(-\delta \bhamvf_{\tilde p_0})(q_0)$ with
$\exp(\delta \bhamvf_{\tilde p_0})(q_0)$.

The rest of this section will be a proof of Proposition~\ref{prop:restated}.
View $\tilde p_0$ as a function on $T^*Y$, and thus
$\hamvf_{\tilde p_0}$ as a vector field on $T^*Y$. We may assume that
$d\tilde{p}_0(q_0) \neq 0$ here viewed as a covector on $T^*Y,$ as
otherwise the result to be proved is vacuous.  Then there are $2n-2$
functions $(\rho_0,\rho_1,\ldots,\rho_{2n-3})$ on $T^*Y$, whose
differentials are linearly independent at $\tilde q_0$, such that
\[
\quad (\hamvf_{\tilde{p}_0}\rho_0)(q_0) > 0, \quad (\hamvf_{\tilde p_0}\rho_j)(q_0) = 0, \quad \rho_1 = \tilde{p}_0.
\]
We also arrange that these functions all vanish at $q_0$. Since it
slightly simplifies matters, we can in fact arrange that $H_{\tilde p_0} = \partial_{\rho_0}$ near $q_0$ and thus 
\[
\hamvf_{\tilde p_0}\rho_0 = 1, \quad \hamvf_{\tilde p_0}\rho_j = 0 \text{ for } j=1,\ldots,2n-3
\]
identically. We extend $(\rho_0,\ldots,\rho_{2n-3})$ to functions on $\bT^*X$ by requiring them to independent of $(x,\sigma)$, so that $(x,\sigma,\rho_0,\ldots,\rho_{2n-3})$ are valid local coordinates on $\bT^*X$ near $q_0$. Now define
\[
\omega_0 = \sum_{j=1}^{2n-3} \rho_j^2, \quad \omega = \omega_0 + x^2. 
\]
In order to construct a commutant, let $\chi_0,\chi_1$ be as in Section \ref{subsect:hyperbolicregion}. Define 
\[
\phi = \rho_0 + \frac{1}{\beta^2\delta}\omega.
\]
We then set $A = \Op_h(a)$, where
\[
a = \chi_0(2-\phi/\delta)\chi_1(1+ (\rho_0+\delta)/(\beta\delta)).
\]
The difference compared with Section \ref{subsect:hyperbolicregion} is in the argument of $\chi_1$. Indeed, there will be an error term (the analogue of $E$ in Lemma \ref{lem:hyperboliccommutator}) with wavefront set contained in 
\[
\{ -\delta\beta - \delta \leq \rho_0 \leq -\delta, \ \omega^{1/2} \leq 2\beta\delta \}.
\]
If $C>0$ is sufficiently large, then this is certainly contained in the set
\[
 \mathsf{D}(\exp(-\delta \bhamvf_{\tilde p_0})(q_0),C\beta\delta) 
\]
and thus lies inside a set of the form
$\mathsf{D}(\exp(-\delta H_{\tilde
  p_0})(q_0),C_0\delta^{2/(2-\theta)}) $ if we choose $\beta=c  \delta^{\theta/(2-\theta)}$
(Note that this time, $\beta\in (0,\infty)$ will be taken to be small,
rather than large as in the hyperbolic propagation argument.)

We also need to consider the difference between $\bhamvf_{\tilde p}$ and $\bhamvf_{\tilde p_0}$ (now vector fields on $\bT^*X$). Here,
\begin{equation} \label{eq:tildeperror}
|\bhamvf_{\tilde p}\phi  - \bhamvf_{\tilde p_0} \phi| \leq M\left(1+ \beta^{-2}\delta^{-1}\omega^{1/2}\right)\omega^{\theta/2}
\end{equation}
locally, where $M > 0$ does not depend on $\beta,\delta$. 

\begin{rem} \label{rem:uniformomega} The construction of $\omega$ above is meant to localize along $\GBB$s through $q_0$. Using the same local coordinates $(\rho_0,\ldots,\rho_{2n-3})$, we could also localize at nearby points $q  \in \gl \cap T^*Y$ by setting
		\[
		\omega_0 = \sum_{j=1}^{2n-3} |\rho_j -\rho_j(q)|^2,
		\]
if $q$ is sufficiently close to $q_0$. If $\omega$ and $\phi$ are defined in the obvious way,
 then the constant $M > 0$ in \eqref{eq:tildeperror} can then be taken uniform for $q$ near $q_0$. As will be clear from the proof below, this implies uniformity of the constants $C_0,\delta_0$ in Proposition \ref{prop:glancing} in a neighborhood of $q_0$. Thus by compactness, we can simply assume that $K = \{q_0\}$.
\end{rem}

Now let $b= 2\delta^{-1/2} (\chi'_0 \chi_0)^{1/2} \chi_1$ and $B = \Op_h(B)$ as before, and write
\[
(i/h)[P,A^*A] = B_0(hD_x)^*(hD_x) + B_1 (hD_x) + (i/h)[\tilde P,A^*A] + h \Diff_h^2 \Psibh^\COMP,
\]
where $\bsymbol(B_0) = 2\partial_{\sigma}(a^2)$ and $\bsymbol(B_1) = 2\partial_x(a^2)$. We then have the following analogue of Lemma \ref{lem:hyperboliccommutator}.

\begin{lem} \label{lem:glancingcommutator}
There exists $\delta_0 >0$ such that for each $\delta \in (0,\delta_0)$ and $\beta \in (0,1) $,
\begin{equation} \label{eq:glancingcommutator}
(i/h)[P,A^*A] =  B^*(hD_x R_1 + R_0 - 1)B + E + hR,
\end{equation}
where $A,B$ are as above, and remaining operators in \eqref{eq:hyperboliccommutator} have the following properties:
\begin{itemize} \itemsep6pt 
	
	\item $R_0 \in \Psibhc$ and $R_1 \in \Psibh^\COMP$ satisfy
	\[
	|\bsymbol(R_0)| \leq C_1\delta^\theta  \beta^{\theta-1}, \quad |\bsymbol(R_1) | \leq C_1\beta^{-1}
	\]
	 where $C_1 > 0$ does not depend on $\beta,\delta$.
	\item $E,R\in \Diff^2_h \Psibh+ \Psibhc$, and 
	\[
	\WFb^2(E) \subset \{-\delta - \delta \beta \leq \rho_0 \leq -\delta,\, \omega \leq 2\beta \delta\}
	\]
\end{itemize}
The wavefront sets of $R_0,R_1, R$ are contained in $\{|\rho_0| \leq 2\delta, \, \omega \leq 2\beta \delta\}$
\end{lem}
\begin{proof}
As in the proof of Lemma \ref{lem:hyperboliccommutator}, we use the notation $E,R$ to denote any operators satisfying the hypotheses of the lemma; these may change from line to line. Fix a cutoff $\psi \in \CI(\bT^*X;[0,1])$ such that $\psi = 1$ near $\{|\rho_0| \leq 2\delta, \, \omega^{1/2} \leq 2\beta\delta\}$ with support in $\{|\rho_0| < 3\delta,\,\omega^{1/2} < 3\beta\delta\}$.
	
	\begin{inparaenum}
		\item First, following the notation of \eqref{eq:hypcommutator}, consider the term $B_0 (hD_x)^*(hD_x)$. Since $a$ is independent of $\sigma$, it follows that $\bsymbol(B_0) = 0$ and hence $B_0 \in h\Psibh^\COMP(X,Y)$. Thus $B_0 (hD_x)^*(hD_x)$ is part of the error $hR$.
	
		\item Now consider $B_1 (hD_x)$, where $\bsymbol(B_1)
                  = -(\partial_x \phi) \psi b^2$. Since $\partial_x
                  \phi = 2 \beta^{-2}\delta^{-1} x$ and $|x| \leq
                  \omega^{1/2} \leq 3\beta \delta$ on $\supp \psi$, we
                  can write $B_1(hD_x) = B^*(hD_x R_1)B + hR$; here
		\[
		|\bsymbol(R_1) | \leq 2\beta^{-2}\delta^{-1}\omega^{1/2} \leq 6 \beta^{-1}.
		\]
		\item Now we have dealt with the analogs of the first two
                  terms in \eqref{eq:hypcommutator}, and we turn to the term
                  $(i/h) [\tilde{P},A^*A].$ If $\tilde P_0$ is an operator with principal symbol $\tilde p_0$, write $\tilde P = \tilde P_0 + (\tilde P - \tilde P_0)$. The principal symbol of $(i/h)[\tilde P -\tilde P_0,A^*A]$ is given by
		\[
		\bhamvf_{\tilde p -\tilde p_0}(a^2) = (\bhamvf_{\tilde p -\tilde p_0}\phi)\psi b^2 + e.
		\]
		In view of \eqref{eq:tildeperror}, we can write $(i/h)[\tilde P -\tilde P_0,A^*A] = B^*R_0B + E + hR$, where
		\[
		|\bsymbol(R_0)| \leq (3^\theta M)  \delta^{\theta}\beta^\theta(1 +  3\beta^{-1}).
		\]
		Thus $R_0$ is as advertised, since $\beta < 1$.
		
		\item Finally, $(i/h)[\tilde P_0, A^*A]$ has principal symbol 
		\[
		\bhamvf_{\tilde p_0}(a^2) = -b^2 + e,
		\]
		hence we can write $(i/h)[\tilde P_0, A^*A] = -B^*B + E + hR$ as desired.
	\end{inparaenum}
\end{proof}

We proceed  as in Section \ref{subsect:hyperbolicregion}, using Lemma \ref{lem:glancingcommutator} to write
\begin{align*}
-(2/h)\Im \left<APu,Au\right>  &= (i/h)\left<[A^*A,P]u,u\right> \\  &= \|Bu\|^2_{L^2} + \left<R_0Bu,Bu\right> + \left< R_1Bu, (hD_x)Bu\right> \\ &- \left< Eu,u\right> - h\left<Ru,u\right> 
\end{align*}
for $u \in H^1_h(X)$. Applying \eqref{eq:L2H1constant} we can bound
\begin{equation} \label{eq:Bbound}
c_0\|Bu\|^2_{H^1_h} \leq  \|Bu\|^2_{L^2} + C\|GPu\|_{H^{-1}_h}^2 + Ch^2\| Gu \|_{H^1_h}^2 + \mathcal{O}(h^\infty)\|u\|_{H^1_h}^2,
\end{equation}
where $c_0>0$ independent of $\beta,\delta$, and where $G$ is elliptic on $\WFb(B)$. We now choose $\beta$ depending on $\delta$:
\begin{lem} \label{lem:glancingerrorbound}
	Let $\varepsilon>0$. There exists $c>0$ such that if $\delta \in (0,\delta_0)$ and $\beta = c \delta^{\theta/(2-\theta)}$, then
	\begin{multline*}
			\lvert\left<R_0Bu,Bu\right>\rvert + \lvert\left< R_1Bu, (hD_x)Bu\right>\rvert \leq \varepsilon\| Bu\|_{H^1_h}^2 \\ + Ch^{-1}\|GPu\|_{H^{-1}_h}^2 + Ch\| Gu \|_{H^1_h}^2 + \mathcal{O}(h^\infty)\|u\|^2_{H^1_h}
	\end{multline*}
	for each $\delta \in (0,\delta_1)$ and $u \in H^1_h(X)$. The
        constant $C = C(\delta)$ depends on $\delta$ through $\beta$.
\end{lem}
\begin{proof}
	First consider $R_1$, in which case
	\[
	\left<R_1 v, (hD_x)v\right> \leq 2C_1\beta^{-1} \|v\|_{L^2}\|hD_x v\|_{L^2} + Ch\| v\|_{H^1_h}^2,
	\]
	where $C_1$ does not depend on $\beta,\delta$. Apply this to $v = Bu$ and use Lemma \ref{lem:Dxbound}. Indeed,
	\[
	\WFb(B) \subset \{|\rho_0| \leq 2\delta, \, \omega^{1/2} \leq 2\beta\delta\},
	\]
	and by our choice of $\rho_0$ we conclude that $|x| \leq 2\beta\delta$ and $|\tilde p_0| \leq 2\beta\delta$ on $\WFb(B)$. Now 
	\[
	|\tilde p| \leq |\tilde p_0 | + |\tilde p -\tilde p_0| \leq 2\beta\delta + C|x|^\theta \leq C_1' (\beta\delta)^\theta
	\]
	on $\WFb(B)$, where $C_1'$ does not depend on
        $\beta,\delta$. Thus, by Lemma \ref{lem:Dxbound} and
        Cauchy--Schwarz,
	\begin{multline*}
		\left<R_1 Bu, (hD_x)Bu\right>\leq  (\varepsilon/2)
                \|Bu\|_{L^2}^2 + C_1''
                \beta^{-2}(\beta\delta)^{\theta} \|Bu\|_{L^2}^2 \\ +
                Ch^{-1}\|GPu\|_{H^{-1}_h}^2 + C h\| Gu \|_{H^1_h}^2 + \mathcal{O}(h^\infty)\|u\|_{H^1_h}^2,
	\end{multline*}
	where again $C_1'' > 0$ is independent of $\beta,\delta$; here
        we have taken the $\delta$ in the notation of
        Lemma~\ref{lem:Dxbound} to be a multiple of
        $(\beta\delta)^{\theta}$. Bounding $\lvert\left<R_0 Bu,Bu \right>\rvert$ is done exactly as in Lemma \ref{lem:hyperbolicerrorbound}, yielding 
	\[
	\lvert\left<R_0 Bu,Bu \right>\rvert \leq C_1'''\beta^{-1}(\beta \delta)^\theta \| Bu\|_{L^2}^2 + \mathcal{O}(h^\infty)\| u \|_{L^2}^2.
	\]
	It therefore suffices to choose $\beta = c\delta^{\theta/(2-\theta)}$ with $c>0$ sufficiently large.
\end{proof}
	
The rest of the argument in Section \ref{subsect:hyperbolicregion} goes through verbatim. Thus if $G \in \Psibh^\COMP$ is elliptic on $\WFb(B)$, and $Q_1 \in \Psibh^\COMP$ is elliptic on $\WFb(E)$ with $\WFb(Q_1) \subset \ELLb(G)$, then
\[
\|Bu\|_{H^1_h} \leq C h^{-1}\|GPu\|_{H^{-1}_h} + C\| Q_1u \|_{H^1_h} + C h^{1/2} \|G u \|_{H^1_h} + \mathcal{O}(h^\infty)\| u \|_{H^1_h}.
\]
Performing the inductive step requires that the commutant be slightly modified at each step; however this does not cause any problems, and proceeds exactly as in \cite{vasy2008propagation}.

\subsection{Proof of Theorem \ref{theo:GBBpropagation}} \label{subsect:propagationofsingularities}
We now prove Theorem \ref{theo:GBBpropagation}, following \cite[Section 8]{vasy2008propagation} quite closely. Without assuming that $Pu = 0$, we prove the slightly stronger statement that 
\[
F = \WFb^{1,r}(u) \setminus \WFb^{-1,r+1}(Pu)
\]
 is the union of maximally extended $\GBB$ within $\cchare \setminus \WFb^{-1,r+1}(Pu)$ for each $r \in \RR \cup\{+\infty\}$. It suffices to prove that for each $q_0 \in F$ there exists $\varepsilon > 0$ and a $\GBB$ 
\[
\gamma : [-\varepsilon,0] \rightarrow F
\]
satisfying $\gamma(0) = q_0$. Indeed, given any $Z \subset \cchare$, let $\mathcal{P}_Z$ denote the set of $\GBB$s defined on open intervals $(\alpha,0]$ with values in $Z$, such that $\gamma(0) = q_0$. There is a natural partial order on $\mathcal{P}_Z$ such that each chain has an upper bound. Thus, provided $\mathcal{P}_Z \neq \emptyset$, Zorn's lemma guarantees the existence of maximally extended $\GBB$ in $Z$ on an interval $(\alpha_{\max} ,0]$, where possibly $\alpha_{\max} = -\infty$. We apply this argument with the set $Z = F$, but arguing verbatim as in \cite[Section 8]{vasy2008propagation}, a maximal $\GBB$ within $F$ is also maximal within $\cchare \setminus \WFb^{-1,r+1}(Pu)$. Replacing the backwards propagation estimates with their forward counterparts, we similarly deduce the existence of a maximal $\GBB$ on $[0,\beta_{\max})$.

By Proposition  \ref{prop:elliptic} we can assume that $q_0 \in \hyp$ or $q_0\in \gl$. In the latter case it suffices to assume $q_0 \in \gl \cap T^*Y$, since the semiclassical Duistermaat--H\"ormander theorem on propagation of singularities applies when $q_0 \in \gl \cap T^*(X\setminus Y)$, see \cite[Appendix E]{zworski:resonances} for example

We begin with the proof when $q_0 \in \hyp$. Fix a normal coordinate patch $\mathcal{U}$ such that $\bT^*_{\mathcal{U}}X$ contains $q_0$. Since the complement of $\WFb^{-1,r+1}(Pu)$ is open, first choose a precompact neighborhood $U \subset \cchare \cap \bT^*_\mathcal{U} X$ of $q_0$ such that $U \cap \WFb^{-1,r+1}(Pu) = \emptyset$.  From the local compactness of $\cchare$, by further shrinking $U$ we  assume that 
	\begin{equation} \label{eq:sigmaincreasing}
	\hamvf_p(x\xi) > 0 \text{ on } \pi^{-1}(U),
	\end{equation}
	since this holds along $\pi^{-1}(\{q_0\})$. Also fix an open subset $U' \subset U$ containing $q_0$ with closure in $U$. By Lemma \ref{lem:GBBgrowth} there exists $\varepsilon_0$ such that every $\GBB$ defined on $[-\varepsilon_0,0]$ with $\gamma(0) \in U'$ satisfies $\gamma([-\varepsilon_0,0]) \subset U$. In particular, $\sigma$ is increasing on any such $\GBB$ by \eqref{eq:sigmaincreasing}. By Proposition~\ref{prop:qualitativehyp}, there is a sequence of points 
	\[
	q_n \in F \cap \{\sigma < 0\} \cap U'
	\]
	tending to $q_0$. Since $q_n \in \cchare$ and $\sigma(q_n) < 0$, it follows that $x(q_n) \neq 0$. By the Duistermaat--H\"ormander theorem on propagation of singularities, there is a maximally extended $\GBB$ 
	\[
	\gamma_n : (-\varepsilon_n,0] \rightarrow F \cap T^*(X\setminus Y)
	\]
	such that $\gamma_n(0)  = q_n$.

        Arguing as in \cite{vasy2008propagation}, the claim is that
        $\varepsilon_n \geq \varepsilon_0$. Indeed, since
        $\gamma_n(0) \in U'$, it would otherwise be the case that
        $\gamma_n(s) \in U$ for all $s \in (-\varepsilon_n,0]$. Now
        $\gamma_n$ extends to $[-\varepsilon_n,0]$ by Proposition
        \ref{prop:equicontinuous}, and $\sigma$ is increasing along
        $\gamma_n$. Therefore $\sigma(\gamma_n(-\varepsilon_n)) < 0$,
        so $x(\gamma_n(-\varepsilon_n)) \neq 0$, which contradicts
        maximality of $\gamma_n$. Thus we have a sequence of $\GBB$s
\[
\gamma_n|_{[-\varepsilon_0,0]} : [-\varepsilon_0,0] \rightarrow F \cap \overline{U}
\] 
with values in a compact set. According to Proposition \ref{prop:equicontinuous}, there is a subsequence converging uniformly to a $\GBB$ 
	\[
	\gamma :[-\varepsilon_0,0] \rightarrow F \cap \overline{U},
	\]
	thus completing the proof.

For the proof when $q_0 \in \gl \cap T^* Y$, we begin with a variant of Proposition \ref{prop:equicontinuous}. Fix a normal coordinate patch $\mathcal{U}$.
\begin{lem} \label{lem:cauchypeano}
	Let $K \subset \bT^*_{\mathcal{U}}X$ be compact, $Z \subset \cchare$ be closed, and $[a,b]\subset \RR$ a compact interval. Fix constants $r, C_0 > 0$. For each $n$, consider a partition 
	\[
	a = s_{n,0} < s_{n,1} < \ldots < s_{n,k_n} = b.
	\]
	Set $q_{n,j} = \gamma_n(s_{n,j})$ and $\delta_{n,j} = |s_{n,j} - s_{n,j-1}|$ for $j = 1,\ldots, k_n$. Suppose that 
	\[
	\gamma_n : [a,b] \rightarrow K
	\] 
	is a sequence of continuous maps, where the restriction of $\gamma_n$ to $[s_{n,j-1},s_{n,j}]$ is either a $\GBB$ with values in $Z$, or the following holds
	\begin{itemize} \itemsep6pt
	\item $q_{n,j}\in Z \cap \gl \cap T^*Y$ and $q_{n,j-1} \in Z$, where 
\begin{equation}\label{shrinkingball}
	q_{n,j-1} \in \mathsf{B}(\exp(-\delta \bhamvf_{\tilde
          p_0})(q_{n,j}),C_0\delta_{n,j}^{1+r}),\quad r>0.
\end{equation}
		\item The restriction of $\gamma_n$ to
                  $[s_{n,j-1},s_{n,j}]$ is a line segment (in local coordinates), and $\delta_{n,j} \leq 2^{-n}|b-a|$.
	\end{itemize}
Then there is a subsequence of $\gamma_n$ converging uniformly to a $\GBB$ $\gamma:[a,b]\rightarrow K \cap Z$.
\end{lem}
\begin{proof}
	Since $K \subset \bT^*_{\mathcal{U}}X$ is compact, we can choose $L > 0$ such that
\begin{equation} \label{eq:lipschitzglancing}
	|\gamma_n(s) - \gamma_n(s')| \leq L|s-s'|
	\end{equation}
	for $s,s'\in[a,b]$, uniformly in $n$. To see this, it suffices to consider the case when $s,s'$ lie in a single interval $[s_{n,j-1},s_{n,j}]$. By \eqref{eq:gammalipshitz}, the result is clear if the restriction of $\gamma_n$ to $[s_{n,j-1},s_{n,j}]$ is a $\GBB$. If the restriction is a line segment, then \eqref{eq:lipschitzglancing} holds with 
	\[
	L = \dfrac{|q_{n,j-1}-q_{n,j}|}{\delta_{n,j}} \leq C_0 + \sup \{|\bhamvf_{\tilde p_0}(q)|: q\in K\},
	\] 
	which is bounded uniformly in $n$.
	By the Arzel\`a--Ascoli theorem, there is a subsequence of $\gamma_n$ converging to a curve $\gamma:[a,b] \rightarrow K$, and since $Z$ is closed, $\gamma$ actually maps into $K\cap Z$ under our hypotheses. It remains to check that $\gamma$ is a $\GBB$.
	
	First, suppose that $\gamma(s_0) \notin \gl \cap T^*Y$. Since $\gl \cap T^*Y$ is closed in $\bdotT^*X$, there is a neighborhood $O$ of $\gamma(s_0)$ that is also disjoint from $\gl \cap T^*Y$. Choose $\delta > 0$ such that $\gamma_n(s) \in O$ for $s \in (s_0-2\delta,s_0+2\delta)$ and $n \geq N_0$. By assumption, the restriction of $\gamma_n$ to $[s_0-\delta,s_0+\delta]$ is a $\GBB$, increasing $N_0$ if necessary, so by Proposition \ref{prop:equicontinuous}, the restriction of $\gamma$ to $[s_0-\delta,s_0+\delta]$ is a $\GBB$.

On the other hand, suppose that $\gamma(s_0) \in \gl \cap T^*Y$. Let $\varpi_0 = \pi^{-1}(q_0)$, and suppose that $f \in \CI(T^*X)$ is $\pi$-invariant. We must show that 
\[
D_\pm(f_\pi \circ \gamma)(s_0) \geq (\hamvf_p f)(\varpi_0).
\]	
Furthermore, at glancing points it suffices to check this when $f$ is one of the $\pi$-invariant functions $\{x,y,\eta\}$. This follows from the fact that 
\[
f(x,y,\xi,\eta) = f_0(y,\eta) + xf_1(x,y,\xi,\eta)
\]
and $x(\varpi_0) = \xi(\varpi_0) = 0$. Let $c_0 = (\hamvf_p f)(\varpi_0)$. We show that for each $\varepsilon > 0$ there exists $\delta>0$ such that
\[
f_\pi(\gamma(s)) - f_\pi(\gamma(s_0)) \geq (c_0 - \varepsilon)(s-s_0)
\]
for each $ s\in (s_0,s_0+\delta)$.

Since the map $\pi$ is proper and $\bT^*X$ is locally compact, from the continuity of $\hamvf_p f$ there is a neighborhood $O \subset \cchare$ of $\gamma(s_0)$ such that 
\[
\inf\{ (\hamvf_p f)(\varpi) : \varpi \in \pi^{-1}(O)\} \geq (c_0-\varepsilon/4).
\]
By uniform convergence, we choose $\delta > 0$ such that $\gamma_n(s) \in O$ for $s \in (s_0,s_0+2\delta)$ and $n \geq N_0$. 

Fix the interval $[\alpha,\beta] = [s_{n,j-1},s_{n,j}]$
containing $s_0$, where we choose $s_0 = s_{n,j-1}$ if $s_0$ happens to be an endpoint. For $s \in (s_0, s_0 + \delta)$ consider the function
\[
F_n = (f_\pi \circ \gamma_n)(s) - (c_0-\varepsilon/2)s.
\]
If the restriction of $\gamma_n$ to $[\alpha,\beta]$ is a $\GBB$, then $D_+ F_n(s) \geq 0$ on the intersection $[\alpha, \beta] \cap (s_0,s_0+\delta)$ by our choice of $O$. Otherwise the restriction of $\gamma_n$ to $[\alpha,\beta]$ is the line segment
\[
\gamma_n(s) = q_{n,j-1} + (s-\alpha)\frac{q_{n,j} - q_{n,j-1}}{\delta_{n,j}}.
\]
Since $f$ is one of $\{x,y,\eta\}$, it is clear that $f_\pi\circ \gamma_n$ is actually differentiable on $[\alpha,\beta]$, and
\[
D_+ (f_\pi\circ \gamma_n)(s) = \frac{f_\pi(q_{n,j})-f_\pi(q_{n,j-1})}{\delta_{n,j}}
\]
is \emph{constant} on $[\alpha,\beta]$. By \eqref{shrinkingball},
\begin{equation} \label{eq:D_+estimate}
|D_+(f_\pi \circ \gamma_n)(s) -(\bhamvf_{\tilde p_0}f_{\pi})(q_{n,j})| \leq C_0 |\beta-\alpha|^r\leq C_0 2^{-nr}.
\end{equation}
uniformly in $n$. By further increasing $N_0$ so that $2^{-N_0}|\beta-\alpha| < \delta$, we can assume that $\beta \in (s_0,s_0+2\delta)$ for $n \geq N_0$. Thus $q_{n,j} \in O \cap \gl \cap T^*Y$. Let $\varpi_{n,j}$ satisfy $\pi(\varpi_{n,j}) = q_{n,j}$.

We now collect several observations. First, since $q_{n,j}$ is a glancing point over $Y$, the equality \eqref{eq:equalityatglancing} holds. Furthermore, since $\tilde p_0$ depends only on $(y,\eta)$ and $f$ is one of $\{x,y,\eta\}$, it follows that 
\begin{equation} \label{eq:equalityatglancing1}
(\hamvf_{p} f)(\varpi_{n,j}) = (\hamvf_{\tilde p_0} f)(\varpi_{n,j}) = (\bhamvf_{\tilde p_0} f_\pi)(q_{n,j}).
\end{equation}
Therefore, since $\varpi_{n,j} \in \pi^{-1}(O)$ for $n \geq N_0$, by combining \eqref{eq:D_+estimate} and \eqref{eq:equalityatglancing1} we obtain
\begin{align*}
D_+ F_n(s) &= D_+(f_\pi \circ \gamma_n)(s) - (c_0-\varepsilon/2) \\&\geq  (\hamvf_{p} f)(\varpi_{n,j}) - C_0 2^{-nr} - (c_0-\varepsilon/2) \\& \geq (c_0 - \varepsilon/4)  - C_0 2^{-nr}- (c_0-\varepsilon/2)  \geq 0
\end{align*}
on $[\alpha,\beta] \cap (s_0,s_0+\delta)$ if $N_0$ is increased so that $C_0 2^{-nr} \leq \varepsilon/4$ for $n \geq N_0$.

Thus we know that $D_+ F_n(s) \geq 0$ on $(s_0,s_0+\delta)$ for $n \geq N_0$. Since $F_n$ has a nonnegative lower right Dini derivative, it is non-decreasing, and so
\[
f_\pi(\gamma_n(s)) - f_\pi(\gamma_n(s_0)) \geq (c_0 - \varepsilon/2)(s-s_0)
\]
for $s \in (s_0, s_0 + \delta)$ and $N \geq N_0$. We obtain the desired inequality for each $s \in (s_0,s_0 + \delta)$ by choosing $n \geq N$ sufficiently large (depending on $s_0$ and $s$) so that
\[
| f_\pi(\gamma_n(s)) - f_\pi(\gamma(s))| + |f_\pi(\gamma(s_0)) - f_\pi(\gamma_n(s_0))| \leq (\varepsilon/2) (s-s_0).
\]		
A similar argument applies for $ s\in (s_0-\delta,s_0)$.
\end{proof}

The proof of Theorem \ref{theo:GBBpropagation} for $q_0 \in \gl \cap T^*Y$ is then a relatively straightforward application of Lemma \ref{lem:cauchypeano}. We again fix a precompact neighborhood $U \subset \cchare \cap \bT^*_\mathcal{U} X$ of $q_0$ such that $U \cap \WFb^{-1,r+1}(Pu) = \emptyset$. Let $L$ be as in the proof of Lemma \ref{lem:cauchypeano} where we take $K = \overline{U}$, and let $C_0$ be as in the statement of Proposition \ref{prop:glancing}. By Lemma \ref{lem:GBBgrowth} we can choose $\varepsilon_0 > 0$ such that 
\[
U' = \mathsf{B}(q_0, (C_0+L)\varepsilon_0) \cap \cchare \subset U
\]
and if $\gamma: [-\varepsilon,0]$ is a $\GBB$ with $\varepsilon \in (0,\varepsilon_0)$ such that $\gamma(0) \in U'$, then $\gamma([-\varepsilon,0]) \subset U$. Let $\delta_n = 2^{-n}\varepsilon_0$. We then define a family of approximate $\GBB$ inductively. First, set $s_{0} = 0$, and suppose that a continuous curve $\gamma$ has already been defined on $[s_{j},0]$ such that if $q_{j} = \gamma(s_{j})$, then
\[
\gamma([s_{j},0]) \subset \mathsf{B}(q_0, (C_0+L) s_{j}|) \subset U.
\]
 We then extend $\gamma$ to an interval $[s_{j+1},s_{j}]$ as follows. 
 
 If $q_{j} = \gamma(s_{j}) \in F \cap \gl \cap T^*Y \cap U'$, then by Proposition \ref{prop:glancing} we can choose $q_{j+1} \in \WFb^{1,r}(u)$ such that
 \[
|q_{j+1} - \exp(-\delta_n \bhamvf_{\tilde p_0})(q_{j})| \leq C_0 (\delta_n)^{\theta/(2-\theta)} \leq C_0 \delta_n.
 \]
Let $s_{j+1} = \max(-\varepsilon_0,s_{j} - \delta_n)$. In particular since $q_{j} \in \overline{U}$, the line connecting $q_{j}$ to $q_{j+1}$ is contained in $\mathsf{B}(q_0,(C_0+L)|s_{j+1}|) \subset U'$. This also shows that $q_{j+1} \in F$.

Otherwise, $q_{j} \in F \setminus (\gl \cap T^*Y)$. We know that there is a maximally extended $\GBB$
\[
\gamma': (-\varepsilon',0] \rightarrow F \setminus (\gl \cap T^*Y)
\]
with $\gamma'(-\varepsilon') = q_{j}$.  Let $s_{j+1} = \max(-\varepsilon_0,s_{j} - \varepsilon')$. If $s_{j+1} = -\varepsilon_0$, then we can extend $\gamma_N$ from $[s_{j},0]$ to all of $[-\varepsilon_0,0]$ by concatenating with $\gamma'$. The other possibility is that $s_{j+1} > -\varepsilon_0$. Then $\gamma'$ extends up to $-\varepsilon'$, so we concatenate with $\gamma'$ and set $q_{j+1} = \gamma'(-\varepsilon')$. Either way, the extension by $\gamma'$ has its image in $\mathsf{B}(q_0,(C_0+L)|s_{j+1}|)$. In the second case, $q_{j+1}\in F \cap \gl \cap T^*Y$ by maximality, and thus we proceed as in the first step.

 The sequence $\gamma_N$ just constructed satisfies the hypotheses of Lemma \ref{lem:cauchypeano} with $K = \overline{U}$ and $Z = F$, letting $r = \theta/(2-\theta)$. This completes the proof of Theorem \ref{theo:GBBpropagation} when $q_0 \in \gl \cap T^*Y$.

\section{Semiclassical paired Lagrangian distributions}
\label{sec:paired}

In this section we collect the technical tools that we will need on
semiclassical paired Lagrangian distributions.  We largely follow the
discussion in \cite{de2014diffraction} in the homogeneous case (see
also the introduction for further references), but we
have been forced to revisit some of the foundations of the subject as
there is no existing treatment in the semiclassical setting.
  
\subsection{Nested conormal distributions} \label{subsect:nestedconormal}
Our paired Lagrangian distributions are locally modeled on oscillatory integrals in $\RR^{m}$ associated with the conormal bundles of two nested submanifolds of $\RR^m$.

\begin{defi}
	We say that an $h$-dependent function $a \in \CI(\RR^m_{x} \times \RR_{\xi'}^k \times \RR_{\xi''}^{n})$ is in the symbol class $S^{r,p}_h(\RR^m_{x};\RR^{k}_{\xi'};\RR^{n}_{\xi''})$ if
\blue{	\[
	| (hD_{\xi'})^\alpha D_{\xi''}^\beta D_x^\gamma a(x,\xi)| \leq C_{\alpha\beta \gamma N}\left<\xi'/h\right>^{r-|\alpha|}\left<(\xi',\xi'')\right>^{p}
	\]}
	for all multiindices $\alpha,\beta,\gamma$ and $N \in \RR$. We say $a \in S^{r,\COMP}_h \subset S^{r,-\infty}_h$ if $\supp a$ is contained in an $h$-independent compact set. 
\end{defi}

\blue{\begin{rem}\label{rem:symbolexample}
An instructive example is as follows.  Let $\chi \in \CcI(\RR),$ equal to
$1$ near the origin; let $\psi\in \CI(\RR)$ equal
$0$ on $(-\infty, 1)$ and $1$ on $(2,\infty).$  On $\RR^1 \times
\RR^1$ we set
\begin{equation}\label{symbolexample}
a(x',x'',\xi',\xi'')=\psi(\xi'/h)\chi(\xi') \chi(\xi'')
\end{equation}
This symbol lies in
$S_h^{0,-\infty}(\RR^2; \RR^1; \RR^1).$
  \end{rem}}

\begin{rem} 
This class of symbols can be interpreted in terms of a certain semiclassical blow-up as follows. Our symbols will be functions on $\RR^{m}\times \RR^{k+n} \times  (0,1)_h$ that lift to certain conormal functions on $\RR^n \times \mathsf{S}$, where $\mathsf{S}$ is defined as the \emph{blow-up}
\[
\mathsf S = [ \RR^{k+n} \times [0,1)_h;\{\xi'=0, \, h=0\}].
\]
The space $\mathsf{S}$ has two boundary hypersurfaces, $\mathrm{ff}$ and $\mathrm{sf}$, corresponding to the lifts of $\{\xi'=0, \, h=0\}$ and its complement within $\{h=0\}$, respectively. We also fix
\[
\rho_{\mathrm{sf}} = \left<\xi'/h\right>^{-1}, \  \rho_{\mathrm{ff}} = h\left<\xi'/h\right>.
\]
These lift to $\mathsf{S}$ as smooth, globally defined boundary defining functions for $\mathrm{sf}, \mathrm{ff}$, and $h = \rho_{\mathrm{sf}} \rho_{\mathrm{ff}}$. 

The lift of $\{h\geq \varepsilon|\xi'|\}$ intersects the interior of
the front face. Valid coordinates here are $(x,\xi'',\Xi,h)$, where
$\Xi = \xi'/h$, and in this region $h$ is a boundary defining function
for $\mathrm{ff}$. Furthermore, elements of
$\mathcal{V}_\B(\mathsf{S})$ (vector fields tangent to all boundary
faces) that are supported near the interior of $\mathrm{ff}$ are spanned over $\CI(\mathsf{S})$ by $\{\partial_x, \partial_{\xi''}, \partial_\Xi, h\partial_h\}$. In order, these vector fields are the lifts of $\{\partial_x,  \partial_{\xi''},  h\partial_{\xi'},  h\partial_h + \xi' \cdot  \partial_{\xi'}\}$.

A different set of coordinates is needed in $\{ |\xi'| \geq \varepsilon h\}$. Restricting in addition to the set where $|\xi'_k| \geq \varepsilon |\xi'_j|$, we can use projective coordinates $(x, \xi'', \theta,\varrho,\Omega)$, where
\[
\theta = \xi'_k, \quad \varrho = h/\xi'_k, \quad  \Omega_j = \xi'_j/\xi'_k
\]
for $j\neq k$.
In particular, $\theta, \varrho$ are boundary defining function for $\mathrm{ff}$ and $\mathrm{sf}$, respectively. 
In this case, $\mathcal{V}_\B(\mathsf{S})$ is spanned over $\CI(\mathsf{S})$ by $\{\partial_x, \partial_{\xi''}, \theta\partial_\theta, \varrho \partial_\varrho,\partial_\Omega\}$, which are the lifts of $\{\partial_x,  \partial_{\xi''}, \xi'\cdot \partial_{\xi'},  h\partial_h, \xi'_k \partial_{\xi'_j}\}$, in order.

Without localizing to the different regions of $\mathsf{S}$, it follows from the previous two paragraphs that $\mathcal{V}_\B(\mathsf{S})$ is spanned over $\CI(\mathsf{S})$ by the lifts of 
\[
\{\partial_{\xi''}, \xi'_j \partial_{\xi'_i}, h\partial_h, h\partial_{\xi'}\},
\] 
where $i,j \in \{ 1,\ldots,k\}$. Thus if we ignore $h\partial_h$ derivatives, $S^{r,\COMP}_h$ corresponds exactly to compactly supported $\rho_{\mathrm{sf}}^{-r}L^\infty(\mathsf{S})$ functions that remain in the same space under arbitrary applications of $\mathcal{V}_\B(\mathsf{S})$.
\end{rem}

On $\RR^m$, consider a splitting of coordinates  $x = (x',x'',x''') \in \RR^k \times \RR^{m-d-k} \times \RR^d$,
and consider the submanifolds
\[
S_1 = \{x'' = 0\}, \quad S_0 = \{x'=0,\,x''=0\}.
\] 
Thus $S_0 \subset S_1 \subset \RR^{m}$ are nested with codimensions $\codim S_1 = m-d-k$ and $\codim S_0 = m-d$. In particular, we have $d = \dim S_0$. Their conormal bundles are given by
\begin{align*}
N^*S_1 = \{x'' = 0, \, \xi'=0,\,\xi'''=0\}, \\ N^*S_0 = \{x' =0, \,x''=0, \,\xi'''=0\},
\end{align*}
where $(x',x'',x''',\xi',\xi'',\xi''')$ are canonical coordinates on $T^* \RR^{m} = \RR^m \times \RR^m$. We view these as model Lagrangians, writing 
\[
{\Lambda}_0 = N^*S_0, \quad {\Lambda}_1 = N^*S_1.
\]
We then consider oscillatory integrals whose amplitudes are elements of 
\[
S^{r,-\infty}_h= S^{r,-\infty}_h(\RR^m_x;\RR^{k}_{\xi'};\RR^{d-m-k}_{\xi''}).
\] 
Observe that elements of $S^{r,-\infty}$ depend on $(x,\xi',\xi'')$, but not $\xi'''$. Given $a \in S^{r,-\infty}_h$, define the oscillatory integral
\begin{equation} \label{eq:modelpairedlagrangian}
u(x) = (2\pi h)^{-3m/4 - k/2 + d/2}\int e^{\sfrac{i}{h}(\left<x',\xi'\right> + \left<x'',\xi''\right>)} a(x,\xi',\xi'')\, d\xi'd\xi''.
\end{equation}
Since $a$ is rapidly decaying in $(\xi',\xi'')$ for each fixed $h>0$, this certainly defines a smooth function on $\RR^m$. We now write
\[
\bar{x} = (x',x''), \quad \bar{\xi} = (\xi',\xi'')
\]
 and analyze mapping properties of the Gauss transform $a \mapsto e^{-ih\langle D_{\bar x},D_{\bar \xi}\rangle}a$ on $S^{r,-\infty}_h$.
\begin{lem} \label{lem:gauss}
  If \blue{$a \in S^{r,-\infty}_h$,} then $e^{-ih\langle D_{\bar x}, D_{\bar \xi} \rangle} a \in S^{r,-\infty}_h$.
	Furthermore,
	\[
	e^{-ih\langle D_{\bar x}, D_{\bar \xi} \rangle}a(x,\xi) - \sum_{j = 0}^N \left<-iD_{\bar x},hD_{\bar \xi}\right>^j a(x,\xi)/j! \in S^{r-N-1,-\infty}_h.
	\]
\end{lem}
\begin{proof}
	Since the dependence on $x'''$ is smooth and parametric, it suffices to consider the case $d=0$, so that $x = \bar x $ and $\xi = \bar \xi$. Set 
	\[
	y = (h^{1/2}x',x''), \quad \eta = (h^{-1/2}\xi',\xi'').
	\] 
	For a given $a \in S^{r,-\infty}_h$, define the rescaled amplitude $b(y,\eta) = a(x,\xi)$, which therefore satisfies
	\[
	| D_{y}^{\alpha} D_{\eta}^{\beta}  b(y,\eta)| \leq C_{\alpha\beta N} h^{-(|\alpha'|+|\beta'|)/2}\langle\eta'/h^{1/2}\rangle^{r-|\beta'|}\langle(h^{1/2}\eta', \eta'')\rangle^{-N},
	\]
	where we have written $\alpha = (\alpha',\alpha'')$ and $\beta = (\beta',\beta'')$. After a change of variables,
	\begin{equation} \label{eq:gauss}
	e^{-ih\langle D_{y},D_{\eta}\rangle }b(y,\eta) = (2\pi)^{-m}\int e^{i\left<z,\zeta\right>}b(y-h^{1/2} z, \eta - h^{1/2}\zeta) \, dz d\zeta.
	\end{equation}
	Write the integral on right hand side of \eqref{eq:gauss} as a sum $A+B$, where $A$ is the result of inserting a cutoff $\chi \in \CcI(\RR^{2m};[0,1])$ into the integrand which is identically one for $|(z,\zeta)| \leq 1$ and vanishes for $|(z,\zeta)| \geq 2$. We then estimate
	\begin{align*}
	|D^\alpha_y D^\beta_\eta A(y,\eta)| &\leq C\sup_{|(z,\zeta)| \leq 2}  |(D^\alpha_y D^\beta_\eta b)(y-h^{1/2}z,\eta- h^{1/2}\zeta)| \\&\leq  C_{\alpha\beta N} h^{-(|\alpha'|+|\beta'|)/2}\langle\eta'/h^{1/2}\rangle^{r-|\beta'|}\langle(h^{1/2}\eta', \eta'')\rangle^{-N},
	\end{align*}
	since $\left<(\eta' - h^{1/2}\zeta')/h^{1/2}\right> \asymp C\left<\eta'/h^{1/2}\right>$ when $|\zeta'|$ is bounded. To estimate $B$, we integrate by parts using the operator 
	\[
	L = |(z,\zeta)|^{-2}(z D_\zeta + \zeta D_z),
	\] 
	defined for $|(z,\zeta)| \geq 1$ and satisfying $L(i\left<z,\zeta\right>) = 1$. By Peetre's inequality,
	\begin{multline*}
	|D^\alpha_y D^\beta_\eta (L^t)^k (1-\chi(z,\zeta))b(y-h^{1/2}z,\eta- h^{1/2}\zeta)|  \\\leq C_{\alpha\beta N} h^{-(|\alpha'|+|\beta'|)/2}\langle\eta'/h^{1/2}\rangle^{r-|\beta'|}\langle(h^{1/2}\eta', \eta'')\rangle^{-N}\left<(z,\zeta)\right>^{N(k)}
	\end{multline*}
	for $|(z,\zeta)|\geq 1$, where $N(k) \rightarrow -\infty$ as $k\rightarrow \infty$. Integrating by parts $k$ times for sufficiently large $k$ shows that $B$ satisfies the same symbol estimates as $A$. This establishes the desired mapping properties of $e^{-ih\langle D_x,D_\xi\rangle}$, since 
	\[
	e^{-ih\langle D_x,D_\xi\rangle}a(x,\xi) = e^{-ih\langle D_y,D_\eta\rangle}b(y,\eta).
	\]
	To obtain the expansion, simply Taylor expand 
	\[
	e^{-ih\langle D_x,D_\xi\rangle} = \sum_{j=0}^N \frac{1}{j!}\left<-iD_x,hD_\xi\right>^j + \frac{(-i)^N}{N!}\int_{0}^1 (1-t)^N e^{-ith\left<D_x,D_\xi\right>}\left<D_x,hD_\xi\right>^{N+1} \, dt.
	\]
	The remainder can be estimated by replacing $h$ with $th$ and repeating the argument above.
\end{proof}

Lemma \ref{lem:gauss} shows that we can always write $u$ given by \eqref{eq:modelpairedlagrangian} in the form
\begin{equation} \label{eq:reducedmodelpairedlagrangian}
u(x) = (2\pi h)^{-3m/4 - k/2 + d/2}\int e^{\sfrac{i}{h}(\left<x',\xi'\right> + \left<x'',\xi''\right>)} \,c(x''',\xi',\xi'')\, d\xi'd\xi''
\end{equation}
where the amplitude $c \in S^{r,-\infty}_h$ depends only on $x'''$ in the base variables. Indeed, by the Fourier inversion formula,
\begin{equation} \label{eq:reducedbase}
c(x''',\xi,\xi'') = e^{-ih\langle D_{\bar x},D_{\bar \xi}\rangle}a(x,\xi',\xi'')|_{x'=x''=0},
\end{equation}
which defines an element of $S^{r,-\infty}_h$ by Lemma \ref{lem:gauss}.

\begin{lem} \label{lem:pseudodifferentialaction}
	Let $b \in S^0(\RR^m;\RR^m)$.  If $c \in S^{r,-\infty}_h$ and
        $u$ is given by \eqref{eq:reducedmodelpairedlagrangian}, then
        there is $\tilde c \in S^{r,-\infty}_h$ such that
	\begin{equation} \label{eq:Bu}
	b(x,hD)u = (2\pi h)^{-3m/4 - k/2 + d/2}\int
        e^{\sfrac{i}{h}(\left<x',\xi'\right> +
          \left<x'',\xi''\right>)} \,\tilde c (x''',\xi',\xi'')\, d\xi'd\xi'',
	\end{equation}
	and moreover
	\[
	\tilde c (x''',\xi',\xi'',h) = e^{ih(-\langle D_{y'''},D_{\xi'''}\rangle -\langle D_{\bar x},D_{\bar \xi}\rangle)}b(x,\xi)c(y''',\xi',\xi'')|_{y'''=x''',\,x'=x''=\xi'''=0}.
	\]
\end{lem}

\begin{proof}
	This follows from the Fourier inversion formula and Lemma \ref{lem:gauss}.
\end{proof}

We now define our class of compactly microlocalized paired Lagrangians
in the model case of nested conormal distributions.

\begin{defi}
	We say that  $u \in I^{l,\COMP}_{h}(\RR^{m};{\Lambda}_0,{\Lambda}_1)$ if $\supp u$ and $\WF(u)$ are compact, and 
	\begin{equation} \label{eq:pairedlagrangiandefi}
	u(x) = (2\pi h)^{-3m/4 - k/2 + d/2}\int e^{\sfrac{i}{h}(\left<x',\xi'\right> + \left<x'',\xi''\right>)} a(x,\xi',\xi'')\, d\xi'd\xi''
	\end{equation}
	with $a \in S^{l-k/2,-\infty}_h$.
\end{defi}

Even when $l = -\infty$, elements of $I^{-\infty,\COMP}_h(\RR^m;\Lambda_0,\Lambda_1)$ are not residual in the sense of $\mathcal{O}(h^\infty)$ remainders:

\begin{lem} \label{lem:residualpairedlagrangian}
The following properties are satisfied.
\begin{enumerate} \itemsep6pt 
	\item 	$I^{-\infty,\COMP}_h(\RR^m;\Lambda_0,\Lambda_1) = I^\COMP_h(\RR^m;\Lambda_1)$. 
	
	\item If $l$ is fixed, then $h^\infty I^{l,\COMP}_h(\RR^m; \Lambda_0,\Lambda_1) = h^\infty \CcI(\RR^m)$.
\end{enumerate}
\end{lem}
\begin{proof} 
\begin{inparaenum} \item 
We can write $u \in I^{-\infty,\COMP}_h(\RR^m;\Lambda_0,\Lambda_1)$ in the form \eqref{eq:pairedlagrangiandefi} with $a \in S^{-\infty,-\infty}_h$. This implies that 	\[
	a(x,\xi,h) = b(x,\xi'/h,\xi'',h)
	\]
	where $b(x,\Xi,h)$ is rapidly decaying in $\Xi \in \RR^{m-d}$, uniformly in $h$. Making the change of variables $\eta = \xi'/h$ and performing the $\eta$ integral in the definition of $u$,
	\[
	u =  (2\pi h)^{-3m/4 + k/2 + d/2}\int e^{\sfrac ih \left<x'',\xi''\right>}\tilde{b}(x,\xi'',h) \,d\xi'',
	\]
	where $\tilde{b} \in \CcI(\RR^{m};\mathcal{S}(\RR^{d-m-k}))$ uniformly in $h$. It now suffices to compare the power 
	\[
	- 3m/4 + k/2 + d/2 = - (m-d-k)/2 - m/4
	\]
	to the usual Lagrangian order convention (which is $m-d-k$ phase variables in $m$ dimensions) to see that $u \in  I^{\COMP}_{h}(\RR^m; \Lambda_1)$. The converse inclusion is obvious.

	\item  
	This is just the observation that $h^\infty S^{r,-\infty}_h = h^\infty S^{-\infty}(\RR^m;\RR^{m-d})$ for any $r$.
\end{inparaenum}
\end{proof}

If $a$ were a semiclassical symbol, then the wavefront set of $u$ given by \eqref{eq:pairedlagrangiandefi} would be contained in $\Lambda_0$. However, in this case the weaker symbolic properties of $a \in S^{r,-\infty}_h$ can generate additional singularities. We define the essential support of $a$ as usual,
\[
(\esssupp a )^\complement = \{(x,\xi',\xi''): a \in h^\infty S^{-\infty} \text{ near } (x,\xi',\xi'')\},
\]
where $h^\infty S^{r,-\infty}_h = h^\infty S^{-\infty}(\RR^m;\RR^{m-d})$ for each $r$, as was already observed in the proof of Lemma \ref{lem:residualpairedlagrangian}.

\begin{lem} \label{lem:pairedlagrangianWF}
	If $u \in I^{l,\COMP}_{h}(\RR^{m};{\Lambda}_0,{\Lambda}_1)$ is given by \eqref{eq:pairedlagrangiandefi}, then 
	\[
	\WF(u) \subset  \{ (x,\xi) \in N^*S_1 \cup N^*S_0): (x,\xi',\xi'') \in \esssupp a\}.
	\]
\end{lem}
\begin{proof}
Let $\psi \in \CcI(\RR^m)$, and write $\psi u$ in the form \eqref{eq:reducedmodelpairedlagrangian}, where the amplitude $\tilde{a}$ is given by 
\[
\tilde a(x''',\xi',\xi'') = e^{-ih\langle D_{\bar x},D_{\bar \xi}\rangle}\psi(x)a(x,\xi',\xi'')|_{x'=x''=0}.
\]
Thus we can write $\tilde{a} = a_1 + a_2$, where $a_1 = 0$ if $(0,0,x''') \notin \supp \psi$, and $a_2 \in S^{-\infty,-\infty}_h$. This gives rise to a corresponding decomposition $u = u_1 + u_2$. By the first part of Lemma \ref{lem:residualpairedlagrangian}, $\WF(u_2) \subset \Lambda_1 = N^*S_1$.
On the other hand,
\[
\mathcal{F}_h(\psi u_1)(\eta) =  (2\pi h)^{m/4 - k/2 - d/2}\int e^{\sfrac i h \langle x''',\eta'''\rangle} a_1(x''',\eta) \, dx''',
\]	
so if $\eta'''_0 \neq 0$, integrating by parts using the operator $L =
|\eta'''|^{-2} \eta'''\cdot (hD_{x'''})$ shows that
$\mathcal{F}_h(\psi u_1)(\eta) = \mathcal{O}(h^\infty)$ in a
neighborhood of $\eta_0$. Therefore we find that $\WF (u_1) \subset
\{\eta'''=0\}$ and lies only over an arbitrarily small neighborhood of
$\{(0,0,x'''): x''' \in \supp \psi\},$ hence is in a small neighborhood
of $N^* S_0.$  Thus have have shown that
\[\WF(u) \subset  N^*S_0 \cup N^*S_1.
\] 
On the other hand, $\WF(u) \subset \{(x,\xi): (x,\xi',\xi'') \in \esssupp a\}$ by the second part of Lemma \ref{lem:residualpairedlagrangian}.
\end{proof}

It follows from Lemma \ref{lem:pseudodifferentialaction} that any properly supported $B \in \Psi^0_h(\RR^m)$ preserves $I^{l,\COMP}_{h}(\RR^{m};{\Lambda}_0,{\Lambda}_1)$. Moreover, if $B = b(x,hD)$ with $b$ a total symbol for $B$, we can write $Bu$ in the form \eqref{eq:Bu}, where
\begin{equation} \label{eq:pairedlagrangianesssupp}
\esssupp \tilde c \subset \{(x,\xi',\xi'') \in \esssupp a: (x,\xi) \in \WF(B))\}
\end{equation}
As a consequence, we can always write $u \in I^{l,\COMP}_h(\RR^m;\Lambda_0,\Lambda_1)$ in the form \eqref{eq:pairedlagrangiandefi}, where $a \in S^{l-k/2,\COMP}_h$ has compact support, modulo an $h^\infty \CcI(\RR^m)$ remainder.
We can also make Lemma \ref{lem:pairedlagrangianWF} more precise by microlocalizing individually to each of the two Lagrangians  $\Lambda_0,\Lambda_1$ carrying possible wavefront set.

\begin{lem} \label{lem:individuallagrangian}
	The following hold for $u \in I^{l,\COMP}_{h}(\RR^m;\Lambda_0,\Lambda_1)$ and $B \in \Psi^0_h(\RR^n)$ of proper support.
	\begin{enumerate} \itemsep6pt
		\item If $\WF(B) \cap \Lambda_1 = \emptyset$, then $Bu\in h^{-l}I^{\COMP}_{h}(\RR^m;\Lambda_0)$.
		\item If $\WF(B) \cap \Lambda_0 = \emptyset$, then $ Bu\in I^{\COMP}_{h}(\RR^m;\Lambda_1)$.
	\end{enumerate}
\end{lem}
\begin{proof}
	Since  $\WF(u) \subset \Lambda_0 \cup \Lambda_1$, we can assume that $B$ is microlocalized near $\Lambda_0 \setminus \Lambda_1$ in the first case, and $\Lambda_1 \setminus \Lambda_0$ in the second case.

	\begin{inparaenum}[(1)]
		\item By \eqref{eq:pairedlagrangianesssupp}, we may assume that $|\xi'| \geq \varepsilon$ on $\supp a$. It follows that the symbol of $Bu$ is an honest semiclassical symbol, and hence $Bu$ is Lagrangian with respect to $\Lambda_0$. It remains to check the overall power of $h.$
		
		\item Again by \eqref{eq:pairedlagrangianesssupp}, we may assume that $|x'| \geq \varepsilon$ on $\supp a$, at which point we proceed as in Lemma \ref{lem:pairedlagrangianWF}.
	\end{inparaenum} 
\end{proof}

\subsection{Change of variables} \label{subsect:coordinateinvariance}
In this section we show that $I^{l,\COMP}_h(\RR^m;\Lambda_0,\Lambda_1)$ is invariant under a diffeomorphism 
\[
\kappa : \RR^m \rightarrow \RR^m
\]
preserving $S_1$ and $S_0$, and define a principal symbol. To simplify matters, we work with half-densities. Thus, given 
\begin{equation} \label{eq:ukappa}
u_\kappa =  (2\pi h)^{-3m/4 - k/2 + d/2}\int e^{\sfrac{i}{h}(\left<x',\xi'\right> + \left<x'',\xi''\right>)} a_\kappa (x''',\xi',\xi'')\, d\xi'd\xi'',
\end{equation}
with $a_\kappa \in S^{l-k/2,-\infty}_h$, we transform $u_\kappa$ according to $ u = |\det \kappa'|^{1/2} (\kappa^*u_\kappa)$. We write $\kappa = (\kappa_1,\kappa_2,\kappa_3)$ relative to the splitting $x= (x',x'',x''')$, and denote the Jacobian by
\[
\kappa' = \begin{pmatrix} \kappa_{11}' & \kappa_{12}' & \kappa_{13}' \\  \kappa_{21}' & \kappa_{22}' & \kappa_{23}' \\
\kappa_{31}' & \kappa_{32}' & \kappa_{33}' 
\end{pmatrix}.
\]
Since $\kappa$ preserves $S_1$ and $S_0$, we see that $\kappa_{11}'$ and $\kappa_{22}'$ are nonsingular at points $(0,0,x''')$, and that $\kappa_{21}',\kappa_{13}',\kappa_{23}'$ vanish at such points.  Let us write
\[
\kappa(x) = (\psi_{11}(x)x' + \psi_{12}(x)x'', \psi_{22}(x)x'', \kappa_3(x)),
\]
By Lemmas \ref{lem:pairedlagrangianWF}, \ref{lem:individuallagrangian}
and a partition of unity, we can assume without loss of generality
that $\psi_{11}$ and $\psi_{22}$ are nonsingular throughout the
support of $u$, since the invariance properties of $u$ away from $S_0$
are well known. Arguing precisely as in \cite[Theorem
18.2.9]{hormander1994analysis}, we obtain the following.

\begin{lem}
	If $\kappa$ and $u_\kappa \in I^{l,\COMP}(\RR^m;\Lambda_0,\Lambda_1)$ are as above, then $u =  |\det \kappa'|^{1/2} (\kappa^*u_\kappa)$ is of the form \eqref{eq:pairedlagrangiandefi} with an amplitude $a \in S^{l-k/2,-\infty}_h$.
\end{lem}

Indeed, $a$ is given by the expression 
\begin{align*}
a(x,\xi',\xi'') &=  a_\kappa(\kappa_3(x),{^t}\psi_{11}(x)^{-1}\xi',{^t}\psi_{22}(x)^{-1}(\xi'' - {^t}(\psi_{11}(x)^{-1}\psi_{12}(x))\xi')) \\&\times |\det\kappa'(x)|^{1/2}|\det \psi_{11}(x)|^{-1} |\det \psi_{22}(x)|^{-1},
\end{align*}
and it is easy to see that $a \in S^{l-k/2,-\infty}_h$. Of course the $(x',x'')$ dependence can be eliminated as in \eqref{eq:reducedbase}.

To define the principal symbol along $\Lambda_1$, consider $u_\kappa$ of the form 
\[
u_\kappa =  (2\pi h)^{-3m/4 - k/2 + d/2}\int e^{\sfrac{i}{h}(\left<x',\xi'\right> + \left<x'',\xi''\right>)} a_\kappa (x',x''',\xi',\xi'')\, d\xi'd\xi''.
\]
Compared with \eqref{eq:ukappa}, we are not assuming that $a_{\kappa}$ in necessarily independent of $x'$. We associate to $u_\kappa |dx|^{1/2}$ the half-density
\begin{equation} \label{eq:halfdensitysymbol}
b_\kappa(x',x''',\xi'')|dx'|^{1/2}|dx'''|^{1/2}|d\xi''|^{1/2},
\end{equation}
where
\[
b_\kappa(x',x''',\xi'') = (2\pi h)^{-k}\int e^{\sfrac{i}{h}\left<x',\xi'\right>} a_\kappa(x',x''',\xi',\xi'')\, d\xi'.
\]
When $x'\neq0$ (so away from $\Lambda_0 \cap \Lambda_1$), this is a representative of the principal symbol of $u_\kappa |dx|^{1/2}$ as a half-density valued element of $I^{\COMP}_h(\RR^m;\Lambda_1)$. 

\begin{defi}
We say that $b(x',x''',\xi'') \in \CcI(\RR^{m-d}_{x',x'''}\times \RR^{m-d-k}_{\xi''})$ is in $S^{l,\COMP}_{\Lambda_1}$ if there exists $a(x',x''',\xi',\xi'') \in S^{l-k/2,\COMP}_h$ such that
\[
b(x',x''',\xi'') = (2\pi h)^{-k}\int e^{\sfrac i h \langle x',\xi' \rangle}a(x',x''',\xi',\xi'') \, d\xi'
\]
modulo $h^\infty \CcI(\RR^{m-d}\times \RR^{m-d-k})$.
\end{defi}
\blue{That this class of symbols degenerates at $x'=0,$ on $\Lambda_0 \cap
\Lambda_1,$ can be seen by differentiating in $x'$: there is a term
with a factor $\xi'/h$ arising from differentiation of the phase.  Away from
$x'=0$ we may integrate by parts with respect to the operator $(h/x') D_{\xi'},$
which, falling on the amplitude $a$, produces a factor
$\langle \xi'/h\rangle^{-1}$ that ensures the resulting amplitude enjoys
the same estimates as $a.$  This strategy fails at $x'=0,$ where
the order of $a$ is effectively raised to $l+1.$  The
singularity only arises in $h \to 0$ asymptotics, though: $b$ is
smooth for every positive $h$ since $a$ is compactly supported in $\xi.$}
\blue{\begin{rem}
If we apply this construction to the example $a$ from
Remark~\ref{rem:symbolexample} (which does have compact support in the fiber
variables, albeit not in the base), we obtain
$$
b(x',\xi'') = \chi(\xi'') (2\pi h)^{-1} \int_0^\infty
e^{ix'\xi'/h}\psi(\xi'/h) \chi(\xi') \, d\xi'\equiv b_0(x',\xi'')-b_1(x',\xi'')
$$
where
$$\begin{aligned}
b_0(x',\xi'')&=\chi(\xi'') (2\pi h)^{-1} \int_0^\infty
e^{ix'\xi'/h}\chi(\xi') \, d\xi',\\
b_1(x',\xi'')&=\chi(\xi'') (2\pi h)^{-1} \int_0^\infty
e^{ix'\xi'/h}(1-\psi)(\xi'/h) \, d\xi'
\end{aligned}
$$
(dropping a factor of $\chi(\xi')$ in $b_1$ since it is moot for $h$
small).  Changing variables to $\eta=\xi'/h$ in the integral, we find that $b_1(x', \xi'') \in \CI$
is independent of $h.$  By contrast,
$$
b_0(x',\xi'')=(2 \pi)^{-1/2}  h^{-1} \chi(\xi'') \mathcal{F}^{-1}(H \chi) (x'/h)
$$
where $H$ denotes the Heaviside function and $\mathcal{F}$ is the
ordinary (non-semiclassical) Fourier transform.   The distribution
$\mathcal{F}^{-1}(H \chi)(y)$ is $\CI$ and for large $\lvert y \rvert$
differs from $i (2\pi)^{-1/2} y^{-1}$ by
a rapidly decreasing function, hence in the asymptotic regime where $x'/h\to\infty,$ $b_0\sim (i/2\pi)
\chi(\xi'')/x'.$ (Formally taking leading order terms in an expansion in
$(x'/h)^{-1}$ is, in effect, our principal symbol construction.)
For $x'/h$ fixed, on the other hand, $b_0$ blows up like a multiple of
$h^{-1}$ as $h
\to 0.$
  \end{rem}}
Observe that $S^{l,\COMP}_{\Lambda_1}$ is itself a degenerate version of the paired Lagrangian distributions we have been studying.
Under $\kappa$, we see that $b_\kappa$ is transformed to
\begin{equation}\label{transformed}
b(x',x''',\xi'') =  (2\pi h)^{-k} \int e^{\sfrac i h \langle x',\xi'\rangle}\big ( e^{-ih\langle D_{x''}, D_{\xi''}\rangle }\tilde{b}(x',x'',x''',\xi',\xi'')|_{x''=0} \big) \, d\xi',
\end{equation}
where we have defined
\begin{align*}
  \tilde{b}(x',x'',x''',\xi',\xi'') &= e^{\sfrac i h \langle \psi_{11}(x)^{-1} \psi_{12}(x) x'',\xi' \rangle} a_\kappa(\kappa_1(x), \kappa_3(x),{^t}\psi_{11}^{-1}(x)\xi',{^t}\psi_{22}^{-1}(x)\xi'')  \\ &\times |\det \kappa'(x)|^{1/2} |\det \psi_{22}(x)|^{-1} |\det \psi_{11}(x)|^{-1}.
\end{align*}
The phase factor $e^{\sfrac i h \langle \psi_{11}(x)^{-1} \psi_{12}(x) x'',\xi' \rangle}$
is harmless when differentiating $e^{-ih\langle D_{x''}, D_{\xi''}\rangle }\tilde{b}$ in $(x',x''',\xi',\xi'')$, since the result is evaluated at $x''=0$, and in particular
\[
e^{-ih\langle D_{x''}, D_{\xi''}\rangle }\tilde{b} \in S^{l-k/2,\COMP}_h.
\]
On the other hand, the phase factor does appear in higher order expansion. Importantly,
\[
e^{-ih\langle D_{x''}, D_{\xi''} \rangle }\tilde b(x',x'',x''',\xi',\xi'')|_{x''=0} = \tilde b(x',0,x''',\xi',\xi'') + h S^{l-k/2+1,\COMP}_h.
\]
This shows that the equivalence class of $b_{\kappa}(x',x''',\xi'')|dx'|^{1/2}|dx'''|^{1/2}|d\xi''|^{1/2}$ is well defined in $S^{l,\COMP}_{\Lambda_1}/hS^{l+1,\COMP}_{\Lambda_1}$, since the pullback of $b_\kappa$ as a half-density is precisely
\[
(2\pi h)^{-k} \int e^{\sfrac i h \langle x',\xi'\rangle} \tilde b (x',0,x''',\xi',\xi'') \, d\xi'.
\]
We have proved the following:

\begin{prop}
	The principal symbol 
	\[
	\sigma_h^{\Lambda_1}(u_\kappa |dx|^{1/2}) = b_{\kappa}(x',x''',\xi'')|dx'|^{1/2}|dx'''|^{1/2}|d\xi''|^{1/2}\in S^{l,\COMP}_{\Lambda_1}/hS^{l+1,\COMP}_{\Lambda_1}
	\] 
	is well defined. 
\end{prop}

\subsection{Pseudodifferential operators with singular symbols} \label{subsect:pseudosingularsymbol}
In this section we discuss a calculus of pseudodifferential operators with singular symbols. Let $X$ be an $n$-dimensional manifold, and $Y\subset X$ a codimension $k$ submanifold. Consider an operator $A$ with Schwartz kernel 
\[
K_A \in I^{l,\COMP}_h(X\times X; N^*((X\times Y)\cap \diag), N^*\diag).
\]
Since $\supp K_A$ and $\WF(K_A)$ are compact by assumption, it follows that 
\[
K_A : \CmI(X) \rightarrow \CcI(X),
\]
and $K_A$ is $h$-tempered.

By the coordinate invariance discussed in Section \ref{subsect:coordinateinvariance}, it suffices to construct this calculus on $X = \RR^n$, where $Y = \{x'=0\}$ for an appropriate splitting of coordinates 
\[
x = (x',x'') \in \RR^k \times \RR^{n-k}.
\] 
If $(x,y,\xi,\eta)$ are the corresponding coordinates on $T^*(\RR^n \times \RR^n)$, we work with the Lagrangian pair 
\begin{equation} \label{lagpair}
\Lambda_1 = \{x=y, \, \eta = -\xi\}, \quad \Lambda_0 = \{x' = y' =0, \, x'' = y'', \, \eta'' = -\xi''\}.
\end{equation}
Working modulo $h^\infty \CcI(\RR^{2n})$, we can write $K_A \in I^{l,\COMP}_{h}(\RR^{2n};\Lambda_0,\Lambda_1)$ as a left quantization
\begin{equation} \label{eq:singularleftquant}
K_A = (2\pi h)^{-n-k}\int e^{\sfrac{i}{h}(\langle x-y, \xi \rangle + \langle x',\eta'\rangle)}a(x'',\eta',\xi)\,d\eta' d\xi,
\end{equation}
where $a\in S^{l-k/2,\COMP}_h(\RR^{n-k}; \RR^k; \RR^{n-k})$. This parametrization arises by using coordinates $(z',z'',x',x'')$ on $\RR^{2n}$, where $z = x-y$; thus 
\[
\Lambda_1 = N^*\{z=0\}, \quad \Lambda_0 = N^*\{x'=0, z=0\}.
\] 
Alternatively, we can use coordinates $(z',z'',y',y'')$, so that $K_A$ can also be written as a right quantization
\begin{equation} \label{eq:singularrightquant}
K_A = (2\pi h)^{-n-k}\int e^{\sfrac{i}{h}(\langle x-y, \xi \rangle + \langle y',\eta'\rangle)}a(y'',\eta',\xi)\,d\eta' d\xi,
\end{equation}
with $a \in S^{l-k/2,\COMP}_h(\RR^{n-k}; \RR^k; \RR^{n-k})$. The principal symbol $\sigma^{\Lambda_1}_h(A)$ of $A$ along $\Lambda_1 = N^*\diag$, which we define simply to be $\sigma_h^{\Lambda_1}(K_A)$, is 
\[
(2\pi h)^{-k} \int e^{\sfrac i h \langle x',\eta' \rangle} a(x'',\eta',\xi) \, d\eta' 
\]
in $S^{l,\COMP}_{\Lambda_1}/hS^{l+1,\COMP}_{\Lambda_1}$. As usual we use the canonical symplectic density on $N^*\diag$ to identify functions with half-densities. 

Next we consider composition of two operators whose Schwartz kernels
are of the form \eqref{eq:singularleftquant}. The proof we give is
closely based on \cite[Proposition 5.8]{de2014diffraction}. Because of
certain logarithmic terms, in some cases there appear arbitrarily
small losses in the order of the composition. Since these losses are
acceptable, we will not explicate when they can be avoided; for a more
precise account, see \cite[Proposition 5.8]{de2014diffraction}.
\blue{We remind the reader that, the homogeneous
  paired Lagrangians considered in \cite{de2014diffraction}, these
  semiclassical operators have smooth Schwartz kernels for all $h>0,$
  so that the composition of two singular pseudodifferential operators
  is always a well-defined operator family.  The constraints on orders only arise in
  interpreting the result as another element in the calculus.}

\begin{prop} \label{prop:singularcomposition}
	Given $K_A \in I^{l,\COMP}_{h}(\RR^{2n};\Lambda_0,\Lambda_1)$ and $K_B \in I^{l',\COMP}_{h}(\RR^{2n};\Lambda_0,\Lambda_1)$ let 
	\[
	L > \max(l,l',l+l'+k/2).
	\] 
	If $l + l' < 0$, then $K_{A B} \in I^{L,\COMP}_{h}(\RR^{2n};\Lambda_0,\Lambda_1)$. Furthermore, if $\delta \in (0,1]$ is such that $l+l' < -\delta$, then 
	\[
	\sigma_h^{\Lambda_1}(A B) = \sigma_h^{\Lambda_1}(A)\cdot \sigma_h^{\Lambda_1}(B)
	\] 
	in $S_{\Lambda_1}^{L,\COMP}/h^{\delta}S_{\Lambda_1}^{L+\delta,\COMP}$.
\end{prop}
\begin{proof}
	As remarked above, the proof is essentially the same as in \cite[Proposition 5.8]{de2014diffraction}. We write $A$ in the form \eqref{eq:singularleftquant} with amplitude $a(x'',\eta',\xi)$, and $B$ in the form \eqref{eq:singularrightquant} with amplitude $b(y'',\mu',\xi)$. Now 
	\[
	\mathcal{F}_h(Bu)(\xi) = (2\pi h)^{-k}\int e^{\sfrac{i}{h} (-\langle y,\xi \rangle + \langle y',\mu') \rangle}b(y'',\mu',\xi) u(y) \, dy d\mu',
	\]
	and hence 
	\[
	K_{A B} = (2\pi h)^{-n-2k}\int e^{\sfrac{i}{h}(\langle x-y,\xi\rangle + \langle y',\mu'\rangle + \langle x',\eta'\rangle )}a(x'',\eta',\xi)b(y'',\mu',\xi) \,d\eta' d\mu' d\xi.
	\]
	Following \cite[Proposition 5.8]{de2014diffraction}, we make the change of variables 
	\[
	\nu' = \eta' + \mu', \quad \zeta' = \xi' -\mu', \quad \zeta'' = \xi'',
	\]
	leaving $\mu'$ unchanged (observe that $\xi''$ is being renamed for later convenience, but is otherwise unchanged).
	Rewriting the phase in terms of these new variables as $\left<x-y,\zeta\right> + \left<x',\nu'\right>$, it follows that
	\[
	K_{AB} = (2\pi h)^{-n-k} \int e^{\sfrac{i}{h}(\langle x''-y'',\zeta''\rangle + \langle x'-y',\zeta' \rangle + \langle x', \nu' \rangle )}c(x'',\nu',\mu',\xi'')\,d\nu' d\mu' d\xi'',
	\]
	where 
	\[
	c(x'',y'',\nu',\zeta',\zeta'') = (2\pi h)^{-k}\int  a(x'',\nu'-\mu',\mu' + \zeta', \zeta'')b(y'',\mu',\mu'+\zeta',\zeta'') \,d\mu'.
	\]
	It remains to verify that $c \in S^{L,\COMP}_h(\RR^{2(n-k)}; \RR^{k};\RR^{n})$. Observe that the integral defining $c$ is over a compact set, since $b$ has compact support in $\mu'$, and indeed $c$ is itself compactly supported.
	We begin by giving sup-norm bounds on $c$, observing that
	\[
	|c(x'',y'',\nu',\zeta',\zeta'')| \leq Ch^{-k}\int \left<(\nu'-\mu')/h\right>^{l-k/2}\left<\mu'/h\right>^{l'-k/2}\, d\mu'.
	\]
	First, suppose that $|\nu'| \geq h$, in which case $\left<\nu'/h\right>$ can everywhere be replaced by $|\nu'/h|$. We then consider the integral over four regions.

	\begin{inparaenum}
		\item $2|\mu'| \leq |\nu'|$. Here $\left<(\nu'-\mu')/h\right>$ is comparable to $\left<\nu'/h\right>$, so the integral over this region is bounded by
		\begin{align*}
		Ch^{-k}\left<\nu'/h\right>^{l-k/2}\int_{|\mu'| \leq (1/2)|\nu'|} \left<\mu'/h\right>^{l'-k/2}\, d\mu' \leq C|\nu'/h|^{l-k/2}(1 + |\nu'/h|^{l'+k/2 + \varepsilon})
		\end{align*}
		for any $\varepsilon>0$.
		
		\item $2|\nu' -\mu'| \leq |\nu'|$. Here $\left<\mu'/h\right>$ is comparable to $\left<\nu'/h\right>$, so as above the integral over this region is bounded by
		\[
		C|\nu'/h|^{l-k/2}(1 + |\nu'/h|^{l'+k/2 + \varepsilon})
		\]
		for any $\varepsilon>0$.

		\item $2|\nu'| \leq |\mu'|$. Here $\left<(\nu'-\mu')/h\right>$ is comparable to $\left<\mu'/h\right>$, so the integral is bounded by 
		\[
		Ch^{-k}\int_{|\mu'| \geq 2|\nu'|} \left<\mu'/h\right>^{l+l'-k}\, d\mu' \leq C|\nu'/h|^{l+l'},
		\]
		since $l+l' < 0$.

		\item $(1/2)|\nu'| \leq |\mu'| \leq 2|\nu'|$ and $|\nu'| \leq 2|\nu'-\mu'|$. Here $\left<(\nu'-\mu')/h\right>$ is comparable to $\left<\mu'/h\right>$, so the integral is bounded by 
		\begin{align*}
		Ch^{-k}\int_{(1/2)|\nu'| \leq |\mu'| \leq 2|\nu'|} \left<\mu'/h\right>^{l+l'-k}\, d\mu'\leq C|\nu'/h|^{l+l'},
		\end{align*}
		since $l+l' < 0$.
	
	\end{inparaenum}
	
	Thus, when $|\nu'| \geq h$, we conclude that 
	\[
	|c(x'',y'',\nu',\zeta',\zeta'')| \leq C(|\nu'/h|^{l+l'+\varepsilon} + |\nu'/h|^{l-k/2 +\varepsilon} + |\nu'/h|^{l'-k/2+\varepsilon})
	\]
	for any $\varepsilon >0$. On the other hand, if $|\nu'| \leq h$, then
	\[
	|c(x'',y'',\nu',\zeta',\zeta'')| \leq Ch^{-k}\int \left<\mu'/h\right>^{l+l'-k}\, d\mu' \leq C.
	\]
	provided that $l+l' < 0$. Bounds on the derivatives are established in precisely the same way. 
	
	It remains to prove the statement about the principal symbols. Note that the product $\sigma^{\Lambda_1}_h(A) \cdot \sigma^{\Lambda_1}_h(B)$ is the product 
	\[
	(\mathcal{F}_h')^{-1}a(x'',x',\zeta',\zeta'')\cdot  (\mathcal{F}_h')^{-1}b(y'',y',\zeta',\zeta'')|_{x=y},
	\] 
	where the first inverse Fourier transform takes $\eta' \mapsto x'$, and the second takes $\mu' \mapsto y'$. Thus $\sigma^{\Lambda_1}_h(A) \cdot \sigma^{\Lambda_1}_h(B)$ at $(x'',\nu',\zeta',\zeta'')$ is the inverse Fourier transform of a convolution,
	\[
	(2\pi h)^{-k}\int e^{\sfrac i h \langle x',\nu'\rangle} \int  a(x'',\nu'-\mu',\zeta',\zeta'')b(x'',\mu',\zeta',\zeta'') \, d\mu' d\nu'.
	\]
	On the other hand, $\sigma_h^{\Lambda_1}(A B)$ is given by
	\[
	(2\pi h)^{-k} \int e^{\sfrac i h \langle x',\nu'\rangle} \int a(x'',\nu'-\mu',\mu' + \zeta', \zeta'')b(x'',\mu',\mu'+\zeta',\zeta'') \,d\mu' \, d\nu'
	\]
	The only difference between these expressions is that $\zeta'$ in the first is replaced by $\mu'+\zeta'$ in the second. Taylor expanding $a$ at $\zeta'$ in the second expression, we can write
	\begin{align*}
	a(x'',\nu'-\mu',\mu' + \zeta', \zeta'') &=  a(x'',\nu'-\mu',\zeta', \zeta'') \\ + &h\cdot \int_{0}^1 \left<\mu'/h,\partial_{\zeta'} a(x'',\nu'-\mu',\zeta' + t\mu', \zeta'')\right> \, dt.
	\end{align*}
	The integral on the right hand is estimated by
	\[
	Ch^\delta \left<\mu'/h\right>^{\delta}\left<(\nu'-\mu')/h\right>^{l-k/2},
	\]
	for any $\delta \in [0,1]$, with similar bounds for its
        derivatives since $\mu'$ is bounded. A similar expansion also holds for $b(x'',\mu',\mu'+\zeta',\zeta'')$ in terms of $b(x'',\mu',\mu',\zeta'')$ modulo a remainder bounded by
	\[
	Ch^\delta \left<\mu'/h\right>^{l'+\delta-k/2},
	\]
	along with derivative bounds. In particular,
	\[
	a(x'',\nu'-\mu',\mu' + \zeta', \zeta'')b(x'',\mu',\mu'+\zeta',\zeta'') = a(x'',\nu'-\mu', \zeta', \zeta'')b(x'',\mu',\zeta',\zeta'')
	\]
	modulo a remainder bounded by
\[
 Ch^\delta \left<(\nu'-\mu')/h\right>^{l-k/2}\left<\mu'/h\right>^{l'+\delta-k/2}
\]
for any $\delta \in [0,1]$, along with derivative bounds. As above, if $l+l' < -\delta$, then the resulting integral in $(\mu',\nu')$ yields an element of $h^\delta S^{L+\delta}_{\Lambda_1}$, since
\[
L+\delta > \max(l,l'+\delta,l+l'+\delta-k/2).
\]
Arguing similarly for the derivatives completes the proof.
\end{proof}

We also need uniform $L^2$ mapping properties of operators with singular symbol. It suffices to consider the local situation $K_A \in I^{l,\COMP}_h(\RR^{2n};\Lambda_0,\Lambda_1)$. 

\begin{lem} \label{lem:singularsymbolL2bounded}
  Let $\Lambda_0,$ $\Lambda_1$ be given by \eqref{lagpair}.  If $K_A \in I^{l,\COMP}_h(\RR^{2n};\Lambda_0,\Lambda_1)$ and $s \geq 0$ is such that $l -s < -k/2$, then 
	\[
	\| A \|_{L^2(\RR^n) \rightarrow L^2(\RR^n)} \leq Ch^{-s}.
	\]
	In particular, if $l < -k/2$, then $A$ is uniformly bounded.
\end{lem}
\begin{proof}
	Using the left quantization \eqref{eq:singularleftquant}, we may assume that $K_A$ is parametrized by
	\begin{equation} \label{eq:singularsymbolL2bounded}
	K_A = (2\pi h)^{-n-k}\int e^{\sfrac{i}{h}(\langle x',\eta'+\xi'\rangle-\langle y',\xi'\rangle + \langle x''-y'',\xi''\rangle)}a(x'',\eta',\xi',\xi'')\, d\xi' d\xi'' d\eta',
	\end{equation}
	where $a \in S^{l-k/2,\COMP}_h(\RR^{n-k}; \RR^k ; \RR^{n-k})$. We bound this operator on $L^2(\RR^n)$ by viewing it as a pseudodifferential operator on $\RR^{n-k}_{x''}$ with values in uniformly bounded operators on $L^2(\RR^k_{x'})$. Thus we write
	\[
	K_A = (2\pi h)^{-n+k}\int e^{\sfrac{i}{h}\langle x''-y'',\xi''\rangle}\mathcal{A}(x'',\xi'')\, d\xi'',
	\]
	where for each $(x'',\xi'')$ the operator $\mathcal{A}(x'',\xi'')$  has kernel
	\[
	K_{\mathcal{A}}(x',y';x'',\xi'') = (2\pi h)^{-2k}\int e^{\sfrac{i}{h}(\langle x',\zeta'\rangle - \langle y', \xi' \rangle)} a(x'',\zeta'-\xi',\xi',\xi'')\,d\zeta' d\xi'.
	\]
	We now show that 
	\[
	\mathcal{A}(x'',\xi'') \in S^0(\RR^{n-k}; \RR^{n-k}; \mathcal{L}(L^2(\RR^{k}))).
	\]
	Because $\mathcal{A}(x'',\xi'')$ has compact support in $(x'',\xi'')$, to prove the lemma it suffices to show
	\[
	\| D_{x''}^{\alpha} D_{\xi''}^\beta \mathcal{A}(x'',\xi'') \|_{L^2(\RR^k) \rightarrow L^2(\RR^k)} \leq Ch^{-s}
	\]
	for all multiindices $\alpha,\beta$. On the other
        hand,
	\[
	\| D_{x''}^{\alpha} D_{\xi''}^\beta \mathcal{A}(x'',\xi'') \|_{L^2(\RR^k) \rightarrow L^2(\RR^k)} = \| \mathcal{F}'_h( D_{x''}^{\alpha} D_{\xi''}^\beta \mathcal{A}(x'',\xi'') )(\mathcal{F}')^{-1}_h \|_{L^2(\RR^k)\rightarrow L^2(\RR^k)},
	\]
	since $h^{-k/2}\mathcal{F}_h'$ is unitary ($\mathcal{F}_h'$
        denotes semiclassical Fourier transform only in the primed variables). The conjugated operator on the right, which we denote by $\widehat{\mathcal{A}}(x'',\xi'')$, has kernel
	\[
(\zeta',\xi') \mapsto  (2\pi h)^{-k} D_{x''}^\alpha D_{\xi''}^\beta a(x'',\zeta'-\xi',\xi',\xi'').
	\]
	Since $a(x'',\eta',\xi',\xi'')$ has compact support in $\eta'$, for any $s \geq 0,$
	\[
	(2\pi h)^{-k}|a(x'',\zeta'-\xi',\xi',\xi'')| \leq Ch^{-k}h^{-s} \left<(\zeta'-\xi')/h\right>^{l-k/2-s}.
	\]
	By Schur's lemma, it follows that $h^s \widehat{\mathcal{A}}(x'',\xi'')$ is uniformly bounded on $L^2(\RR^k)$ provided that $l-k/2-s < -k$, completing the proof.
\end{proof}

 \begin{rem} \label{rem:SSweakerbounded}
Lemma \ref{lem:singularsymbolL2bounded} is equally valid if $K_A$ is given by the oscillatory integral \eqref{eq:singularsymbolL2bounded} where the amplitude has compact support in $\eta'$ uniformly with respect to the other variables, but is not of compact support in $(x'',\xi',\xi'')$, provided that bounds of the form
\[
|D_{x''}^\alpha D_{\xi''}^\beta a(x'',\eta',\xi',\xi'')| \leq C_{\alpha\beta}\left< \eta'/h \right>^{l-k/2}\left<\xi''\right>^{-|\beta|}
\]
are valid.
\end{rem}

\subsection{Homogeneous paired Lagrangian distributions}

We also need another class of paired Lagrangian distributions, which have wavefront set at fiber-infinity. Again, it will suffice to consider conormal bundles of nested submanifolds. Let $(x,\xi',\xi'')$ be coordinates on $\RR^{m} \times \RR^k \times \RR^{n} $. 

\begin{defi}
	We say that an $h$-dependent function $a = a(x,\xi',\xi'';h)\in \CI(\RR^{m}_x \times \RR^{k}_{\xi'} \times \RR^n_{\xi''})$  is in $S^{q,r}(\RR^{m}_x; \RR^{n}_{\xi''}; \RR^k_{\xi'})$ if it satisfies the product-type estimates
	\[
	| D_{\xi'}^\alpha D_{\xi''}^\beta D_x^\gamma a(x,\xi',\xi'')| \leq C_{\alpha\beta \gamma}\left<(\xi',\xi'')\right>^{q-|\beta|} \left<\xi'\right>^{r-|\alpha|}
	\]
	for all multiindices $\alpha,\beta,\gamma$.
\end{defi}

We use the same notation as in Section \ref{subsect:nestedconormal}, so that $S_0 \subset S_1 \subset \RR^m$, as well as $\Lambda_0 = N^*S_0$ and  $\Lambda_1 = N^*S_1$. We consider oscillatory integrals of the form
\begin{equation} \label{eq:modelhomogeneous}
(2\pi h)^{d-m}\int e^{\sfrac{i}{h} (\langle x',\xi' \rangle + \langle x'',\xi'' \rangle )}a(x,\xi',\xi'') \, d\xi' d\xi'',
\end{equation}
where $a \in S^{q,r} = S^{q,r}(\RR^m_x; \RR^{m-k-d}_{\xi''}; \RR^k_{\xi'})$.

\begin{defi}
	We say that $u \in I^{p,l}_{\hmg,c}(\RR^{m};{\Lambda}_0,{\Lambda}_1)$ if $\supp u$ is compact, and $u$ is of the form \eqref{eq:modelpairedlagrangian} for some $a \in S^{q,r}$, where $q = p - m/4 + k/2 + d/2$ and $r = l-k/2$.
\end{defi}

This is a direct semiclassical adaptation of the paired Lagrangian distributions studied in \cite{de2014diffraction}, and for this reason we take various facts for granted that were explicitly demonstrated for the related space $I^{l,\COMP}_h(\RR^m;\Lambda_0,\Lambda_1)$. We will need the following:

\begin{itemize} \itemsep6pt 
	\item Any $u$ of the form \eqref{eq:modelhomogeneous} can be written in terms of an amplitude $\tilde{a}(x''',\xi',\xi'')$ depending only on $x'''$ in the base variables.
	\item The space $I^{p,l}_{\hmg,c}(\RR^{m};{\Lambda}_0,{\Lambda}_1)$ is invariant under diffeomorphisms of $\RR^m$ preserving $S_1$ and $S_0$, which allows for the definition of $I^{p,l}_{\hmg}$ on a general manifold.
\end{itemize}
However, we will not need to develop any symbol calculus for this class of distributions.

In this context, $I^{p,l}_{\hmg}(\RR^{m};{\Lambda}_0,{\Lambda}_1)$ arises when multiplying $u \in I^{\COMP}_h(X;N^*Y)$ by $v\in I^{[\mu]}(Z)$, where $Y,Z$ are two transverse submanifolds of a manifold $X$. It suffices to consider the model case; thus we take $X = \RR^{m}$ with coordinates 
\[
(x',x'',x''') \in \RR^{d_1} \times \RR^{d_2} \times  \RR^{m-d_1-d_2},
\]
and then set 
\[
Y = \{x'=0\}, \quad Z = \{x''=0\}.
\] 
Thus $Y$ and $Z$ have codimension $d_1$ and $d_2$ in $\RR^{m}$, respectively, while $Y\cap Z$ has codimension $d_1+d_2$. 
\begin{lem} \label{lem:transversemult}
	If	$u \in I^{\COMP}_h(\RR^m;N^*Y)$ and $v\in I^{[\mu]}_c(Z)$, then
	\begin{align*}
	uv &\in I^{\mu + d_2/2,\COMP}_h(\RR^m; N^*(Y\cap Z), N^* Y)\\ &+ h^{-\mu-m/4+d_1/2}I^{\mu -m/4 + d_2/2,-\infty}_\hmg(\RR^m; N^*(Y\cap Z), N^* Z)
	\end{align*}
\end{lem} 
\begin{proof}
Since $\codim Y = d_1$, modulo a $\CcI(\RR^m)$ remainder we can write
	\[
	u = (2\pi h)^{-m/4 - d_1/2}\int e^{\sfrac{i}{h}\left<x',\xi'\right>} a(x,\xi')\, d\xi',
	\]
	where $a(x,\xi') \in \CcI(\RR_{x}^{m}\times \RR^{d_1}_{\xi'})$. On the other hand, modulo $\CcI(\RR^m)$, we can find a Kohn--Nirenberg symbol $b(x,\eta'')  \in S^{\mu}(\RR^{m}_{x};\RR^{d_2}_{\xi''})$ such that
	\begin{equation} \label{eq:vinsertcutoff}
	v = (2\pi h)^{-d_2} \int e^{\sfrac{i}{h}\left<x'',\xi''\right>} b(x,\xi''/h)\, d\xi''.
	\end{equation}
	Here we made the usual semiclassical change of variables $\eta'' =\xi''/h$. The product $uv$ is given by
	\begin{equation} \label{eq:insertcutoff}
	uv = (2\pi h)^{-m/4-d_1/2-d_2} \int  e^{\sfrac{i}{h}(\left<x',\xi'\right> + \left<x'',\xi''\right>)}a(x,\xi')b(x,\xi''/h)\, d\xi' d\xi''.
	\end{equation}
	Now insert a smooth cutoff function $\chi(\xi'')$ such that
        $\chi = 1$ near $\xi''=0$. Thus we may split $uv=w_0+w_1$ as a sum of two
        oscillatory integrals where $w_0$ has amplitude
        $\chi ab$, and $w_1$ has amplitude $(1-\chi)ab$. For the
        term $w_0$, let
	\[
	c_0(x,\xi',\xi'') = \chi(\xi'')a(x,\xi')b(x,\xi''/h).
	\]
	Thus $c_0 \in S^{\mu,\COMP}_h$, and $-m/4-d_1/2 - d_2 = -3m/4 - d_2/2 +\dim(Y\cap Z)/2$ since $\dim(Y\cap Z) = m- d_1 -d_2$. In particular,
	\begin{align*}
	w_0 &= (2\pi h)^{-m/4-d_1/2-d_2}\int e^{\sfrac{i}{h}(\left<x',\xi'\right> + \left<x'',\xi''\right>)} c_0(x,\xi',\xi'') \, d\xi' d\xi'' \\ &\in I^{\mu + d_2/2,\COMP}_h(\RR^m; N^*(Y\cap Z), N^* Y).
	\end{align*}
	For the second term $w_1$, observe that $|\xi''| \geq C_0$ on $\supp (1-\chi(\xi''))$ for some $C_0>0$. Let
	\[
	c_1(x,\xi',\xi'') = h^{\mu}(1-\chi(\xi''))a(x,\xi')b(x,\xi''/h).
	\]
	Since $c_1$ is in fact compactly supported in $\xi'$, we certainly have the symbol bounds 
	\[
	|D_{\xi'}^\alpha D_{\xi''}^\beta D_{x}^\gamma c_1(x,\xi',\xi'')| \leq C_{\alpha\beta\gamma N} \left<\xi'\right>^{-N}\left<(\xi',\xi'')\right>^{\mu-|\beta|}.
	\]
	This shows that
	\begin{align*}
	w_1 &= h^{-\mu} (2\pi h)^{-m/4+d_1/2}(2\pi h)^{-d_1-d_2}\int e^{\sfrac{i}{h}(\left<x',\xi'\right> + \left<x'',\xi''\right>)} c_1(x,\xi',\xi'') \, d\xi' d\xi'' \\ &\in h^{-\mu-m/4+d_1/2}I^{\mu -m/4 + d_2/2,-\infty}_\hmg(\RR^m; N^*(Y\cap Z), N^*Z)
	\end{align*}
	as desired.
\end{proof}
\begin{rem}\label{rem:oneterm}
If we assume that $(x,\xi'') \mapsto b(x,\xi''/h)$ in
\eqref{eq:insertcutoff} has compact support in $(x,\xi'')$, then we
are left with only a $w_0$ term in the proof above, i.e., an element
of $I^{\mu + d_2/2,\COMP}_h(\RR^m; N^*(Y\cap Z), N^* Y).$
  \end{rem}

  \begin{rem}\label{rem:smallsupport}
    In the notation of the proof of Lemma \ref{lem:transversemult},
    let $O \subset T^*\RR^m$ be an open neighborhood of
    $\WF(u) \cap N^*Y$. Then
    $w_0 \in I^{\mu + d_2/2,\COMP}_h(\RR^m; N^*(Y\cap Z), N^* Y)$ can
    always be chosen so that $\WF(w_0) \subset O$. Indeed, by Lemma
    \ref{lem:pairedlagrangianWF} and semiclassical wavefront set
    calculus,
	\begin{multline*}
	\WF(w_0) \subset (N^*Y \cap \WF(u)) \\ \cup  \{(0,0,x''',\xi',\xi'',0): (0,0,x''',\xi',0,0) \in \WF(u), \, \xi'' \in \supp \chi\}.
	\end{multline*}
	Now $\WF(w_0)$ is closed, and $\WF(w_0) \cap N^*Y \subset O$; since $O$ is open, the result follows by taking $\chi$ with sufficiently small support. Observe that this can be thought of as decomposing $v = v_0 + v_1$ itself into a sum, where we insert $\chi(\xi'')$ and $1-\chi(\xi'')$ into \eqref{eq:vinsertcutoff}. 
\end{rem}

We now return to the setting of Section
\ref{subsect:pseudosingularsymbol}: let $X$ be an $n$-dimensional
manifold, and $Y\subset X$ a codimension $k$ submanifold. We then
consider operators with Schwartz kernels
\[
K_A \in I^{p,l}_\hmg(X\times X; N^*((X\times Y)\cap \diag), N^*(X\times Y)).
\]
We also need to consider the case when $X\times Y$ is replaced with $Y\times X$.
Although $K_A$ is not compactly microlocalized, it nevertheless defines an $h$-tempered family of operators $A: \CI(X) \rightarrow \CmI(X)$.

As in Section \ref{subsect:pseudosingularsymbol}, it suffices work on $X = \RR^n$ with coordinates $x = (x',x'') \in \RR^k \times \RR^{n-k}$, where $Y = \{x'=0\}$. If $(x,y,\xi,\eta)$ are the corresponding coordinates on $T^*(\RR^n \times \RR^n)$, let 
\[
\Lambda_R = N^*\{y'=0\},\quad \Lambda_L = N^*\{x'=0\}.
\] 
We work with the Lagrangian pair 
\begin{equation}\label{lagpair2}
\Lambda_1 = \Lambda_R \text{ or } \Lambda_L, \quad \Lambda_0 = \{x'=y'=0, \, x'' = y'', \, \xi'' = -\eta''\}.
\end{equation}
For instance, if $\Lambda_1 = \Lambda_R$, then we can parametrize $K_A \in I^{p,l}_{\hmg}(\RR^{2n};\Lambda_0,\Lambda_1)$ by
\begin{equation} \label{eq:badparametrization}
K_A = (2\pi h)^{-n-k}\int e^{\sfrac{i}{h}(\langle x',\xi'\rangle-\langle y',\zeta'\rangle + \langle x''-y'',\xi''\rangle)}a(x'',\zeta',\xi',\xi'')\, d\xi' d\xi'' d\zeta',
\end{equation}
where now $a(y'',\zeta',\xi',\xi'')$ satisfies the symbol bounds
\[
| D_{\zeta'}^\alpha D_{\xi}^\beta D_{y''}^\gamma a(y'',\zeta',\xi',\xi'')| \leq C_{\alpha\beta \gamma}\left< \xi\right>^{l-n/2-|\beta|} \left< (\zeta',\xi) \right>^{p+(n-k)/2-|\alpha|}.
\]
We need uniform mapping properties of $A$, which can be deduced as in \cite[Proposition 5.14]{de2014diffraction}.

\begin{lem} \label{lem:badL2bounded}
	Let $\Lambda_0$, $\Lambda_1$ be defined by \eqref{lagpair2}.  Let $K_A \in I^{p,l}_{\hmg}(\RR^{2n};\Lambda_0,\Lambda_1)$ be of the form \eqref{eq:badparametrization}. If $p+l < -k/2$ and  $p < -n/2$, then 
	\[
	\| A \|_{L^2(\RR^n) \rightarrow L^2(\RR^n)} \leq Ch^{-k}.
	\]
	
\end{lem}
\begin{proof}
For concreteness, assume that $\Lambda_1 = \Lambda_R$; the same proof can be repeated for $\Lambda_L$. We argue as in Lemma \ref{lem:singularsymbolL2bounded}, viewing $A$ as a pseudodifferential operator on $\RR^{n-k}$ with an operator-valued symbol $\mathcal{A}(x'',\xi'')$ given by
	\[
	K_\mathcal{A}(x',y';y'',\xi'') =  (2\pi h)^{-2k}\int e^{\sfrac{i}{h}(\langle x',\zeta'\rangle - \langle y', \xi' \rangle)} a(x'',\zeta',\xi',\xi'')\,d\zeta' d\xi'.
	\]
	Conjugating by the Fourier transform as in Lemma \ref{lem:singularsymbolL2bounded}, the problem is reduced to showing that the operator with Schwartz kernel 
	\begin{equation} \label{eq:conjugatedkernel}
	(\zeta',\xi') \mapsto \left<\xi''\right>^{|\beta|} D_{y''}^{\alpha} D_{\xi''}^{\beta}a(x'',\zeta',\xi',\xi'')
	\end{equation}
	has uniformly bounded operator norm on $L^2(\RR^k)$ (we multiplied by a factor of $(2\pi h)^k$). As in \cite[Proposition 5.14]{de2014diffraction}, it suffices to show that the Hilbert--Schmidt norm of this operator is uniformly bounded. 
	
	Write $a= a_1 + a_2$, where $\left<\xi \right> \leq \left<\zeta'\right>$ on $\supp a_1$, and $\left<\zeta' \right> \leq 2\left<\xi\right>$ on $\supp a_2$. For $a_1$,
	\[
	 \left<\xi''\right>^{|\beta|}|D_{y''}^{\alpha} D_{\xi''}^{\beta}a_1(y'',\zeta',\xi',\xi'')| \leq C\left< \xi\right>^{l-n/2} \left< \zeta' \right>^{p+(n-k)/2}
	\]
	by the support assumption on $a_1$, and the proof proceeds just as in \cite[Proposition 5.14]{de2014diffraction}. For $a_2$ the proof is even simpler, since then
	\[	 \left<\xi'\right>^{k/2 + \delta} \left<\zeta' \right>^{k/2+\delta}\left<\xi''\right>^{|\beta|}|D_{y''}^{\alpha} D_{\xi''}^{\beta}a_1(y'',\zeta',\xi',\xi'')| \leq C\left<(\zeta',\xi)\right>^{p+l+k/2+2\delta}. 
	\]
	If $\delta>0$ is sufficiently small, then the right hand side is uniformly bounded, and this implies that the kernel \eqref{eq:conjugatedkernel} is uniformly square-integrable in $(\zeta',\xi')$.
	\end{proof}

We continue studying operators with kernels in $I^{p,l}_{\hmg}(\RR^{2n};\Lambda_0,\Lambda_1)$, but the results that follow are no longer coordinate invariant.

\begin{lem} \label{lem:kernelsupnorm}
Let $\Lambda_0,$ $\Lambda_1$ be as in \eqref{lagpair2}. Let $K_A \in I^{p,l}_{\hmg}(\RR^{2n};\Lambda_0,\Lambda_1)$ be of the form \eqref{eq:badparametrization}. If $l < -n/2$ and $p < -n/2 -k/2$, then $K_A$ is continuous, and
	\[
	|K_A(x',x'',y',y'')| \leq Ch^{-n-k}\left<x'/h \right>^{-N}\left<y'/h \right>^{-N}\left<(x''-y'')/h\right>^{-N}
	\]
	for each $N\geq 0$.
\end{lem}
\begin{proof}
Again, assume that $\Lambda_1 = \Lambda_R$.	As in Lemma \ref{lem:badL2bounded}, decompose $a= a_1 + a_2$, where $\left<\xi \right> \leq \left<\zeta'\right>$ on $\supp a_1$, and $\left<\zeta' \right> \leq 2\left<\xi\right>$ on $\supp a_2$. The hypotheses imply that $a \in L^1(\RR^{n+k})$, so $K_A$ is continuous and $|K_A(x',x',y',y'')| \leq Ch^{-n-k}$. Furthermore, integration by parts shows that
	\[
	|x'/h|^{N_1}|y'/h|^{N_2}|(x''-y'')/h|^{N_3}|K_A(x',x',y',y'')| \leq Ch^{-n-k}
	\]
	for every $N_1,N_2,N_3\geq 0$.
\end{proof}

We now proceed with some $L^\infty(\RR^k;L^2(\RR^{n-k}))$ bounds which
improve the loss in $h$ that occurs in Lemma
\ref{lem:badL2bounded}; these bounds will be essential to obtaining
optimal estimates for the size of the reflected wave in our
propagation argument later on.

We write $u(x')$ for the function $x'' \mapsto
u(x',x'')$ on $\RR^{n-k}$.

\begin{lem} \label{lem:improvedpairingbound}
	Let $K_A \in I^{p,l}_{\hmg}(\RR^{2n};\Lambda_0,\Lambda_1)$ be of the form \eqref{eq:badparametrization}. If $l < -n/2$ and $p < -n/2-k/2$, then 
	\[
	\lvert \left< Au, v\right>\rvert  \leq C\| u \|_{L^\infty(\RR^k;L^2(\RR^{n-k}))}\| v\|_{L^\infty(\RR^k;L^2(\RR^{n-k}))}.
	\]
\end{lem}
\begin{proof}
	Write the $L^2(\RR^n)$ pairing,
	\[
	 \left<Au,v \right> = \int \left( \int K_A(x',x'',y',y'') u(y',y'') v(x',x'') \, dx'' dy'' \right) dx' dy'.
	\]
By Lemma \ref{lem:kernelsupnorm} and Schur's lemma,
\begin{align*}
|\left<Au,v\right>| &\leq Ch^{-2k}\int \left<x'/h \right>^{-N}\left<y'/h \right>^{-N} \| u(y') \|_{L^2(\RR^{n-k})}  \| v(x') \|_{L^2(\RR^{n-k})} \, dx'dy' \\ &\leq 
 C\| u \|_{L^\infty(\RR^k;L^2(\RR^{n-k}))}\| v\|_{L^\infty(\RR^k;L^2(\RR^{n-k}))},
\end{align*}
which completes the proof.
\end{proof}

We also need an $L^\infty(\RR^k;L^2(\RR^{n-k})) \rightarrow L^1(\RR^k;L^2(\RR^{n-k}))$ boundedness result which similarly improves upon the loss in $h$ in Lemma \ref{lem:badL2bounded}.

\begin{lem} \label{lem:Linftyonefactor}
	Let $K_A \in I^{p,l}_{\hmg}(\RR^{2n};\Lambda_0,\Lambda_1)$ be of the form \eqref{eq:badparametrization}. If $l < -n/2$ and $p < -n/2-k/2$, then
	\[
	\| A \|_{ L^\infty(\RR^k;L^2(\RR^{n-k})) \rightarrow  L^1(\RR^k;L^2(\RR^{n-k}))} \leq C.
	\]
\end{lem}
\begin{proof}
By Cauchy--Schwarz,
\begin{multline*}
\int \| Au(x') \|_{L^2(\RR^{n-k})} \, dx' \\\leq  \int \left( \int |K_A(x,y)|| K_A(x,z)| |u(y',y'')|^2 dy'dy'' dz'dz'' dx''\right)^{1/2} dx',
\end{multline*}
and by Lemma \ref{lem:kernelsupnorm}, changing variables to replace
$x'',z''$ by $(x''-z'')/h$, $(x''-y'')/h,$ and beginning with the
$y''$ integral, we find that this is bounded by a constant times
$\| u \|_{L^\infty(\RR^k;L^2(\RR^{n-k}))}$
\end{proof}

Finally, we will need to consider composition of
$A \in I^{p,-\infty}_{\hmg}(\RR^{2n};\Lambda_0,\Lambda_1)$ with a
family of pseudodifferential operators on $\RR^{n-k}$ depending
parametrically on $\RR^k$ (cf.\ the discussion of ``tangential''
operators in Section \ref{subsect:semiclassicalpseudos}). Thus we
consider an operator $Q \in \CI(\RR^k; \Psi^0_h(\RR^{n-k}))$ with
Schwartz kernel
\begin{equation} \label{eq:tangentialkernel}
K_Q(x',x'',y',y'') = (2\pi h)^{-n+k} \delta(x'-y') \int e^{\sfrac i h \langle x''-y'', \eta'' \rangle}q(y',y'',\eta'') \,d\eta'',
\end{equation}
where $q \in S^0(\RR^n;\RR^{n-k})$. Supposing that $\Lambda_1 = \Lambda_R$, compose with $A$ given \eqref{eq:badparametrization}:
\begin{multline*}
K_{AQ} (x',x'',y',y'')\\ = (2 \pi h)^{-n-k} \int e^{\sfrac{i}{h}(\langle x',\xi'\rangle-\langle y',\zeta'\rangle + \langle x''-y'',\xi''\rangle}a(x'',\zeta',\xi',\xi'')q(y',y'',\xi'') \, d\xi d\zeta' dz'' d\eta'.
\end{multline*}
Clearly the resulting operator is in $I^{p,-\infty}_{\hmg}(\RR^{2n};\Lambda_0,\Lambda_1)$. The same argument works if $\Lambda_1 = \Lambda_L$.

\begin{lem} \label{lem:operatorvaluedcomposition}
	Let $Q,A$ be given by \eqref{eq:tangentialkernel} and \eqref{eq:badparametrization}, respectively. Suppose that 
	\[
	(x'',\zeta',\xi',\xi'') \in \esssupp(a) \Longrightarrow (x',x'',\xi'') \notin \WF(Q) \text{ for each } x' \in \RR^{k}.
		\]
	If $p < -n/2$, then
	\[
	\| AQ \|_{L^2(\RR^n)\rightarrow L^2(\RR^n)} = \mathcal{O}(h^\infty), \quad  \| QA \|_{L^2(\RR^n)\rightarrow L^2(\RR^n)} = \mathcal{O}(h^\infty).
	\]
\end{lem}
\begin{proof} 
	As in the proof of Lemma \ref{lem:badL2bounded}, we can view $A$ as being a $\mathcal{L}(L^2(\RR^k))$-valued operator, provided $p < -n/2$. Similarly, we can view $Q$ as an operator on $\RR^{n-k}$ with a $\mathcal{L}(L^2(\RR^k))$-valued symbol; in this case the symbol of $Q$ just acts on $L^2(\RR^k)$ as a multiplication operator. The assumed relation between $\esssupp(a)$ and $\WF(Q)$ guarantees that essential supports of their operator valued-symbols do not intersect, hence 
		\[
		\| AQ \|_{L^2(\RR^n)\rightarrow L^2(\RR^n)} = \mathcal{O}(h^\infty)
		\]
	by the calculus of operator-valued pseudodifferential operators. Either directly or by taking adjoints, $QA$ is similarly negligible.
	\end{proof}

\section{Diffractive improvements}
\label{sec:diffractive}

We now return to our operator $P = -h^2\Delta_g + V$ on $X$ and prove Theorems \ref{theo:propagation}, \ref{theo:improvedglancing}. Recall that we establish these theorems only when $\alpha > 1$.

\subsection{Decomposing the potential}
\label{subsect:potentialdecomp} We need to consider properties of the
potential appearing in our operator $P = -h^2\Delta_g + V$ more
carefully. All the material in this section applies to arbitrary
codimension. Thus we let $(X,g)$ be an $n$ dimensional Riemannian
manifold and $Y \subset X$ a codimension $k$ submanifold. We work in a
coordinate patch $\mathcal{U}$, identified with a subset of $\RR^n$,
with coordinates $(x',x'')$, where $Y \cap \mathcal{U} = \{x'=0\}$. \blue{We
will frequently take advantage of this coordinate decomposition to
write functions on $\mathcal{U}$ as functions in $x'$ with values in
some function space in $x''$, in order to obtain mixed-norm bounds.}
Assume that
\[
V \in I^{[\mu]}(\RR^n;N^*\{x'=0\})
\]
has compact support in $\mathcal{U}$. Thus we can  write
\[
V(x) = (2\pi  h)^{-k} \int e^{\sfrac i h \langle x',\eta' \rangle} v(x,\eta'/h) d\eta'
\]
for some $v(x,\eta') \in S^{\mu}(\RR^n;\RR^k)$ with compact support in the $x$ variables. As in the remark following Lemma \ref{lem:transversemult}, we decompose $V = V_0 + V_1$, where 
\begin{equation} \label{eq:V0potential}
V_0(x) = (2\pi h)^{-k} \int e^{\sfrac i h \langle x', \eta' \rangle}\chi(\eta'/\tau) v(x,\eta'/h) \, d\eta',
\end{equation}
and $V_1 = V - V_0$. Here $\chi \in \CcI(\RR^k;[0,1])$ is identically
one near $\eta'=0$, and $\tau>0$ is a parameter which will be
chosen small, so as to limit $\WF (V_0)$ to a neighborhood of the
zero-section in the conormal bundle to $\{ x'=0 \}$.

We remark for later use that provided $\mu<-k,$ we have a trivial $L^\infty$
estimate with decay in $h$,
\begin{equation}
  \label{LinfinityV1}
 \| V_1 \|_{L^\infty} = \mathcal{O}(h^{-k-\mu}).
\end{equation}
We also have a useful mixed-norm bound which will be used occasionally in place of Lemma \ref{lem:kernelsupnorm} to directly bound certain multiplication operators (the proof is completely analogous to that of Lemma \ref{lem:kernelsupnorm}):

\begin{lem}\label{lem:V1mixednorm}
If $\mu<-k$, then $\| V_1
\|_{L^1(\RR^k;L^\infty(\RR^{n-k}))} = 
        \mathcal{O}(h^{-\mu})$.
      \end{lem}
      \begin{proof}
        Recall that
        \[
V_1 (x) = (2\pi h)^{-k} \int e^{\sfrac i h \langle x', \eta' \rangle}(1-\chi(\eta'/\tau)) v(x,\eta'/h) \, d\eta'.
        \]
Owing to the support properties of $\chi$ we have the symbol estimate \[\lvert D_{\eta'}^\beta (1-\chi(\eta'/\tau))
       v(x'',\eta'/h) \rvert \leq C_\beta h^{-\mu} \langle \eta'
       \rangle^{\mu-|\beta|}\]
       for all multiindices $\beta$, where $C_\beta$ depends on $\tau$ as well. Repeated integration by parts shows that
       \[
       |V_1(x)| \leq  C_N h^{-\mu-k} \langle x'/h \rangle^{-N} 
       \]
which implies the desired estimate by integration and
change of variables.
      \end{proof}

  Fix $A\in \Psi^{\COMP}_h(\RR^n)$ with compact support in $\mathcal{U}$, which will later play the role of the commutant in a positive commutator argument. Write $\Lambda_0 = N^*(\{x'=0\} \cap \diag)$ and $\Lambda_1 =
N^*\diag$. According to the proof of Lemma \ref{lem:transversemult}
(see Remark~\ref{rem:oneterm}),
\[\label{kernelsofcompositions}
K_{V_0A}, \,K_{AV_0} \in I^{\mu+k/2,\COMP}_h(\RR^{2n}, \Lambda_0, \Lambda_1).
\]
The kernel of $A$ has wavefront set a compact subset of $(O \times O')
\cap N^*\diag$, where $O$ is open in $T^*X$ (with the usual notation
$O' = \{(x,-\xi): (x,\xi) \in O\}$.) As noted in Remark~\ref{rem:smallsupport}, by taking $\tau>0$ sufficiently small in
\eqref{eq:V0potential}, we can arrange that the kernels satisfy 
\begin{equation} \label{eq:commutatorWF}
\WF(K_{V_0 A}) \cup \WF(K_{A V_0}) \subset O \times O'.
\end{equation}
This is therefore true of the commutator $[A, V_0]$ as well. We also need to compute the principal symbol of $[A,V_0]$ along $N^*\diag$. A priori, 
\[
K_{[A,V_0]} \in I^{\mu+k/2,\COMP}_h(\RR^{2n}, \Lambda_0, \Lambda_1),
\]
but of course the principal symbol of $[A,V_0]$ along $N^*\diag$ vanishes, so in fact
\[
K_{[A,V_0]} \in hI^{\mu+k/2+1,\COMP}_h(\RR^{2n}, \Lambda_0, \Lambda_1).
\]
To compute the principal symbol of $[A,V_0]$, it is easiest to use the change of variables formulas from Section \ref{subsect:coordinateinvariance}. 
\begin{lem}	
	With $a = \sigma_h(A)$, the principal symbol of $(i/h)[A,V_0]$ along $\Lambda_1$ is $\hamvf_{a}V_0$ in $S^{\mu+k/2+1}_{\Lambda_1}/hS^{\mu+k/2+2}_{\Lambda_1}$
\end{lem}
\begin{proof}
	Set $b(y'',\eta') = e^{-ih \langle D_{y'}, D_{\eta'} \rangle }v(y,\eta'/h)\chi(\eta')|_{y'=0}$, so the kernel of $AV_0$ is
	\[
	K_{AV_0} (x,y)= (2\pi h)^{-n-k}\int e^{\sfrac i h (\langle y',\eta' \rangle + \langle x - y ,\xi \rangle) } a(x,\xi){b}(y'',\eta')\, d\eta' d\xi,
	\]
	where without loss we can assume that $a$ is the total left symbol of $A$. To put this in the framework of Section \ref{subsect:coordinateinvariance}, set $z=x-y$, so that in terms of coordinates $(y',z,x'')$,
	\[
	K_{AV_0} (x,y)= (2\pi h)^{-n-k}\int e^{\sfrac i h (\langle y',\eta' \rangle + \langle z ,\xi \rangle) } a(y'+z',x'',\xi){b}(x''-z'',\eta') \, d\eta' d\xi.
	\]
	It remains to express this in terms of coordinates
        $(x',z,x'')$, namely we pull back by the map $(x',z,x'')
        \mapsto (x'-z',z,x'')$. By \eqref{transformed}, the symbol of this pullback is
	\begin{multline*}
	e^{-ih\langle D_{z}, D_{\xi} \rangle} \big( e^{-\sfrac i h \langle z',\eta'\rangle}a(x',x'',\xi)b(x''-z'',\eta')\big) |_{z=0} \\ = a(x',x'',\xi)b(x'',\eta') + \langle \eta', \partial_{\xi'}a(x',x'',\xi'')\rangle b(x'',\eta') \\ - ih \langle \partial_{\xi''} a(x',x'',\xi''), \partial_{x''} b (x'',\eta')\rangle + hS^{\mu + k/2 + 2,\COMP}_{\Lambda_1}.
	\end{multline*}
	In the same $(x',z,x'')$ coordinates, the total symbol of $V_0A$ along $\Lambda_1$ is 
	\[
	a(x',x'',\xi) b(x'',\eta').
	\]
	Subtracting this second expression from the first, we obtain
        the desired result (after integration by parts in $\eta'$).
\end{proof}

\begin{rem} \label{rem:zeroorder}
 If $Q \in \Psi_h^0(\RR^n)$, then the kernel of $[Q,V_0]$ is not strictly part of the paired Lagrangian calculus developed in the previous sections; we will need to consider such an operator in Lemma \ref{lem:ellipticthresholdestimate} below. We therefore record two facts that remain true for $[Q,V_0]$.

First, let $q(x,\xi)$ be the total left symbol of $Q$. Arguing as in Section \ref{subsect:coordinateinvariance}, it follows that $K_{[Q,V_0]}$ can be written in the form \eqref{eq:singularsymbolL2bounded}, with amplitude
\[
e^{-ih \langle D_{z''}, D_{\xi''} \rangle } \left( q(x,\xi'+\eta',\xi'')b(x''-z'',\eta') - q(x,\xi',\xi'')b(x'',\eta') \right)|_{z''=0}.
\]
Taylor expanding $q(x,\xi'+\eta',\xi'')$ about $(x,\xi',\xi'')$ and integrating by parts in $\eta'$ shows that  the kernel of $h^{-1}[Q,W_1]$ can be written in the form \eqref{eq:singularsymbolL2bounded}, with an amplitude $a(x'',\eta',\xi',\xi'')$ that is compactly supported in $\eta'$ and satisfies
\[
|D_{x''}^\alpha D_{\xi''}^\beta a(x'',\eta',\xi',\xi'')| \leq C_{\alpha\beta}\left< \eta'/h \right>^{\mu + 1}\left<\xi''\right>^{-|\beta|}.
\]
According to Remark \ref{rem:SSweakerbounded}, if $\mu < -k$ then this implies that for some $\gamma \in (0,1/2]$,
\[
\|[Q,V_1]\|_{L^2 \rightarrow L^2} = \mathcal{O}(h^{2\gamma} ).
\]
Secondly, let $O$ be an open neighborhood of $\WF(Q)$ in $\OL{T^*}X$. Taking $\tau > 0$ sufficiently small in \eqref{eq:V0potential}, we can still arrange that 
\[
\WF(K_{[Q,V_0]}) \subset O \times O',
\]
as in \eqref{eq:commutatorWF}. The point here is that this is true even when $Q$ does not have compact microsupport. 
\end{rem}

As for the residual term $V_1$, we have
\begin{equation}\label{eq:AV1}
\begin{aligned}
K_{AV_1} &\in h^{-\mu}I^{\mu -n/2 + k/2,-\infty}_{\hmg,c}(\RR^{2n}, \Lambda_0,N^*\{y'=0\}), \\
K_{V_1A} &\in h^{-\mu}I^{\mu -n/2 + k/2,-\infty}_{\hmg,c}(\RR^{2n}, \Lambda_0,N^*\{x'=0\}).
\end{aligned}
\end{equation}
Observe that there is no gain in the commutator $[A,V_1]$ in terms of
powers of $h$ \blue{(or order of singularity) over $AV_1$ or $V_1A$: we will
simply estimate the summands in the commutator separately.}

\begin{lem} \label{lem:AVbound}
	Let $A \in \Psi^\COMP_h(\RR^n)$. If $\mu < -k/2$,  and $T \in \CcI(\RR^k; \Psi^\COMP_h(\RR^{n-k}))$ satisfies
	\[
	(x,\xi) \in \WF(A) \Longrightarrow (x,\xi'') \in \ELL(T),
	\]
	then
	\[
	\| AV_1 u \|_{L^2} + \| V_1 A u\|_{L^2} \leq C\|Tu\|_{L^2} + \mathcal{O}(h^\infty)\| u \|_{L^2}.
	\]
\end{lem}
\begin{proof}
	Choose $\psi \in \CcI(\RR^n)$ such that $A\psi = \psi A = A$, and let $t \in \CcI(\RR^{n-k})$ be such that $t(\xi'') = 1$ on $\{\xi'': (x,\xi) \in \WF(A)\}$. It suffices to prove the lemma with
	\[
	T = \psi \Op_h(t) \psi.
	\] 
	To do this, we apply Lemma \ref{lem:operatorvaluedcomposition}
        with $Q = 1-T_0,$ with $T_0$ satisfying the same properties as
        $T$ but microsupported in the elliptic set of $T$. This allows
        us to replace $u$ with $T_0 u$ modulo $\mathcal{O}(h^\infty)\|u\|_{L^2}$ errors. We apply
        Lemma~\ref{lem:badL2bounded} to bound $V_1 A T_0 u,$ while the
        $A V_1 T_0 u$ term is bounded similarly, following commutation
        of $V_1$ with $T_0;$ by tangential smoothness of $V,$ this
        yields an error term in the calculus
        $\CcI(\RR^k; \Psi^\COMP_h(\RR^{n-k}))$ which can be estimated
        by $\| T u\|_{L^2}$ where $T$ is elliptic on $\WF(T_0)$.
\end{proof}

We will also need a slightly more refined decomposition of $V_0$ itself. With $\chi$ as in \eqref{eq:V0potential}, write $V_0 = W_0 + W_1$, where
\begin{align} \label{eq:W0potential}
W_0 &= (2\pi h)^{-k} \int e^{\sfrac i h \langle x', \eta' \rangle}\chi(\tilde\tau \eta/h) v(x,\eta'/h) \, d\eta' \notag \\ &=  (2\pi)^{-k}\int e^{i \langle x', \eta' \rangle}\chi(\tilde\tau\eta) v(x,\eta') \, d\eta',
\end{align}
and $\tilde\tau > 0$ is a parameter. The point of this decomposition is that for $\mu + |\alpha| < -k$,
\begin{equation} \label{eq:W1decays}
D^\alpha_{x',x''} W_1 \rightarrow 0 \text{ uniformly as } \tilde\tau \rightarrow 0,
\end{equation}
whereas $W_0$ is smooth and independent of $h$. Also observe that the paired Lagrangian properties of $AV_0$ and $V_0A$ described above also apply to $AW_1$ and $W_1A$.

\subsection{Elliptic estimates} \label{subsect:ellipticthreshold}
We prove an elliptic estimate for $P = -h^2\Delta_g + V$ involving ordinary semiclassical wavefront set. Although everything in this section applies to arbitrary codimension, for simplicity we restrict to codimension one; thus we assume that $V \in I^{[-1-\alpha]}(Y)$, where $\alpha > 0$.

Since we are ultimately interested in $L^2$ based wavefront set, the estimates we give are quite crude in terms of Sobolev regularity.

\begin{prop} \label{prop:ellipticthreshold}
	Let $\alpha > 0$ and $s \leq \alpha +r$, where $s,r \in \RR \cup\{+\infty\}$. Suppose that $u$ is $h$-tempered in $H^1_h(X)$. If $\WF^r(u) = \emptyset$, then 
	\[
	\WF^{1,s}(u) \subset \Sigma \cup \WF^{-1,s}(Pu).
	\]	
\end{prop}

Recall that the notation $\WF^{k,s} (u)$ for ordinary semiclassical wavefront set relative to $H^k_h(X)$ was introduced in Definition~\ref{def:sobolevWF}.

Proposition \ref{prop:ellipticthreshold} follows from the quantitative estimate in Lemma \ref{lem:ellipticthresholdestimate} below; since stronger results are true away from $\OL{T^*_Y} X$, for the proof we assume that all operators have compact support in a coordinate patch $\mathcal{U}$ about $Y$.

\blue{We now obtain a semiclassical elliptic estimate. In contrast to Proposition~\ref{prop:elliptic}, this
  estimate concerns ordinary, rather
  than b-pseudodifferential operators.  Since the operators in
  question 
  do not respect the interface $Y,$ the resulting estimate has an
  $\alpha$-dependent loss on the right side.}
\begin{lem} \label{lem:ellipticthresholdestimate}
	If $A,G \in \Psi^0_h$ satisfy $\WF(A) \subset \ELL(G) \cap \ELL(P)$, then
	\[
	\| Au \|_{H^1_h} \leq C \|GPu\|_{H^{-1}_h} + Ch^\alpha \| u \|_{L^2} +  \mathcal{O}(h^\infty)\|u \|_{H^1_h}
	\]  
	for each $u \in H^1_h(X)$.
\end{lem}

\begin{proof} 
	The proof makes use of the decomposition $V= W_0 + W_1 + V_1$
        described in Section \ref{subsect:potentialdecomp}.  Let
        $P_{W_0}=-h^2\Delta_g+W_0$ and let $p_{W_0}$ denote its
        principal symbol. Note that
	\[
	\left<\zeta\right>^{-2} p \neq 0  \text{ near } \WF(A),
	\] 
	where we have written $\zeta = (\xi,\eta)$. If $\tilde\tau_0>0$ is sufficiently small  (where $\tilde\tau$ is the parameter appearing in \eqref{eq:W0potential}), then there is $c_0 > 0$ such that
	\begin{equation} \label{eq:ellipticboundedbelow}
	\left<\zeta\right>^{-2} |p_{W_0}| > c_0 \text{ near } \WF(A) \text{ for all } \tilde \tau \in (0,\tilde \tau_0).
	\end{equation}
	Let $Z \in \Psi^{-2}_h$ be everywhere elliptic with principal symbol $\left<\zeta\right>^{-2}$, and then set
	\begin{equation} \label{eq:ellipticsymbol}
	q = \left< \zeta \right>^2 \frac{\sigma_h(A)}{\sigma_h(P_{W_0})} \in S^0(T^*X).
	\end{equation}
	If $Q \in \Psi^0$ has principal symbol $q$, then \eqref{eq:ellipticboundedbelow} and \eqref{eq:ellipticsymbol} show that we can take $\WF(Q) \subset \WF(A)$, and that 
	\[
	\| Qu\|_{L^2} \leq C_0 \| Au \|_{L^2} + Ch\| Gu\|_{L^2} + \mathcal{O}(h^\infty) \| u \|_{L^2}
	\]
	where $C_0>0$ is uniform in $\tilde \tau \in (0,\tilde \tau_0)$. Furthermore, we can write
	\[
	A = ZQP_{W_0} + hF, \quad F \in \Psi^{-1}_h,
	\]
	where we may assume that $\WF(F) \subset \ELL(G)$. Now estimate
	\begin{align*}
	\|Au\|_{H^1_h} &\leq \| ZQPW_0 u\|_{H^1_h} + Ch\| F u \|_{H^1_h} \\ 
	&\leq C\| Q(P-W_1-V_1)u \|_{H^{-1}_h} + Ch\|  G u  \|_{L^2} + \mathcal{O}(h^\infty) \| u \|_{L^2}.  
	\end{align*}
Given $\varepsilon>0$, choose $\tilde \tau$ sufficiently small so that $\| W_1\|_{L^\infty} \leq \varepsilon$. This yields
\begin{align*}
\| QW_1 u \|_{L^2} &\leq \| W_1Q u\|_{L^2} + \| [Q,W_1] u \|_{L^2}  \\ &\leq \varepsilon \| Qu \|_{L^2} +  \| [Q,W_1]u \|_{L^2} \\
& \leq C_0\varepsilon \| Au \|_{L^2} + \| [Q,W_1]u \|_{L^2}  + Ch \| G u \|_{L^2} + \mathcal{O}(h^\infty) \| u \|_{L^2}
\end{align*}
We need to bound the $L^2$ norm of $[Q,W_1]u$. If $Q \in
        \Psi^\COMP_h$, then by Remark~\ref{rem:oneterm},
	\[
	K_{[Q,W_1]} \in h I^{-\alpha + 1/2,\COMP}_h(X, N^*((X\times Y) \cap \diag),N^*\diag)
	\]
	and  Lemma \ref{lem:singularsymbolL2bounded} would apply.
	However, since we are merely assuming that $Q \in \Psi^0_h$, the kernel of $[Q,W_1]$ is not strictly part of the paired Lagrangian calculus developed here. 
	
 On the other hand, the proof of Lemma \ref{lem:singularsymbolL2bounded} still applies in this setting, as explained in Remark \ref{rem:zeroorder}. In the notation of latter remark, let $O$ be an open neighborhood of $\WF(G)$ in $\OL{T^*}X$ such that $O \subset \WF(G)$. If $\tau > 0$ is sufficiently small so that $\WF(K_{[Q,W_1]}) \subset O \times O'$, then we can bound
	\[
	\| [Q,W_1]u\|_{L^2} \leq Ch^{2\gamma} \| G u\|_{L^2} + \mathcal{O}(h^\infty) \| u \|_{L^2}.
	\]
	for some $\gamma \in (0,1/2]$. 
	
	In order to bound the final term $\|QV_1 u\|_{L^2}$ simply use the estimate $\| V_1
        \|_{L^\infty} = \mathcal{O}(h^\alpha)$ by \eqref{LinfinityV1};
        hence
        \[
        \|QV_1 u\|_{L^2} \leq Ch^{\alpha}\| u \|_{L^2}.
        \]
	By taking $\varepsilon$ sufficiently small,
	\[
	\| Au \|_{H^1_h} \leq \| QPu \|_{H^{-1}_h} + Ch^\alpha\| u \|_{L^2} + h^{\min(1/2,\gamma)}\|Gu\|_{L^2} + \mathcal{O}(h^\infty)\|u \|_{L^2}.
	\]
	Finally, recall that $G$ is elliptic on $\WF(Q)$, hence $QPu$ can be replaced with $GPu$. The proof is then finished by an iterative argument, increasing the semiclassical regularity by $\min(1/2,\gamma)$ at each step.
\end{proof}

\begin{rem} \label{rem:microlocalizeellipticerror}
  The remainder term $h^\alpha \| u \|_{L^2}$ in Lemma
  \ref{lem:ellipticthresholdestimate} is not microlocalized. On the
  other hand, suppose that $A \in \Psi^\COMP_h$ in Lemma
  \ref{lem:ellipticthresholdestimate}. If
  $T \in \CcI(\RR;\Psi^\COMP_h(\RR^{n-1}))$ satisfies
\[
(x,y,\xi,\eta) \in \WF(A) \Longrightarrow (x,y,\eta) \in \ELL(T),
\]
then we can replace this term by $h^\alpha \|Tu \|_{L^2}$. This
follows since by Lemma~\ref{lem:operatorvaluedcomposition}
  (cf.\ Lemma \ref{lem:AVbound}) we can replace $QV_1 u$ in the proof with $QV_1 Tu$ modulo a $\mathcal{O}(h^\infty)_{L^2\rightarrow L^2}$ remainder.
\end{rem}

\subsection{Ordinary and b-wavefront sets}
We now present two results relating ordinary and b-wavefront sets. The
first allows us to replace microlocalization  by b-pseudodifferential
operators at $T^*Y \subset \bT^*X$ with tangential operators.

\begin{lem} \label{lem:btangentialwavefront}
	Let $\alpha >0$ and $s \in \RR \cup \{+\infty\}$. Suppose that $u$ is $h$-tempered in $H^1_h(\RR^{n})$. If $\WFb^{-1,s}(Pu) = \emptyset$ and $q_0 = (0,y_0,0,\eta_0) \notin \WFb^{1,s}(u)$, then there exists $T \in \CcI(\RR; \Psi^\COMP_h(\RR^{n-1}))$ with  $(0,y_0,\eta_0) \in \ELL(T)$
such that
\[
\| Tu \|_{H^1_h} \leq Ch^s.
\]
\end{lem}
\begin{proof}
Let $T \in \CcI((-\delta,\delta); \Psi^\COMP_h(\RR^{n-1}))$ satisfy  $(0,y_0,\eta_0) \in \ELLb(T)$ and 
\[
\WF(T) \subset \{|x| + |y-y_0| + |\eta-\eta_0| < \delta\}
\]
Define $f(x,y,\sigma,\eta) \in \symbb^0(\bT^*\RR^n)$ by
	\[
	f(x,y,\sigma,\eta) = \chi(\sigma^2/((C_0\delta)^2 \left<\eta\right>^2)),
	\]
	where $\chi = \chi(s) \in \CcI(\RR)$ is one for $|s| \leq 1$ and vanishes when $|s| \geq 2$. The parameter $C_0>0$ will be chosen later. Let 
	\[
	F = \Opbh(f) \in \Psibh^0,
	\]
	be properly supported. Since $|\sigma| \leq C\left<\eta \right>$ on $\supp f$, it follows that $TF \in \Psibh^0$, where
	\[
	q_0 \in \ELLb(TF), \quad \WF(TF) \subset \{ |x| + |y-y_0| + |\sigma| + |\eta-\eta_0| < C_1\delta\}
	\]
	for some $C_1>0$. Taking $\delta>0$ sufficiently small implies that $\| TFu\|_{H^1_h} \leq Ch^s$.

	On the other hand, we can write $T = T\varphi$, where $\varphi \in \CcI(\RR^n)$ has $\supp \varphi \subset \{|x| < \delta \}$. Therefore,
	\[
	\| T(1-F)u \|_{H^1_h} = \| T\varphi(1-F)u\|_{H^1_h} \leq C\|\varphi(1-F)u\|_{H^1_h},
	\]
	as $T$ is uniformly bounded on $H^1_h(\RR^n)$.
	Now observe that $\varphi(1-F) \in \Psibh^0$ has compact support in $\{|x| < \delta\}$, and 
	\[
	|\sigma| \geq C_0 \delta \text{ on } \WFb(\varphi(1-F)).
	\] 
	By taking $C_0>0$ sufficiently large, there exists $\delta_0>0$ such that Lemma \ref{lem:sigmanotzero} applies to $\varphi(1-F)$ for $\delta \in (0,\delta_0)$. In particular,
	\[
	\| T(1-F)u \|_{H^1_h} \leq  C\| Pu \|_{H^{-1}_h} + \mathcal{O}(h^\infty)\| u \|_{H^1_h},
	\]
	which completes the proof.
\end{proof} 
Note that the proof (or alternatively the Closed Graph Theorem) in fact yields the quantitative statement
\begin{equation}
  \label{quantitativetangential}
  \|T u\|_{H^1_h} \leq C h^s \|G_\B u \|_{H^1_h} + C\| Pu \|_{H^{-1}_h} + \mathcal{O}(h^\infty)\| u \|_{H^1_h},
\end{equation}
where $G_\B \in \Psibh^0$ is elliptic near $q_0.$

Next, we show by a similar argument that at glancing points (or rather
their preimages in $T^*_YX$), microlocalization by ordinary
semiclassical pseudodifferential operators can be replaced with
microlocalization by tangential operators.

\begin{lem} \label{lem:tangentialwavefront}

	Let $\alpha >0$ and $s \in \RR \cup \{+\infty\}$. Suppose that $u$ is $h$-tempered in $H^1_h(\RR^n)$. Let 
	\[
	\varpi_0 = (0,y_0,0,\eta_0) \in \pi^{-1}(\gl) \cap T^*_Y X
	\]
	If $\WF^{-1,s}(Pu) =\emptyset$ and $\varpi_0 \notin \WF^s(u)$, then there exists $T \in \CI(\RR_x; \Psi^\COMP_h(\RR^{n-1}_y))$ with $(0,y_0,\eta_0) \in \ELL(T)$ such that
	such that
	\[
	\| Tu \|_{H^1_h} \leq Ch^s.
	\]
\end{lem}
\begin{proof}

	The proof is similar to that of Lemma \ref{lem:btangentialwavefront}. Let $T \in \CcI((-\delta,\delta); \Psi^\COMP_h(\RR^{n-1}))$ with total left symbol $t = t(x,y,\eta)$ satisfy $(0,y_0,\eta_0) \in \ELLb(T)$ and 
\[
\WF(T) \subset \{|x| + |y-y_0| + |\eta-\eta_0| < \delta\}.
\]
Let $\varphi \in \CcI(\RR^n)$ be such that $\supp \varphi
\subset\{|x| + |y-y_0| < \delta\}$ and $T = \varphi T$.
Because $\varpi_0$ is a glancing point, we know that $\tilde{p}(x_0,y_0,\eta_0) = 0$. Thus for $\delta>0$ sufficiently small,
\[
|\tilde p| \leq C' \delta^{\theta} \text{ on } \supp t.
\]
where $\theta \in (0,1]$ is a H\"older exponent for $V$ (recall that $\alpha > 0$). Define $f(x,y,\xi,\eta) \in S^0(T^*\RR^n)$ by
\[
f(x,y,\xi,\eta) = \chi(\xi^2/(C_0\delta^{\theta/2} \left<\eta\right>)^2),
\]
where $\chi = \chi(s) \in \CcI(\RR)$ is one for $|s| \leq 1$ and vanishes when $|s| \geq 2$. Let 
\[
F = \Op_h(f) \in \Psi^0_h
\]
As in Lemma \ref{lem:btangentialwavefront}, since $|\xi| \leq C\left<\eta \right>$ on $\supp f$, it follows that $FT \in \Psi^\COMP_h$, such that
\[
\varpi_0 \in \ELL(FT), \quad \WF(FT) \subset \{ |x| + |y-y_0| + |\eta-\eta_0| < C_1\delta, \, |\xi| < C_1\delta^{\theta/2}\}.
\]
Taking $\delta>0$ sufficiently small implies that $\| FTu\|_{H^1_h} \leq Ch^s$. 

Next, choose $T' \in \CcI((-\delta,\delta); \Psi^\COMP_h(\RR^{n-1}))$ with same properties as $T$, replacing $\delta$ with $(1 + \varepsilon)\delta$ for $\varepsilon>0$ arbitrarily small. We may choose $T'$ so that
\[
T = T'T + \mathcal{O}(h^\infty)_{H^1_h \rightarrow H^1_h}.
\]
Let $t'$ be a total symbol for $T'$. Decompose the function $1-f = f_1 + f_2$, where 
\[
C_0 \delta^{\theta/2} \leq |\xi| \leq 2C_1\langle \eta \rangle \text{ on }\supp f_1, \quad  |\xi| \geq C_1 \langle \eta \rangle \text{ on } \supp f_2.
\] 
Writing, $F_i = \Op(F_i)$, we have that $F_1 T' \in \Psi^\COMP_h$ with principal symbol $ f_1 t'$. Now $|\xi| > C_0\delta^{\theta/2}$ on $\supp (f_1)$, so if $C_0 > 0$ is sufficiently large, then 
\[
p(x,y,\xi,\eta) = \xi^2 + k^{ij}(x,y)\eta_i \eta_j + V(x,y) > c_0 \delta^\theta
\]
on $\supp (f_1 t')$, where $c_0 > 0$ does not depend on $\delta$. Thus $\WF(F_1 T') \subset \ELL(P)$, so applying Lemma \ref{lem:ellipticthresholdestimate} to the function $Tu$,
\begin{align*}
\| F_1 Tu \|_{H^1_h} &\leq \| (F_1 T')T u \| _{H^1_h} + \mathcal{O}(h^\infty)\| u \|_{H^1_h}
\\&\leq C \|PTu\|_{H^{-1}_h} + Ch^{\alpha}\| Tu \|_{L^2} + \mathcal{O}(h^\infty)\|u \|_{H^1_h}.
\end{align*}
On the other hand, for the term $F_2$, if we take $C_1>0$ sufficiently large, then $p \geq c\langle \zeta \rangle^2$ on $\supp(\varphi f_2)$. Thus $\WF(F_2 \varphi) \subset \ELL(P)$, so again
\[
\| F_2 Tu \|_{H^1_h} = \| (F_2 \varphi) Tu \|_{H^1_h} \leq C \|PTu\|_{H^{-1}_h} + Ch^{\alpha}\| Tu \|_{L^2} + \mathcal{O}(h^\infty)\|u \|_{H^1_h}
\]
In order to handle either of the terms involving $F_1$ or $F_2$, it therefore suffices to bound $\|PTu\|_{H^{-1}_h}$. This is done by writing $PTu = TPu + [P,T]u$ and bounding $\| TPu\|_{H^{-1}_h} \leq C\| Pu\|_{H^{-1}_h}$. As for the commutator,
\[
[P,T] = h (hD_x) T_1+ hT_0 + [V,T],
\]
where $T_i \in \CI(\RR; \Psi^\COMP_h(\RR^{n-1}))$. Here we can view $[V,T] \in \C^0(\RR; h\Psi^\COMP_h(\RR^{n-1}))$. Since $T'$ is elliptic on $\WF(T)$,
\[
\| [P,T]u\|_{H^{-1}_h} \leq Ch \| T' u \|_{L^2} + \mathcal{O}(h^\infty)\|u \|_{L^2}.
\]
Altogether, we have
\[
\| Tu \|_{H^1_h} \leq C \| Pu \|_{H_h^{-1}} + C\| A u \|_{H^1_h} + Ch^{\min(1,\alpha)}\| T'u \|_{H^1_h} +  \mathcal{O}(h^\infty)\| u \|_{H^1_h},
\]
Since the wavefront set of $T'$ is larger than that of $T$ by an arbitrarily small amount the proof is finished by induction, improving the semiclassical regularity by $h^{\min(1,\alpha)}$ at each step.
\end{proof}

Lemmas \ref{lem:btangentialwavefront} and \ref{lem:tangentialwavefront} can be combined using the following observation: if $A_\B \in \Psibh^\COMP$ and $T \in \CI(\RR; \Psi^\COMP_h(\RR^{n-1}))$ are such that $(x,y,\eta) \in \WF(T)$ implies $(x,y,\sigma,\eta) \notin \WFb(A_\B)$ for any $\sigma \in \RR$, then
\begin{equation} \label{eq:bintermsoftangential}
\| A_\B u\|_{L^2} \leq \| Tu \|_{L^2} + \mathcal{O}(h^\infty)\|u\|_{L^2}.
\end{equation}
The proof is similar to that of Lemma \ref{lem:operatorvaluedcomposition}.

\begin{lem} \label{lem:WFbimpliesWF}
	Let $\alpha > 0$ and $r\in \RR \cup \{+\infty\}$. Let $\varpi_0 \in \pi^{-1}(\gl)$ and $q_0 = \pi(\varpi_0)$. Suppose that $u$ is $h$-tempered in $H^1_h(X)$ and $Pu \in L^2(X)$. If 
	\[
	\WF^r(Pu) = \emptyset, \quad \varpi_0 \notin \WF^{r}(u),
	\]
	then $q_0 \notin \WFb^{1,r}(u)$.
\end{lem}
\begin{proof}
	This follows by combining Lemma \ref{lem:tangentialwavefront}
        with \eqref{eq:bintermsoftangential} and Lemma~\ref{lem:L2H1}.
\end{proof}

\subsection{Improvement at hyperbolic points} \label{subsect:truepositivecommutator}

We are now ready to prove Theorem \ref{theo:propagation}. Fix $\varpi_0 \in \pi^{-1}(\hyp)$, and write
\[
\varpi_0 = (0,y_0, \xi_0,\eta_0)
\]
with respect to a fixed normal coordinate patch $\mathcal{U}$, where $\xi_0 > 0$ for concreteness. 

\begin{prop} \label{prop:truepropagation}
Let $\alpha > 1$ and $s \leq r+ \alpha$, where $s,r \in \RR \cup \{+\infty\}$. Suppose that $u$ is $h$-tempered in $H^1_h(X)$ with $Pu \in L^2(X)$, such that
	\[
	\pi(\varpi_0) \notin \WFb^{1,r}(u), \quad \WF^{s+1}(Pu) = \emptyset.
	\]
	If there is a neighborhood $U \subset T^*X$ of $\varpi_0$ such that $U \cap \WF^s(u) \cap \{x <0\} = \emptyset$, then $\varpi_0 \notin \WF^s(u)$.
\end{prop}

As usual, Proposition \ref{prop:truepropagation} follows from a quantitative estimate via a positive commutator argument.

\begin{prop} \label{prop:truehyperbolic} If $G \in \Psi_h^\COMP$ is elliptic at $\varpi_0$ and $Q_\B \in \Psibh^\COMP$ is elliptic at $\pi(\varpi_0)$, then there exist $Q, Q_1\in \Psi_h^\COMP$, where
	\begin{gather*}
	\WF(Q) \subset \ELL(G) \text{ and } \varpi_0 \in \ELL(Q),\\
	\WF(Q_1) \subset \ELL(G) \cap \{x < 0\},
	\end{gather*}
	such that
	\begin{equation} \label{eq:truehyperbolicestimate}
	\| Qu \|_{L^2} \leq  C h^{-1}\| Pu\|_{L^2} + C\| Q_1u \|_{L^2}
        + Ch^\alpha \|Q_\B u \|_{H^1_h}+\mathcal{O}(h^\infty) \|u\|_{L^2},
	\end{equation}
	for each $u\in H^1_h(X)$ with $Pu \in L^2(X)$.
	\end{prop}

Note that $G$ can be used to control the sizes of $\WF(Q)$ and $\WF(Q_1)$, but the term involving $Pu$ is not microlocalized. The term involving $Q_\B u$ is microlocalized, but only in the sense of b-wavefront set; by Theorem \ref{theo:GBBpropagation}, it can be controlled by the singularities along backwards $\GBB$s from $\pi(\varpi_0)$. 

\begin{rem} \label{rem:smoothu}
	By a regularization argument it suffices to prove Proposition \ref{prop:truehyperbolic} (and also Proposition \ref{prop:glancingoutput} in the next section) for $u \in \CI(X)$. Indeed, given $u \in H^1_h(X)$ with $Pu \in L^2(X)$ we can choose $u_j \in \CI(X)$ such that 
	\[
	u_j \rightarrow u \text{ in } H^1_h(X), \quad h^2\Delta_gu_j \rightarrow h^2\Delta_g u \text{ in }L^2(X)
	\]
	(see \cite[Lemma E.47]{zworski:resonances} for instance). This of course implies $Pu_j \rightarrow Pu$ in $L^2(X)$ as well.\end{rem}

One key to the proof of Proposition \ref{prop:truehyperbolic} is the use of the microlocal energy estimates discussed in Section \ref{subsect:energy}. Suppose
that $u \in \CI(X)$ is supported in a normal coordinate patch near $Y
\subset X$. If $(x,y) \in \RR \times \RR^{n-1}$ are the corresponding
normal coordinates, we can apply Lemma \ref{lem:energyestimate} to the
operator $L=P$ with $x_1=x$ and $x' = y$. Indeed,
\[
P = (hD_x)^*(hD_x) - h^2\Delta_k + V, 
\]
and $V \in \C^1(X)$ since $\alpha > 1$. The hypotheses on $R  = -\tilde P = h^2\Delta_k - V$ are satisfied in a sufficiently small neighborhood $(-\varepsilon,\varepsilon)\times U$ of a point $\tilde{q}_0 \in \pi^{-1}(\hyp)$.

We use an approach quite close to that of Proposition \ref{prop:hyperbolic}. Define the functions
\[
\omega = |\xi-\xi_0|^2 + |y-y_0|^2 + |\eta - \eta_0|^2, \quad \phi = x+\frac{1}{\beta^2\delta}\omega.
\]
We use the same cutoffs $\chi_0,\chi_1$ as in Proposition \ref{prop:hyperbolic}. 
We also fix a cutoff $\psi \in \CI(T^*X;[0,1])$ such that $\psi =1 $ near $\{|x| \leq 2\delta, \, \omega^{1/2} \leq 2\beta\delta\}$ with support in $\{|x| <3\delta ,\, \omega^{1/2} < 3\beta\delta\}$. Now set
\begin{equation*} 
a=\chi_0(2-\phi/\delta) \chi_1(2+x/\delta).
\end{equation*}
The support properties of $a$ can be read off from the analogue of
Lemma \ref{lem:asupport}; in particular,
\[
\supp a \subset\{\lvert x \rvert \leq 2\delta,\,  \omega^{1/2} \leq 2
\beta \delta\},
\] 
hence $\psi=1$ on the support of $a.$ Recall that we are assuming $V \in I^{[-1-\alpha]}(Y)$ with $\alpha > 1$. We will use a decomposition $V= V_0 + V_1$ as in Section \ref{subsect:potentialdecomp} \emph {which may depend on $\delta$, but not $\beta$}. 

 To proceed with the positive commutator argument, write
\[\label{poscommV1}
-(2/h)\Im \left<APu - AV_1u ,Au\right>  = (i/h)\left<[A^*A,-h^2\Delta_g + V_0]u,u\right>.
\]
The right hand side is treated symbolically within the paired Lagrangian calculus. For convenience, set
\[\label{fdef}
f = \sigma_h(-h^2\Delta_g).
\]
Let $P_{V_0} = -h^2\Delta_g + V_0$ and $p_{V_0} = f + V_0$.  For
simplicity, write $z = (x,y)$ and $\zeta = (\xi,\eta)$. The point of
the next lemma is that it holds uniformly with respect to the
decomposition $V = V_0 + V_1$, i.e., with respect to the choice of
$\tau$ in \eqref{eq:V0potential}.

\begin{lem} \label{lem:epsilondeltachoice}
Let $f = \sigma_h(-h^2\Delta_g) \in \CI(T^*X)$. There exists
$\beta,\delta_0, c_0, \tau_0 > 0$ such that for each $\delta \in
(0,\delta_0)$ and $\tau \in (0,\tau_0)$ in \eqref{eq:V0potential},
\[
\hamvf_f \phi \geq 2c_0, \quad |\partial_{z_i}V_0 \cdot \partial_{\zeta_i} \phi| \leq \hamvf_f \phi/(4n) \text{ for } i =1,\ldots,n.
\]
on $\supp \psi$.
\end{lem}
\begin{proof} 
For any $g \in \CI(T^*X)$,
\[
|\hamvf_g \omega| \leq C_0\omega^{1/2}
\]
uniformly on any fixed neighborhood $U$ of $\varpi_0$. This is therefore true for the smooth part $f = \sigma_h(-h^2\Delta_g)$ of $p_W$. As for the potential, the crucial point here is that if $U$ is a fixed neighborhood, then for $i= 1,\ldots, n$,
\[
|\partial_{z_i}V_0 \cdot \partial_{\zeta_i} \omega| \leq C_1 \omega^{1/2}
\]
on $U$ for a constant $C_1 > 0$ that is independent of the choice of the parameter $\tau > 0$ in \eqref{eq:V0potential}; this is obvious from the oscillatory integral representation of $V_0$.

On the other hand,
\[
\hamvf_f x = 2\xi.
\]
If we fix a sufficiently small neighborhood  $U$ of $\varpi_0$, it follows that $\hamvf_f x \geq 3c_0$ on $U$ for some $c_0 > 0$. Fix  $\beta> 3(C_0 + 2nC_1)/c_0$, and suppose that $\delta_0 > 0$  is such that $\supp \psi \subset U$ for $\delta \in (0,\delta_0)$. Then,
\begin{equation} \label{eq:Hfphi}
\hamvf_{f}\phi \geq 3c_0 - C_0\beta^{-2}\delta^{-1} \omega^{1/2} \geq 3c_0 - 3C_0 \beta^{-1} \geq 2c_0
\end{equation}
on $\supp \psi$, and in addition
\begin{equation} \label{eq:HWphi}
|\partial_{z_i}V_0 \cdot  \partial_{\zeta_i} \phi| \leq 3C_1 \beta^{-1} \leq c_0/(2n) \leq \hamvf_{f}\phi/(4n) 
\end{equation}
on $\supp \psi$.
\end{proof}

We now examine properties of the commutator as a whole; note that
$\beta > 0$ and $\delta_0 > 0$ have been fixed by Lemma
\ref{lem:epsilondeltachoice}, and we are now taking
$\delta \in (0,\delta_0)$. First, consider the smooth part
$f = \sigma_h(-h^2\Delta_g)$ of $p_{V_0}$. Define
\[
b = (2\delta)^{-1/2}(\hamvf_f \phi)^{1/2}(\chi_0 \chi_0')^{1/2}\chi_1,
\]
which is well-defined and smooth in light of \eqref{eq:Hfphi}. We then compute
\begin{align*}
\hamvf_{f} (a^2) &= -2\delta^{-1} (\hamvf_{f}\phi)(\chi_0 \chi_0')\chi_1^2 + 2\delta^{-1}(\hamvf_{f} x)\chi_0^2(\chi_1 \chi_1') \\
& = -b^2 + e,
\end{align*}
noting that $\supp e \subset \{-2\delta \leq x \leq -\delta\} \cap \supp b$. Fix  compactly supported operators $B$ and $E$ in $\Psi^{\COMP}_h$ with principal symbols $b$ and $e$, respectively.

Next, fix compactly supported operators $R_1,\ldots,R_n \in
\Psi^{\COMP}_h(X)$ with principal symbols 
\[
r_i = (\hamvf_f\phi)^{-1}(\partial_{\zeta_i} \phi) \psi.
\]
In particular, $\psi \hamvf_{V_0}\phi = (\hamvf_f \phi)  \sum  \partial_{z_i}V_0 \cdot r_i$, and  $\sum |  \partial_{z_i}V_0 \cdot r_i| \leq 1/4$ by our choice of $\beta$. Moreover,
\[
\hamvf_{V_0}(a^2) = -2\delta^{-1}(\hamvf_{V_0} \phi)(\chi_0 \chi_0')\chi_1^2 = -b^2\left(\dfrac{\hamvf_{V_0} \phi}{\hamvf_f \phi}\right) = -b^2 \sum
 \partial_{z_i} V_0 \cdot r_i,
\]
since $\psi=1$ on $\supp b$.
Note that $\hamvf_{V_0}(a^2) \geq -(1/4)b^2$, but we do not use this directly within the symbol calculus.

Instead, for a given $\delta \in (0,\delta_0)$, fix an open set $O \subset \supp \psi$ containing $\WF(B)$. Since $\WF(A) \subset \WF(B)$ we can choose $V_0$ such that 
\[
\WF(K_{[A^*A,V_0]}) \subset O\times O'.
\]
By further shrinking $\tau$ in \eqref{eq:V0potential}, we can  arrange that the kernels of $B^*(\partial_{z_i} V_0) R_iB$ also have wavefront set contained in $O \times O'$, since the operators $B,B^*,R_i$ are independent of $V_0$. By Proposition \ref{prop:singularcomposition},
\begin{multline*}
(i/h)[P_{V_0}, A^*A] + B^*B + B^*\sum (\partial_{z_i} V_0)R_i B + E \\ \in I^{-\alpha+(1/2)+\varepsilon_0,\COMP}_h(X, N^*((X\times Y) \cap \diag), N^*\diag)
\end{multline*}
for any $\varepsilon_0 > 0$. If this operator is denoted by $F$, then by construction the principal symbol of $F$ along $N^*\diag$ vanishes, and hence
\[
F\in h I^{-\alpha+(3/2)+\varepsilon_0,\COMP}_h(X, N^*((X\times Y) \cap \diag), N^*\diag).
\]
The key here is that since all of the operators above have kernels with wavefront set in $O\times O'$, so does $F$. 

Now we consider the identity
\begin{equation} \label{eq:V_0commutator}
\left<(i/h)[P_{V_0}, A^*A]u,u\right> = \| Bu\|^2_{L^2} + \sum \left<(\partial_{z_i} V_0)R_iBu,Bu\right> + \left<Eu,u\right> + \left<Fu,u\right>.
\end{equation}
The second, third, and fourth terms on the right hand side of \eqref{eq:V_0commutator} are bounded in absolute value as follows. For the second term, we use the bound 
\[
\|R_i u \|_{L^2} \leq 2\sup |\sigma_h(R_i)|\|u\|_{L^2} + \mathcal{O}(h^\infty)\|u\|_{L^2},
\]
and the fact that $2\sum \sup |\partial_{z_i} V_0||r_i| \leq 1/2$ by construction. Therefore
\[
\sum |\left<(\partial_{z_i} V_0)R_iBu,Bu\right>| \leq (1/2)\|Bu\|^2_{L^2} + \mathcal{O}(h^\infty)\|u \|_{L^2}.
\]
To bound the third term, choose $Q_1 \in \Psi^\COMP_h$ as in the
statement of the proposition such that $\WF(E) \subset \ELL(Q_1)$ and estimate
\[
|\left<Eu,u\right>| \leq C\| Q_1 u\|^2_{L^2} + \mathcal{O}(h^\infty)\| u \|^2_{L^2}.
\]
For the fourth term, we apply Lemma \ref{lem:singularsymbolL2bounded}:  since $\alpha > 1$, fix $\gamma >0$ such that  
\[
-\alpha + 2\gamma + \varepsilon_0 < -1,
\] 
where recall $\varepsilon_0 > 0$ is arbitrarily small. Taking $s = 1-2\gamma$, 
\[
-\alpha + (3/2) + \varepsilon_0- s = -\alpha + (1/2) + \varepsilon_0 + 2\gamma < -1/2.
\]
Therefore, by Lemma \ref{lem:singularsymbolL2bounded},
\begin{equation} \label{eq:gammabound}
\| F \|_{L^2 \rightarrow L^2} \leq Ch^{2\gamma}.
\end{equation}
Let $G \in \Psi^\COMP_h(X)$ be elliptic on $\WF(B)$; since $O$ was an
arbitrary neighborhood of $\WF(B)$, we can assume that $O \subset
\ELL(G)$ as well. Thus we can bound 
\[
|\left< Fu,u\right>| \leq Ch^{2\gamma} \| Gu\|^2 + \mathcal{O}(h^\infty)\|u\|^2_{L^2}.
\]
Combining \eqref{eq:V_0commutator} and \eqref{poscommV1}, we obtain the useful bound
\begin{multline*}
\| Bu \|^2_{L^2}  \leq Ch^{-1} \|APu\|_{L^2} \| Au \|_{L^2} + Ch^{-1}|\left<AV_1u,Au\right>| \\ 
+ Ch^{2\gamma}\| Gu \|^2_{L^2} + \| Q_1u\|^2_{L^2} + \mathcal{O}(h^\infty)\|u\|^2_{L^2}.
\end{multline*}
Note that the various terms involving $\|Au\|_{L^2}$ can be bounded in
terms of $\|Bu\|_{L^2}$. This is done as at the end of Section
\ref{subsect:hyperbolicregion}, yielding 
\begin{equation} \label{eq:AintermsofB}
\|Au\|_{L^2} \leq C\|Bu\|_{L^2} + Ch\| Gu \|_{L^2} + \mathcal{O}(h^\infty)\|u\|_{L^2}.
\end{equation}
It remains to bound the term $h^{-1}|\left<AV_1u,Au\right>|$.
Using Lemma~\ref{lem:operatorvaluedcomposition}
  (cf.\ Lemma \ref{lem:AVbound}),
we can choose a tangential operator $ T \in \CcI(\RR; \Psi^\COMP_h(\RR^{n-1}))$ with
\[
\WF(T) \subset \{|x| < 3\delta, \,|y-y_0|^2 + |\eta-\eta_0|^2 \leq 9\beta^2 \delta^2 \},
\] 
such that
\begin{equation} \label{eq:insertT}
\|AV_1u\|_{L^2} = \|AV_1Tu\|_{L^2} + \mathcal{O}(h^\infty)\|u\|_{L^2}.
\end{equation}
The same lemma shows that
\begin{equation} \label{eq:insertTcommutator}
\|[V_1,A^*A] u\|_{L^2} = \|[V_1,A^*A] Tu\|_{L^2} + \mathcal{O}(h^\infty)\|u\|_{L^2}.
\end{equation}
The next step is to apply Lemmas \ref{lem:improvedpairingbound}, \ref{lem:Linftyonefactor}, \ref{lem:energyestimate}.

\begin{lem}
For each $\varepsilon > 0$ there exists $C_\varepsilon > 0$ such that	
\[
|\left< AV_1u,Au\right>| \leq\varepsilon \| Bu \|^2_{L^2} + C_\varepsilon \big( h^{-1} \| Pu \|^2_{L^2}  + \|Q_1 u \|^2_{L^2} + h^{2\gamma} \| Gu\|^2_{L^2} + h^{2\alpha}\| Tu \|^2_{H^1_h} \big),  
\]
where $T\in \CcI(\RR; \Psi^\COMP_h(\RR^{n-1}))$ is as above.
\end{lem}
\begin{proof}
	Recall from \eqref{eq:AV1} that
	\[
	AV_1 \in h^{-1-\alpha}I^{-\alpha -n/2 - 1/2,-\infty}_\hmg(X\times X; N^*((X\times Y) \cap \diag), N^*(X\times Y)).
	\]
Arguing as in the preceding paragraph, 
	\[
	|\left< AV_1u,Au\right>| = |\left<V_1Tu, TA^*Au\right>| + \mathcal{O}(h^\infty) \| u\|_{L^2}^2.
	\]
	Instead of using Lemma
        \ref{lem:improvedpairingbound}, we may easily bound a pairing of the form
        $\left<V_1 w,v\right>$ by
        Lemma~\ref{lem:V1mixednorm}.
        This yields
	\[
|\left< V_1 Tu,TA^*Au\right>| \leq  C h^{\alpha+1 }\| T u\|_{L^\infty((-3\delta,3\delta);L^2(\RR^{n-1}))}\| TA^*A u\|_{L^\infty((-3\delta,3\delta);L^2(\RR^{n-1}))}.
	\]
	Here we used that $A$ has compact support in $\{ |x| < 3\delta\}$. If $\delta$ is sufficiently small, then Lemma \ref{lem:energyestimate} is applicable. By Cauchy--Schwarz,
	\begin{align*}
	h^{-1} |\left< AV_1u,Au\right>| &\leq  C_\varepsilon h^{2\alpha }\| T u\|^2_{L^\infty((-3\delta,3\delta);L^2(\RR^{n-1}))} + \varepsilon\| TA^*A u\|^2_{L^\infty((-3\delta,3\delta);L^2(\RR^{n-1}))} \\ &+ \mathcal{O}(h^\infty)\| u \|_{L^2}^2
	\end{align*}
	for each $\varepsilon > 0$. Let $T_1 \in \CI(\RR_{x}; \Psi^\COMP_h(\RR^{n-1}_y))$ be elliptic on $\WF(T)$. Applying Lemma \ref{lem:energyestimate}, we deduce that
	\[
	\| T u\|_{L^\infty((-\varepsilon,\varepsilon);L^2(\RR^{n-1}))} \leq Ch^{-1}\|Pu\|_{L^2} + C\| T_1u \|_{H^1_h} + \mathcal{O}(h^\infty)\| u \|_{H^1_h}.
	\]
	As for the next term, we again apply Lemma \ref{lem:energyestimate}, but this time writing
	\begin{multline} \label{eq:TA*A}
	\| T A^*A u\|_{L^\infty((-3\delta,3\delta);L^2(\RR^{n-1}))} \leq Ch^{-1}\int_{-3\delta}^{3\delta} \|PA^*A u(s)\|_{L^2(\RR^{n-1})}\, ds \\ + C\| A^*A u \|_{H^1_h} + \mathcal{O}(h^\infty)\| u \|_{H^1_h}.
	\end{multline}
	Since $A^* \in \Psi_h^\COMP(X)$, the second term on the right hand side of \eqref{eq:TA*A} is estimated by $\| Au \|_{L^2} + \mathcal{O}(h^\infty)\|u\|_{L^2}$. The first term on the right hand side of \eqref{eq:TA*A} is bounded by a constant times
	\begin{multline*}
	h^{-1} \int_{-3\delta}^{3\delta} \|A^*A Pu(s)\|_{L^2(\RR^{n-1})} +  \| [P,A^*A] u(s)\|_{L^2(\RR^{n-1})} \, ds \\ \leq Ch^{-1}(\|APu\|_{L^2} +   \| [P_{V_0},A^*A] u\|_{L^2}) +h^{-1} \int_{-3\delta}^{3\delta} \| [V_1,A^*A] u(s)\|_{L^2(\RR^{n-1})} \, ds.
	\end{multline*}
	Now recall that $(i/h)[P_{V_0}, A^*A] = -B^*B + B^*\sum (\partial_{x_i} V_0)R_i B + E + F$, and hence 
	\begin{align*}
	h^{-1}\| [P_{V_0},A^*A] u\|_{L^2} &\leq \|B^*Bu\|_{L^2} + \sum \|B^*(\partial_{z_i} V_0) Bu\|_{L^2} + \|Eu\|_{L^2} + \| Fu\|_{L^2} \\ &\leq  C(\|Bu\|_{L^2} + \| Q_1 u \|_{L^2} + h^\gamma \| Gu\|_{L^2}).
	\end{align*}
	The final step is to replace $[V_1,A^*A]u$ by $[V_1,A^*A]Tu$ modulo a $\mathcal{O}(h^\infty)\|u\|_{L^2}$ error as in \eqref{eq:insertTcommutator}, and then apply Lemmas \ref{lem:Linftyonefactor}, \ref{lem:energyestimate}:
	\begin{align*}
	h^{-1} \int_{-3\delta}^{3\delta} \| [V_1,A^*A] T u(s)\|_{L^2(\RR^{n-1})} \, ds&\leq C h^\alpha \|Tu \|_{L^\infty((-3\delta,3\delta);L^2(\RR^{n-1}))} \\& \leq Ch^\alpha \big( h^{-1}\|Pu\|_{L^2} + \| T_1u \|_{H^1_h}\big).
	\end{align*}
	A final application of \eqref{eq:AintermsofB} finishes the proof with $T_1$ instead of $T$; this is of course not a restriction, since $\WF(T)$ can be shrunk at will.
\end{proof}

Altogether, we have established the following:

\begin{lem} \label{lem:hyperbolicwithtangential}
	There exists $\beta, \delta_0,\gamma >0$ and such that the following holds for each $\delta \in (0,\delta_0)$. Let $G \in \Psi^\COMP_h$ be elliptic on $\WF(B)$ and $Q_1$ be elliptic on $\WF(E)$. With $T\in \CI(\RR_{x}; \Psi^\COMP_h(\RR^{n-1}_y))$ as above,
	\[
	\|Bu\|_{L^2} \leq Ch^{-1}\|Pu\|_{L^2} + Ch^{\gamma}\|Gu\|_{L^2} + C\|Q_1u\|_{L^2} + Ch^{\alpha}\|T u\|_{H^1_h}
	\]
	for every $u \in H^1_h(X)$ with $Pu \in L^2(X)$.
\end{lem}

We now make a further argument to eliminate the $Gu$ term on the right hand side
of our estimates.  The semiclassical regularity is improved
inductively by $h^\gamma$ at each step. Each time, we reduce
$\delta > 0$ by an arbitrarily small amount; notice that the
decomposition $V = V_0 + V_1$ changes with every step as well by
shrinking $\tau > 0$ in \eqref{eq:V0potential}.

This nearly proves Proposition \ref{prop:truehyperbolic}, except that
we have a term $\| Tu \|_{H^1_h}$ on the right hand side involving a
tangential operator; this is easily remedied by an application of
\ref{lem:btangentialwavefront}, which allows us to estimate $\| Tu
\|_{H^1_h}$ by $\| Q_\B u \|_{H^1_h}$ modulo acceptable terms.

Finally, we will prove Theorem \ref{theo:propagation} using Proposition \ref{prop:truepropagation}. 

\begin{proof} [Proof of Theorem \ref{theo:propagation}]

Let $u$ be $h$-tempered in $H^1_h(X)$ with $Pu \in L^2(X)$, and assume that 
 \[
 \WF^{s+1}(Pu) = \emptyset.
 \]
In the notation of Theorem \ref{theo:propagation}, let $\varpi_\pm = (0,y_0,\pm \xi_0,\eta_0)$, where without loss we assume $\xi_0 > 0$. Note that both $\gamma_\pm((-\varepsilon,0))$ are disjoint from $T^*_Y X$ for sufficiently small. To prove the theorem, assume that there is a sequence of points $\varepsilon_n > 0$ tending to zero such that $\gamma_-(-\varepsilon_n) \notin \WF^r(u)$. 
We must then show that  $\gamma_+([-\varepsilon_0,0])$ is contained in $\WF^s(u)$ for some $\varepsilon_0 >0$.

First let $s \in [r,r+\alpha]$. We can assume that $\pi(\varpi_\pm) \notin \WFb^{1,r}(u)$, since otherwise by Theorem \ref{theo:GBBpropagation},
\[
\gamma_+((-\varepsilon,0)) \subset \WF^r(u) \subset \WF^s(u),
\]
for some $\varepsilon>0$, thus completing the proof. By Proposition \ref{prop:truepropagation}, there is a sequence
\[
\varpi_j \in \WF^{s}(u) \cap \{x < 0\}
\]
tending to $\varpi_+$. By Lemma \ref{lem:caratheodory}, if
$j$ is sufficiently large, then there exists
$\varepsilon_0>0$ such that the backwards bicharacteristics $\gamma_j$
from $\varpi_j$ exists for $t \in [-\varepsilon_0,0]$. Moreover, again
by Lemma \ref{lem:caratheodory}, $\gamma_j \rightarrow \gamma_+$
uniformly on $[-\varepsilon_0,0]$. By H\"ormander's theorem on
propagation of singularities, $\gamma_j([-\varepsilon_0,0])$ is
contained within $\WF^s(u)$. Since $\WF^s(u)$ is closed, letting
$j\rightarrow \infty$ shows that
$\gamma_+([-\varepsilon_0,0]) \subset \WF^s(u)$ as well.

If $s <r$, then apply the same argument but with $r' =s$ instead of $r$.
\end{proof}

\subsection{Improvement at glancing points}

We begin proving Theorem \ref{theo:improvedglancing} by establishing a local result similar to \cite[Proposition 7.4]{de2014diffraction}. The difference is that the threshold condition is $s \leq r+ \alpha -1$ rather than $s \leq r+ (\alpha -1)/2$, and crucially we are able to microlocalize the background regularity more finely.

 Given a normal coordinate patch $\mathcal{U}$, let $\mathsf{B}(\varpi,\varepsilon)$ denote the Euclidean ball about $\varpi \in T^*_\mathcal{U} X$ of radius $\varepsilon > 0$ induced by local coordinates $(x,y,\xi,\eta)$. Also choose $\alpha_0 < \alpha $ and set
 \[
 \theta = \min(1, \alpha_0-1) \in (0,1].
 \]
 Thus $\theta$ is a H\"older exponent for $\hamvf_p$.
The following proposition applies equally well at glancing and hyperbolic points (but of course at hyperbolic points the threshold is weaker than the one established in Section \ref{subsect:truepositivecommutator}). 

\begin{prop} \label{prop:glancingoutput}
	Let $\alpha > 1$ and $s\leq r+\alpha-1$, where  $s,r\in \RR$. Suppose that $u$ is $h$-tempered in $H^1_h(X)$ and $Pu \in L^2(X)$ and $\WF^{s+1}(Pu) = \emptyset$. Let
	 \[
	K \subset \Sigma \cap T^*_{Y \cap\, \mathcal{U}} X
	\]
	be compact. There exist $C_0,C_1,\delta_0 > 0$ such that for
        each $\varpi_0 \in K$ and $\delta \in (0,\delta_0)$,
        if
	\begin{gather} \label{eq:xibound}
	\mathsf{B}(\exp(-\delta \hamvf_p)(\varpi_0),C_0\delta^{1+\theta}) \cap \WF^s(u) = \emptyset, \quad \WF^{s+1}(Pu) = \emptyset, \notag \\
	\{|x| + |y-y_0| +|\sigma| + |\eta-\eta_0| < C_1\delta \} \cap \WFb^{1,r}(u) = \emptyset,  
	\end{gather}
	then $\varpi_0 \notin \WF^s(u)$.
\end{prop}

\begin{proof} 
According to Remark \ref{rem:smoothu}, we can assume that $u \in \CI(X)$. It will suffice to consider the case $K = \{\varpi_0\}$ (cf.\ Remark \ref{rem:uniformomega} and the discussion in \cite[Section 7]{de2014diffraction}). We may also assume that $dp(\varpi_0) \neq 0$, otherwise there is nothing to prove. 

Choose local coordinates $(\rho_0,\ldots,\rho_{2n-1})$ vanishing at $\varpi_0$ such that 
\[
(\hamvf_{p}\rho_0)(\varpi_0) > 0, \quad (\hamvf_{p}\rho_i)(\varpi_0) = 0 \text{ for } i = 1,\ldots, 2n-1. 
\]
We use the same decomposition of $V = W_0 + W_1 + V_1$ as in Section \ref{subsect:potentialdecomp}. As usual, set 
\[
\omega = \sum_{i=1}^{2n-1} \rho_i^2, \quad \phi = \rho_0 + \frac{1}{\beta^2\delta}\omega.
\]
Also, fix a cutoff $\psi \in \CI(T^*X;[0,1])$ such that $\psi =1 $ near $\{|\rho_0| \leq 2\delta, \, \omega^{1/2} \leq 2\beta\delta\}$ with support in $\{|\rho_0| <3\delta ,\, \omega^{1/2} < 3\beta\delta\}$.

Fix a neighborhood $U$ of $\varpi_0$ on which $\hamvf_{p} \rho_0 > 4c_0$
for some $c_0 > 0$. On the other hand, using the H\"older regularity
of $\hamvf_p \in \C^{0,\theta}$,
\[
|\hamvf_p \omega| \leq M\omega^{1/2}(\omega^{\theta/2} + |\rho_0|^{\theta})
\] 
on $U$. Therefore $\hamvf_p \phi \geq 4c_0 -  3M\beta^{-1}((3\beta\delta)^\theta + (3\delta)^{\theta})$ on $U$.  If we choose $\beta = c\delta^{\theta}$, with $c>0$ sufficiently large, then we can arrange that 
\[
\hamvf_p \phi > 3c_0
\]
on $\supp \psi$. Given $\delta>0$ (and setting $\beta = c\delta^\theta$ as above) we can choose $\tilde \tau>0$ depending on $\delta$ such that
\[
\hamvf_{p_{W_0}} \phi > 2c_0 \text{ on } \supp \psi.
\]
Further shrinking $\tilde \tau$ if necessary (again depending on $\delta$) and using \eqref{eq:W1decays}, we can also arrange that
\[
|\partial_{z_i} W_1 \cdot \partial_{\zeta_i} \phi| \leq c_0/(2n) \leq \hamvf_{p_{W_0}} \phi/(4n)
\]
on $\supp\psi$.

Let $A= \Op_h(a)$, where  $a = \chi_0(2-\phi/\delta)\chi_1(1+(\rho_0+\delta)/(\beta\delta))$. Write
\[
-(2/h)\Im \left<A(P - V_1)u ,Au\right>  = (i/h)\left<[A^*A,P_{W_0} + W_1]u,u\right>.
\]
Now the term $(i/h)[P_{W_0},A^*A] \in \Psi^\COMP_h$ has principal symbol $\hamvf_{p_{W_0}}a^2$. This we write as
\[
\hamvf_{p_{W_0}}a^2 = -b^2 + e,
\]
where as usual, $b = (2\delta)^{1-2}(\hamvf_{p_{W_0}}\phi)^{1/2}(\chi_0\chi_0')^{1/2}\chi_1$. On the other hand, $\supp e$ is contained in the set
\[
\{-\delta-\delta \beta \leq \rho_0 \leq -\delta, \, \omega^{1/2} \leq 2\beta \delta\};
\]
note that with the choice $\beta=c \delta^{\theta}$ this is contained
in $\mathsf{B}(\exp(-\delta \hamvf_p)(\varpi_0),C_0\delta^{1+\theta})$ for
all $\delta$ sufficiently small.

Next, consider the term $(i/h)[W_1,A^*A]$. First, if $\delta > 0$ is given we can arrange the decomposition $V = W_0 + W_1 + V_1$ so that 
\[
\WF(K_{[W_1,A^*A]}) \subset O \times O',
\]
where $O$ is an arbitrary neighborhood of $\WF(A)$.

Exactly as in Section \ref{subsect:truepositivecommutator}, fix compactly supported operators $R_1,\ldots,R_n \in \Psi^{\COMP}_h$ with principal symbols 
\[
r_i = (\hamvf_{p_{W_0}}\phi)^{-1}(\partial_{\zeta_i} \omega) \psi.
\]
In particular, $\psi \hamvf_{W_1}\omega = (\hamvf_f \phi)  \sum  \partial_{z_i}W_1 \cdot r_i$, and  $\sum |  \partial_{z_i}W_1 \cdot r_i| \leq 1/4$. Moreover,
\[
\hamvf_{W_1}(a^2) = -2\delta^{-1}(\hamvf_{W_1} \phi)(\chi_0 \chi_0')\chi_1^2 = -b^2\left(\dfrac{\hamvf_{W_1} \phi}{\hamvf_{p_{W_0}} \phi}\right) = -b^2 \sum
\partial_{z_i} W_1 \cdot r_i,
\]
since $\psi=1$ on $\supp b$. On the other hand, as compared to Section \ref{subsect:truepositivecommutator} there is an additional contribution to the commutator: fix compactly supported operators $L_1,\ldots, L_n \in \Psi^\COMP_h$ with principal symbols
\[
\ell_i = (\beta\delta)^{-1} (\partial_{\zeta_i} \rho_0 )\chi_0\chi_1'.
\]  
We can take 
\[
\WF(L_i) \subset \{-\delta-\delta \beta \leq\rho_0 \leq -\delta, \, \omega^{1/2} \leq 2\beta \delta\}.
\]
as well. By further refining the choice of $V = W_0 + W_1 + V_1$, we can  arrange that the kernels of $B^* (\partial_{z_i} W_1) R_iB$ and $(\partial_{z_i} W_1) L_i$ also have wavefront set contained in $O \times O'$. By Proposition \ref{prop:singularcomposition},
\begin{multline*}
(i/h)[P_{W_0} + W_1, A^*A] + B^*B + B^*\sum (\partial_{z_i} W_1)R_i B + \sum (\partial_{z_i} W_1)L_i +  E \\ \in I^{-\alpha+(1/2)+\varepsilon_0,\COMP}_h(X, N^*((X\times Y) \cap \diag), N^*\diag)
\end{multline*}
for any $\varepsilon_0 > 0$, noting the additional terms involving $L_i$ as compared to the corresponding expression in Section \ref{subsect:truepositivecommutator}. If this operator is denoted by $F$, then by construction the principal symbol of $F$ along $N^*\diag$ vanishes, and hence
\[
F\in h I^{-\alpha+(3/2)+\varepsilon_0,\COMP}_h(X, N^*((X\times Y) \cap \diag), N^*\diag).
\]
Since all of the operators above have kernels with wavefront set in $O\times O'$, so does $F$. Now we consider the identity
\begin{multline} \label{eq:PW_0commutator}
\left<(i/h)[P_{W_0}+ W_1, A^*A]u,u\right> = \| Bu\|^2_{L^2} \\ + \sum \left<(\partial_{z_i} W_1)R_iBu,Bu\right> + \sum \left<(\partial_{z_i} W_1)L_iu,u\right>+ \left<Eu,u\right> + \left<Fu,u\right>.
\end{multline}
The second, third, and fourth terms on the right hand side of \eqref{eq:V_0commutator} are bounded in absolute value as in Section \ref{subsect:truepositivecommutator}: for the second term, we use the bound 
\[
\|R_i u \|_{L^2} \leq 2\sup |\sigma_h(R_i)|\|u\|_{L^2} + \mathcal{O}(h^\infty)\|u\|_{L^2},
\]
and the fact that $2\sum \sup |\partial_{z_i} W_1||r_i| \leq 1/2$ by construction. Therefore
\[
\sum \lvert\left<(\partial_{z_i} W_1)R_iBu,Bu\right>\rvert \leq (1/2)\|Bu\|^2_{L^2} + \mathcal{O}(h^\infty)\|u \|_{L^2}
\]
To bound the third and fourth terms, choose $Q_1 \in \Psi^\COMP_h$ such that $\WF(E) \subset \ELL(Q_1)$ and estimate
\[
\sum \lvert \left<(\partial_{z_i} W_1)L_iu,u\right>\rvert + \lvert\left<Eu,u\right>\rvert \leq C\| Q_1 u\|^2_{L^2} + \mathcal{O}(h^\infty)\| u \|^2_{L^2}.
\]
For the fifth term, by Lemma \ref{lem:singularsymbolL2bounded},
\[
\| F \|_{L^2 \rightarrow L^2} \leq Ch^{2\gamma}
\]
with the same exponent $\gamma$ as in \eqref{eq:gammabound}.
Let $G \in \Psi^\COMP_h(X)$ be elliptic on $\WF(B)$; since $O$ was an arbitrary neighborhood of $\WF(B)$, we can assume that $O \subset \ELL(G)$ as well. Thus we can bound
\[
\lvert\left< Fu,u\right>\rvert \leq Ch^{2\gamma} \| Gu\|^2 + \mathcal{O}(h^\infty)\|u\|^2_{L^2}.
\]
We therefore conclude that
\begin{multline*}
\| Bu \|^2_{L^2}  \leq Ch^{-1} \|APu\|_{L^2} \| Au \|_{L^2} + Ch^{-1}|\left<AV_1u,Au\right>| \\ 
+ Ch^{2\gamma}\| Gu \|^2_{L^2} + \| Q_1u\|^2_{L^2} + \mathcal{O}(h^\infty)\|u\|^2_{L^2}.
\end{multline*}
Here $G$ is elliptic on $\WF(B)$, and 
\[
\WF(Q_1) \subset \WF(G) \cap  \{-2\delta-2\delta \beta \leq \rho_0 \leq -\delta/2, \, \omega^{1/2} \leq 3\beta \delta\}.
\]
Note that the various terms involving $\|Au\|_{L^2}$ can be bounded in terms of $\|Bu\|_{L^2}$,
\[
\|Au\|_{L^2} \leq C\|Bu\|_{L^2} + Ch\| Gu \|_{L^2} + \mathcal{O}(h^\infty)\|u\|_{L^2}.
\]
It remains to bound the term $h^{-1}|\left<AV_1u,Au\right>|$. As compared to Section \ref{subsect:truepositivecommutator}, we are no longer able to use the energy estimates, which leads to a loss of $h^{-1}$ in the threshold condition.

Just as in \eqref{eq:insertT},
 if $C_1>0$ is sufficiently large we can choose a tangential psedudodifferential operator $T$ with 
\[
\WF(T) \subset \{|x| + |y-y_0| + |\eta-\eta_0| < C_1\delta\}
\]
such that
\[
h^{-1}\lvert\left<AV_1u,Au\right>\rvert \leq h^{-1} \|
V_1 A Tu\|  + \mathcal{O}(h^\infty)\|u\|_{L^2}^2.
  \]
Then \eqref{LinfinityV1} yields
\[
h^{-1}\lvert\left<AV_1u,Au\right>\rvert\leq \varepsilon\| Au \|_{L^2}^2 + C_\varepsilon h^{2\alpha-2}\| T u \|^2_{L^2} + \mathcal{O}(h^\infty)\|u\|_{L^2}^2.
\]
On the other hand, by \eqref{quantitativetangential}, we can choose $Q_\B \in \Psibh^\COMP$ so that  
\[
\| Tu \|_{L^2}^2 \leq C\|Pu\|^2_{L^2} + \|Q_\B u \|^2_{L^2},
\]
where $\WF(Q_\B) \subset \{ |x| + |y-y_0| +|\sigma|+  |\eta-\eta_0| \} < C_1\delta \}$, increasing $C_1$ if necessary. An inductive argument completes the proof (the commutant must be modified slightly at each step, as pointed out at the end of Section \ref{subsect:glancing}).
\end{proof}

\begin{proof}[Proof of Theorem \ref{theo:improvedglancing}]
Let $\varpi_0 \in \pi^{-1}(\gl) \cap T^*_YX$, and suppose that no bicharacteristic segment of the form $\gamma(-\varepsilon,0)$, where $\gamma(0) = \varpi_0$, is contained in $\WF^r(u)$ for any $\varepsilon > 0$; we wish to show that $\varpi_0 \notin \WF^r(u)$.
Let $s$ be such that $\varpi_0 \notin \WF^{s}(u)$; this always exists by our tempered assumption. According to Lemma \ref{lem:WFbimpliesWF}, this also implies that $q_0 = \pi(\varpi_0) \notin \WFb^{1,s}(u)$. We now show that 
\[
\varpi_0 \notin \WF^{s_0}(u) \text{ for } s_0 = \min(r,s+\alpha-1).
\]
Observe that $s + \alpha - 1 > s$ since $\alpha > 1$.
Since $\varpi_0 \notin \WF^s(u)$, let $U$ be a neighborhood of $\varpi_0$ of the form $U = \mathsf{B}(\varpi_0,\varepsilon_0)$, where $\varepsilon_0 > 0$ is chosen so that $U \cap \WF^s(u) = \emptyset$. By further shrinking $\varepsilon_0$, we can also assume that $U_\B \cap \WFb^{1,s}(u) = \emptyset$, where
\[
U_\B = \{ |x| + |y-y_0| + |\sigma| + |\eta-\eta_0| < \varepsilon_0 \}. 
\]
By Lemma \ref{lem:btangentialwavefront} and Remark \ref{rem:microlocalizeellipticerror}, we can conclude that 
\[
\WF^{s_0}(u) \cap U \subset \chare.
\]
We now argue as in \cite[Lemma 8.1]{de2014diffraction}: using Proposition \ref{prop:glancingoutput} and ordinary semiclassical propagation of singularities away from $Y$, we can
therefore construct a backward bicharacteristic segment through
$\varpi_0$ contained in $\WF^{s_0}(u)$; the proof is an even simpler analogue of Lemma \ref{lem:cauchypeano}. This yields a contradiction, and thus we may reach the desired regularity $s=r$ by iteration.
\end{proof}

\appendix

\section{Proof of Proposition \ref{prop:1d}} \label{appendix:1d}

\subsection{Plane wave solutions}
We construct exact solutions of $(P-1)u = 0$ on $[0,x_0)$ of the form
\[
u_\pm(x) = e^{\pm ix/h}(1+b_\pm(x)),
\]
subject to the conditions $b_\pm(0) = 0$ and $b'_\pm(0) = 0$. We then obtain $\C^2$ solutions to $(P-1)u = 0$ on $(-\infty,x_0)$ after extending $b_\pm$ by zero to $(-\infty,0)$. Thus $u_\pm$ are precisely the continuations of the plane wave solutions $e^{\pm ix/h}$ from $(-\infty,0)$ to $(-\infty,x_0)$.

 Although the functions $b_\pm$ are globally defined on $[0,x_0)$, their region of asymptotic validity is small (in an $h$-dependent way). First consider the case $b= b_+$, so that $b_+$ satisfies the equation
\begin{equation} \label{eq:perturbed}
h^2 b''(x) + 2i hb'(x) = (1+b(x)) V(x).
\end{equation}
Viewing the right hand side as a correction, the unperturbed equation has linearly independent solutions $1$ and $e^{-2ix/h}$. By variation of parameters, \eqref{eq:perturbed} is equivalent to the integral equation $b = J b$, where
\[
(J b)(x) = \frac{1}{2ih} \int_{0}^x \left(1-e^{2i(s-x)/h} \right)V(s)(1+b(s))\, ds.
\]
This equation can be solved by successive approximation. Thus we set $b_0 = 0$, inductively define $b_{n+1} = Jb_{n}$. Let
\[
\sigma(x) =  \frac{1}{h} \int_0^x |V(s)| \, ds = \frac{x^{\alpha+1}}{(\alpha+1)h}
\]
on $[0,x_0)$. A simple inductive argument shows that
\[
|b_{n+1}(x)- b_{n}(x)| \leq \frac{\sigma(x)^{n+1}}{(n+1)!}, \quad |b'_{n+1}(x) - b'_{n}(x)| \leq \frac{2 \sigma(x)^{n+1}}{h (n+1)!}
\]
for $n\geq 0$. Differentiating once more and using the formula for $J$, it follows that 
\[
b = \sum_{n=0}^\infty (b_{n+1} - b_{n})
\] 
is a $\C^2([0,\infty))$ function solving \eqref{eq:perturbed} with $b(0) = 0$ and $b'(0)= 0$. Moreover, $b = b_1 + \varepsilon$, where $b_1 = J(0)$ and the remainder satisfies
\[
|\varepsilon(x)| \leq e^{\sigma(x)} - 1 - \sigma(x), \quad  |h\varepsilon'(x)/2| \leq e^{\sigma(x)} - 1 - \sigma(x) 
\]
 on $[0,\infty)$.
We now find the behavior of $b_1(x)$ as $x/h \rightarrow \infty$.  We will frequently use the rescaled variable $y = x/h$, and by a slight abuse of notation write $b_1 = b_1(y)$ when convenient.

\begin{lem} \label{lem:b1asymptotics}
	In terms of $y = x/h$, the function $b_1$ satisfies 
	\[
	h^{-\alpha} b_1(y) =  -2^{-\alpha-2} e^{i\alpha \pi/2}\,\Gamma(\alpha+1) \, e^{-2iy}  + \frac{y^{\alpha+1}}{2i(\alpha+1)} +\frac{y^\alpha}{4} + \mathcal{O}(y^{\alpha-1}) 
	\]
	 as $y\rightarrow \infty$, where the right hand side does not depend on $h$.
\end{lem}
\begin{proof}
Integrating by parts once,
\begin{align} \label{eq:b1asymptotics}
h^{-\alpha} b_1(y) &=  \frac{ e^{-2iy}}{\alpha+1} \int_0^y e^{2is} s^{\alpha+1} \, ds \notag \\& = 2^{-\alpha-2} e^{i(\alpha+2)\pi/2}\left(\Gamma(\alpha+2) - \Gamma(\alpha+2,-2iy) \right),
\end{align}
\blue{where $\Gamma(a,z)$ is the incomplete gamma function, defined as
$$
\Gamma(a,z) \equiv \int_z^\infty t^{a-1} e^{-t} \, dt,
$$ with the integral taken along any path not crossing the negative
real axis.}
Since $y$ is real, there is an asymptotic expansion
\[
\Gamma(\alpha+2,-2iy) \sim (-2iy)^{\alpha+1} e^{2iy} \sum_{k=0}^\infty a_k (-2iy)^{-k} 
\]
as $y\rightarrow \infty$, where $a_0 = 1$ and $a_k = (\alpha+2 -1)\cdots (\alpha+2-k)$ for $k>0$ (see \cite[Chapter 3, \S 1.1]{olver2014asymptotics}). Truncating after two terms,
\[
\Gamma(\alpha+2,-2iy) = e^{2iy}  \left(  2^{\alpha+1}e^{-i(\alpha+1)\pi/2} y^{\alpha+1} + (\alpha+1)2^{\alpha}e^{-i\alpha\pi/2}y^\alpha + \mathcal{O}(y^{\alpha-1}) \right).
\]
Plugging this into \eqref{eq:b1asymptotics} finishes the proof.
\end{proof}

For future use, define the quantity
\[
\gamma_\pm(\alpha) = -2^{-\alpha-2}e^{\pm i \alpha \pi /2}\Gamma(\alpha +1).
\]
Since $V$ is real, we can define the complementary solution $u_-$ simply by $u_- = \bar{u}_+$, so that $u_- = e^{-iy}(1 + \bar b_1 + \bar \varepsilon)$.

\subsection{WKB solutions}
We would like to connect the solutions $u_\pm$ with a WKB-type
solution which is valid for $x \in (0,\infty)$. \blue{To do this we will
require precise remainder estimates that will permit matching
solutions at an $h$-dependent family of points $x_0\to 0$ that satisfies
$x_0/h \to \infty$ so Lemma~\ref{lem:b1asymptotics} will apply.}
Let $f = 1-V$, so $P-1 = (hD_x)^2 - f$.  Define the phase
\[
\phi(x) = \int_0^x f^{1/2}(s)\, ds.
\] 
According to \cite[Chapter 6, Theorem 2.2]{olver2014asymptotics}, there exists an exact solution to $Pu=0$ on $(0,1)$ of the form
\[
v_+(x) = f(x)^{-1/4}e^{i\phi(x)/h}(1+\delta(x)).
\]
The remainder satisfies 
\[
|\delta(x)| \leq e^{h\tau(x)}-1, \quad f(x)^{-1/2}|h\delta'(x)| \leq e^{h\tau(x)}-1,
\]
where
\[
\tau(x) = \int_{x}^\infty \big|f(s)^{-1/4} \partial_s^2 \left( f(s)^{-1/4} \right) \big| \, ds.
\]
In particular, $v_+(x)= f(x)^{-1/4}e^{i\phi(x)/h} + \mathcal{O}(h)$
uniformly on any compact subset of $(0,\infty)$. Observe that $f=1$
and $\delta$ vanishes outside the support of $V,$ \blue{since
  $\tau(x)$ vanishes}. Thus $v_+ = c_0 e^{ix/h}$ for $x \gg 0$, where
\[
c_0 = \int_{0}^{x_1} f^{1/2}(s) \, ds - x_1
\]
for any fixed point $x_1 \gg 0$  outside the support of $V$. 

There exist constants $A,B$ such that $v_+ = Au_+ + Bu_-$. Setting $u= v_+ = Au_+ + Bu_-$, the solution $u$ satisfies
\[
u = \begin{cases} Ae^{ix/h} + Be^{-ix/h} & x < 0, \\
c_0 e^{ix/h}, &x \gg 0. \end{cases}
\]
Therefore $R = B/A$ and $T = c_0/A$, where $R,T$ are as in \eqref{eq:scattering}. The constants $A,B$ are found by computing the semiclassical Wronskians 
\[
\mathcal{W}_h(u_\pm,v_+)(x) = hu_\pm(x) \cdot h v'_+(x) - hu'_\pm(x) \cdot v_+(x)
\]
at an appropriate $h$-dependent point (the Wronskian is of course constant). Indeed, we have the identity
\begin{equation} \label{eq:wronskian}
v_+ = \frac{\mathcal W(u_+,v_+)}{\mathcal W(u_+,u_-)}u_- - \frac{\mathcal W(u_-,v_+)}{\mathcal W(u_+,u_-)}u_+.
\end{equation}

\subsection{Wronskians}

We continue to write $y = x/h$. Fix $\eta$ satisfying
\[
\frac{2+\alpha}{2(\alpha+1)} < \eta < 1.
\]
and set $x_0 = h^{\eta}$. Then $y_0 = x_0/h \rightarrow \infty$  whereas
\begin{equation} \label{eq:littleo}
h^\alpha y_0^{\alpha+1} = x_0^{\alpha+1}/h  = o(h^{\alpha/2}).
\end{equation}
Since $x_0^{\alpha+1}/h \rightarrow 0$, we see that
\[
e^{i \phi(x_0)/h} = e^{ix_0/h}e^{i (\phi(x_0)-x_0)/h} = e^{iy_0}\Big(1 + h^\alpha \frac{i y_0^{\alpha+1}}{2i(\alpha+1)} +\zeta(x_0) \Big),
\]
where $\zeta(x_0) = \mathcal{O}\left(x_0^{2\alpha+2}/h^2\right) + \mathcal{O}\left(x^{2\alpha+1}/h\right)$. We then check that
\[
x^{\alpha+2}/h = h^{(\alpha+2)\eta-1} < h^{((2+\alpha)^2/2(\alpha+1))-1} = h^{1+\alpha},
\]
so in particular, $\zeta(x_0) = o(h^\alpha)$ and $h\zeta'(x_0) = o(h^\alpha)$. Since $\alpha \in (0,1)$, it follows that  $f(x)^{-1/4} \sim 1 + x^\alpha/4$ as $x\rightarrow 0^+$, hence
\[
\tau(x) \sim \alpha x^{\alpha-1}/4 \text{ as } x\rightarrow 0^+.
\]
Therefore
\[
f(x_0)^{-1/4} = 1+ x_0^\alpha/4 +o(h^{\alpha}), \quad h(f^{-1/4})'(x_0) = o(h^\alpha).
\] 
Finally, the errors in $u_\pm(x_0)$ and $v_+(x_0)$ are bounded by
\[
|\varepsilon(x_0)| + |h\varepsilon(x_0)| = o\left(h^\alpha\right), \quad |\delta(x_0)| +  |h\delta'(x_0)| = o(h^\alpha).
\]
From this we conclude that
\begin{align*}
v_+(x_0) &= e^{iy_0}\Big(1+ h^{\alpha} \Big(\frac{y_0^{\alpha+1}}{2i(\alpha+1)} + \frac{y_0^\alpha}{4} \Big) + o(h^\alpha)\Big),\\
v_+'(x_0) &= ie^{iy_0}\Big(1+ h^{\alpha} \Big(\frac{y_0^{\alpha+1}}{2i(\alpha+1)} - \frac{y_0^\alpha}{4} \Big) + o(h^\alpha)\Big).
\end{align*}
Similarly,
\begin{align*}
u_\pm(x_0) &= e^{\pm iy_0}\Big( 1 \pm h^\alpha\Big( \frac{y_0^{\alpha+1}}{2i(\alpha+1)} \pm \frac{y_0^\alpha}{4}\Big) + o(h^\alpha)\Big) + h^\alpha e^{\mp iy_0}\gamma_\pm(\alpha), \\
hu_\pm'(x_0) &= \pm ie^{\pm iy_0}\Big(1\pm h^{\alpha} \Big(\frac{y_0^{\alpha+1}}{2i(\alpha+1)} \mp \frac{y_0^\alpha}{4} \Big) + o(h^\alpha)\Big) \mp i h^\alpha e^{\mp iy_0} \gamma_\pm(\alpha).
\end{align*}
Calculating the Wronskians by evaluating at $x_0$,
\[
\mathcal W(u_+,v_+)= 2ih^\alpha\gamma_+(\alpha) + o(h^\alpha), \quad \mathcal W(u_-,v_+)  = 2i + o(h^\alpha).
\]
We also have $\mathcal W(u_+,u_-) = -2i$ by evaluating the Wronskian at $x=0$. Using \eqref{eq:wronskian}, we see that $v_+ = Au_+ + Bu_-$ with $A = 1+o(1)$ and $B = - h^\alpha\gamma_+ +o(h^\alpha)$. Dividing through by $A$ also shows that $u_+ + Ru_- = Tv_+$, where the reflection and transmission coefficients satisfy
\[
R = 2^{-\alpha-2}e^{i\alpha \pi/2}\Gamma(\alpha+1)h^{\alpha} + o(h^\alpha), \quad T = c_0 + o(1),
\]
thereby completing the proof.

\bibliographystyle{alphanum} 
	
	\bibliography{central_bibliography}

\end{document}